\documentclass[12pt,reqno]{amsart}

\usepackage{geometry} 
\usepackage{graphicx}
\usepackage[latin1]{inputenc}
\usepackage{amsmath,amssymb,wasysym}
\usepackage{amsthm,enumerate, mathdots}
\usepackage{latexsym}
\usepackage{xypic}
\usepackage[all]{xy}
\usepackage{pxfonts}
\usepackage{multirow}
\usepackage{color}
\usepackage{hyperref}
\usepackage{bbold}
\usepackage{aurical}
\usepackage[T1]{fontenc}

\theoremstyle{plain}
\newtheorem{thm}{Theorem}[section]
\newtheorem{prop}[thm]{Proposition}
\newtheorem{lmm}[thm]{Lemma}
\newtheorem{cor}[thm]{Corollary}

\theoremstyle{definition}
\newtheorem{rmk}[thm]{Remark}{\rm}
\newtheorem{defi}[thm]{Definition}
\newtheorem{expl}[thm]{Example}{\rm}
{\rm}
\newtheorem{hyp}[thm]{Hypothesis}{\rm}
{\rm}

\newcommand{\spa}{\mbox{ }}

\newcommand{\defin}{\stackrel{def} =}

\newcommand{\3}[1]{\overline{#1}}
\newcommand{\4}[1]{\widetilde{#1}}

\newcommand{\7}[1]{\boldsymbol{#1}}
\newcommand{\9}[1]{{}^{#1}}

\newcommand{\F}{\mathbb{F}}
\newcommand{\Z}{\mathbb{Z}}

\newcommand{\N}{\mathbb{N}}

\newcommand{\cat}{\mathcal{C}}
\newcommand{\ee}{\mathcal{E}}

\newcommand{\FF}{\mathcal{F}}
\newcommand{\g}{\mathcal{G}}
\newcommand{\hh}{\mathcal{H}}

\newcommand{\kk}{\mathcal{K}}
\newcommand{\LL}{\mathcal{L}}
\newcommand{\mm}{\mathcal{M}}
\newcommand{\nn}{\mathcal{N}}

\newcommand{\pp}{\mathcal{P}}

\newcommand{\TT}{\mathcal{T}}

\newcommand{\incl}{\mathrm{incl}}
\newcommand{\Hom}{\mathrm{Hom}}
\newcommand{\Mor}{\mathrm{Mor}}

\newcommand{\Aut}{\mathrm{Aut}}
\newcommand{\Inn}{\mathrm{Inn}}
\newcommand{\Out}{\mathrm{Out}}
\newcommand{\Iso}{\mathrm{Iso}}

\newcommand{\Ob}{\mathrm{Ob}}
\newcommand{\Ker}{\mathrm{Ker}}
\newcommand{\Id}{\mathrm{Id}}
\newcommand{\Syl}{\operatorname{Syl}\nolimits}

\newcommand{\typ}{\mathrm{typ}}

\newcommand{\ploc}{(S, \FF, \LL)}

\newcommand{\ww}{\mathbb{W}} 
\newcommand{\dd}{\mathbb{D}} 

\newcommand{\locl}{\text{\bbfamily L}}
\newcommand{\locm}{\text{\bbfamily M}}

\newcommand{\loct}{\text{\bbfamily T}}

\newcommand{\loc}{(\LL, \Delta, S)}

\newcommand{\Sset}{\mathrm{\mathbf{sSets}}}

\newcommand{\topo}{\mathrm{\mathbf{Top}}}

\newcommand{\Setsc}{\mathrm{\mathbf{Sets}}}

\newcommand{\Sing}{\mathrm{Sing}}

\usepackage{amssymb}

\DeclareMathAlphabet\EuR{U}{eur}{m}{n}
\SetMathAlphabet\EuR{bold}{U}{eur}{b}{n}

\newcommand{\Inj}{\operatorname{Inj}\nolimits}

\newcommand{\diag}{\operatorname{diag}\nolimits}

\newcommand{\pr}{\operatorname{pr}\nolimits}

\newcommand{\colim}{\operatornamewithlimits{colim}}

\newcommand{\cala}{\mathcal{A}}

\newcommand{\calc}{\mathcal{C}}
\newcommand{\cald}{\mathcal{D}}

\newcommand{\calt}{\mathcal{T}}

\newcommand{\autcat}{\mathcal{A}ut}   

\newcommand{\proj}{\operatorname{proj}\nolimits}

\newcommand{\op}{^{\textup{op}}}

\newcommand{\fus}{_{\textup{fus}}}

\let\oldcirc=\circ
\renewcommand{\circ}{\mathchoice
    {\mathbin{\scriptstyle\oldcirc}}{\mathbin{\scriptstyle\oldcirc}}
    {\mathbin{\scriptscriptstyle\oldcirc}}
    {\mathbin{\scriptscriptstyle\oldcirc}}}

\newcommand{\hclim}[1]{\setbox1=\hbox{\rm hocolim}
    \setbox2=\hbox to \wd1{\rightarrowfill} \ht2=0pt \dp2=-1pt
    \mathop{\vtop{\baselineskip=5pt\box1\box2}}
    _{#1}}

\newcommand{\higherlim}[2]{\displaystyle\setbox1=\hbox{\rm lim}
    \setbox2=\hbox to \wd1{\leftarrowfill} \ht2=0pt \dp2=-1pt
    \setbox3=\hbox{$\scriptstyle{#1}$}
    \ifdim\wd1<\wd3
    \mathop{\hphantom{^{#2}}\vtop{\baselineskip=5pt\box1\box2}^{#2}}_{#1}
    \else
    \mathop{\vtop{\baselineskip=5pt\box1\box2}}\limits_{#1}\nolimits^{#2}
    \fi}

\newcommand{\aut}{\operatorname{aut}\nolimits}

\renewcommand{\Im}{\operatorname{Im}\nolimits}

\newcommand{\longleft}[1]{\;{\leftarrow%
\count255=0 \loop \mathrel{\mkern-6mu}%
    \relbar\advance\count255 by1\ifnum\count255<#1\repeat}\;}
\newcommand{\longright}[1]{\;{\count255=0 \loop \relbar\mathrel{\mkern-6mu}%
    \advance\count255 by1\ifnum\count255<#1\repeat\rightarrow}\;}
\newcommand{\Right}[2]{\overset{#2}{\longright#1}}
\newcommand{\RIGHT}[3]{\mathrel{\mathop{\kern0pt\longright#1}
        \limits^{#2}_{#3}}}
\newcommand{\Left}[2]{{\buildrel #2 \over {\longleft#1}}}
\newcommand{\LEFT}[3]{\mathrel{\mathop{\kern0pt\longleft#1}\limits^{#2}_{#3}}
}
\newcommand{\dRIGHT}[3]{\mathrel{%
   \mathop{\vcenter{\baselineskip=0pt\hbox{$\kern0pt\longright#1$}%
   \hbox{$\kern0pt\longright#1$}}}\limits^{#2}_{#3}}}
\newcommand{\LRIGHT}[3]{\mathrel{%
   \mathop{\vcenter{\baselineskip=0pt\hbox{$\kern0pt\longleft#1$}%
   \hbox{$\kern0pt\longright#1$}}}\limits^{#2}_{#3}}}
\newcommand{\RLEFT}[3]{\mathrel{%
   \mathop{\vcenter{\baselineskip=0pt\hbox{$\kern0pt\longright#1$}%
   \hbox{$\kern0pt\longleft#1$}}}\limits^{#2}_{#3}}}
\newcommand{\onto}[1]{\;{\count255=0 \loop \relbar\joinrel
    \advance\count255 by1
    \ifnum\count255<#1 \repeat \twoheadrightarrow}\;}

\newtheoremstyle{slant}{}{}{\slshape}{}{\bfseries}{.}{.5em}{}%
\newtheoremstyle{special}{}{}{\slshape}{}{\bfseries}{.}{.5em}{\thmnote{#3}}

\newtheorem{Th}{Theorem}
\newtheorem{Coro}[Th]{Corollary}

\theoremstyle{remark}

\begin{document}

\begin{abstract}

A partial group is a generalization of the concept of group recently introduced by A. Chermak in \cite{Chermak}. By considering partial groups as simplicial sets, we propose an extension theory for partial groups using the concept of (simplicial) fibre bundle. This way, the classical extension theory for groups naturally extends to an extension theory of partial groups. In particular, we show that the category of partial groups is closed by extensions. We also describe the cohomological obstructions for existence and uniqueness of extensions, generalizing the usual obstructions for group extensions.

The second part of the paper considers extensions of (finite) localities, which are a particular type of partial group, mimicking the $p$-local structure of finite groups. The goal here is to give sufficient conditions for an extension of localities to produce a new locality.

\end{abstract}

\title{An extension theory for partial groups and localities}
\author{A. Gonz\'alez}
\address{Department of Mathematics\\
Kansas State University\\
66506 Manhattan KS\\
United States of America}
\email{agondem@math.ksu.edu}

\maketitle

\tableofcontents


The theory of $p$-local finite groups emerged from work by Puig in \cite{Puig1} and \cite{Puig2} from an algebraic point of view, and from work of Broto, Levi and Oliver (et al.) in \cite{BLO1} and \cite{BLO2} from a topological point of view. Roughly speaking, this theory provides an axiomatic characterization of the $p$-local structure of finite groups.

In these notes, we introduce an extension theory for $p$-local finite groups. Several partial results exist in the literature about extensions of $p$-local finite groups, for example \cite{OV} or \cite{BCGLO2}, each treating a particular situation. One of the advantages of our approach is that it provides a unifying extension theory. A second major advantage is that our theory also includes (and generalizes) the existing extension theory for groups.

In order to develop our theory, we first need a slight change of point of view. The categorical language of $p$-local finite groups becomes rather heavy when dealing with extensions. Instead, we use the recently introduced \emph{partial groups} of \cite{Chermak}. Chermak introduced these objects, together with the so-called localities (a specialization of partial groups), in order to solve (positively) a famous conjecture on $p$-local finite groups.

As pointed out to the author by Broto, the definition of partial group carries intrinsically a simplicial structure. From this point of view, an \emph{extension of partial groups} is simply a fibre bundle of the simplicial objects.

\begin{Th}\label{ThmA}

The category of partial groups is closed by extensions.

\end{Th}

Furthermore, this extension theory allows a cohomological obstruction theory for existence and uniqueness of extensions, which essentially generalizes the known obstruction theory for existence and uniqueness of group extensions (see \cite[Chapter IV]{Brown} or \cite[Chapter IV]{MacLane}).

\begin{Coro}\label{CorB}

Let $\locm', \locm''$ be partial groups, and let $\varepsilon \colon \locm'' \to \Out(\locm')$ be an outer action.
\begin{enumerate}[(i)]

\item There is an obstruction class $[\kappa] \in H^3(\locm''; Z(\locm'))$ to the existence of extensions of $(\locm', \locm'', \varepsilon)$: such extensions exist if and only if $[\kappa] = 0$.

\item If there is any, the set of isomorphism classes of extensions of $(\locm', \locm'', \varepsilon)$ is in one-to-one correspondence with the set $H^2(\locm''; Z(\locm'))$.

\end{enumerate}

\end{Coro}

The proof of the above results, and specially of Theorem \ref{ThmA}, require an exhaustive analysis of the category of partial groups, both from an algebraic and a simplicial points of view. Once this is proved, we specialize our study to extensions of localities. Essentially, we give sufficient conditions for an extension of localities to produce a new locality.

Let $\locl' \to \locl \to \locl''$ be an extension of partial groups, where in addition we assume that $(\locl', \Delta', S')$ and $(\locl'', \Delta'', S'')$ are both localities. Under these assumptions, we see (Proposition \ref{propext3-1}) that such an extension determines a certain locality $(\loct, \Delta, S)$, with $\loct \subseteq \locl$. The extension $\locl' \to \locl \to \locl''$ is called \emph{good} if both $(\locl', \Delta', S')$ and $(\locl'', \Delta'', S'')$ are saturated, $\loct$ contains $\locl'$ as a normal partial subgroup and $\Delta$ contains all the centric radicals of the fusion system associated to $\loct$.

\begin{Th}\label{ThmC}

If an extension $\locl' \to \locl \to \locl''$ is good, then the following holds.
\begin{enumerate}[(i)]

\item The fusion system associated to $\loct$ is saturated, and the $p$-completed nerve of $\loct$ is equivalent to the classifying space of a $p$-local finite group.

\item The inclusion $\loct \subseteq \locl$ induces an equivalence of $p$-completed nerves.

\item $\FF'$ is a normal subsystem of $\FF$.

\end{enumerate}

\end{Th}

We also give sufficient conditions for an extension to be good in the above sense. To this end, we define the concepts of \emph{rigid} extension (to make sure that $\locl' \subseteq \loct$) and \emph{admissible} extension (to make sure that $\Delta$ contains all centric radical subgroups). In particular, the latter is inspired in and generalizes the concept introduced in \cite{OV}. See \ref{rigidext} for the exact definition.

\begin{Th}\label{ThmD}

If an extension $\locl' \to \locl \to \locl''$ is both rigid and admissible, then it is good.

\end{Th}

As an immediate consequence we deduce the following.

\begin{Coro}\label{CorE}

Every extension of finite groups gives rise to a good extension of localities.

\end{Coro}

The following result, partially communicated to the author by R. Levi, is another application of our results, which also illustrates a different point of view for our results. Let $F \to X \to B$ be a fibration where both $F$ and $B$ are $p$-good in the sense of \cite{BK}. In general it is not known whether $X$ is also $p$-good, since $p$-completion does not preserve fibrations. In this sense, the above result gives sufficient conditions for the space $X$ to be $p$-good, in the particular case where both $F$ and $B$ are nerves of partial groups.

\begin{Coro}\label{CorF}

Let $F \to X \to B$ be a fibration where both $B$ and $F$ are homotopy equivalent to classifying spaces of $p$-local finite groups. Then, $X$ is homotopy equivalent to the classifying space of a $p$-local finite group. Moreover, there exist proper localities $(\locl_F, \Delta_F, S_F)$ and $(\locl_B, \Delta_B, S_B)$, and a commutative diagram of fibre bundles
$$
\xymatrix{
F \ar[r] & X \ar[r] & B \\
|\locl_F| \ar[u] \ar[r] & |\locl| \ar[r] \ar[u] & |\locl_B| \ar[u]
}
$$
where the bottom row is (the realization of) a good extension and all the vertical arrows are equivalences after $p$-completion.

\end{Coro}

Recently, Chermak has announced a construction of quotients of localities by partial normal subgroups, in \cite{Chermak2, Chermak3}. His constructions seem to describe the converse process to our extension construction for localities, and we plan to study the relation between Chermak's work and the present work as a sequel in collaboration with O. Garaialde.

\emph{Organization of the paper.} The paper is organized as follows. In Section \ref{Sploc} we briefly review some basic facts about $p$-local finite groups. In Section \ref{partial} we review the theory of partial groups. Section \ref{simplicial} is a review on simplicial sets. In Section \ref{Spgsimp} we describe partial groups as simplicial sets. Section \ref{Maps} analyzes morphisms of partial groups from both an algebraic and a simplicial points of view. In Section \ref{FBPG} we introduce extensions of partial groups (as fibre bundles), and prove Theorem \ref{ThmA} and Corollary \ref{CorB}. In Section \ref{isoext} we specialize to extensions of localities, culminating in the proof of Theorems \ref{ThmC} and Theorem \ref{ThmD}. Section \ref{Sapp} contains the proofs of Corollaries \ref{CorE} and \ref{CorF}. Finally, we include Appendix \ref{AppA}, where we analyze the homotopy type of the nerve of a transporter system and the associated locality.

\emph{A note to the reader.} The duality in nature of partial groups, algebraic and simplicial, and  extension theory for partial groups described in these notes is joint work with Carles Broto. In this sense, the first part of this paper overlaps undergoing work of the author with Broto. This leads to some differences in notation, which still do not represent any difference in the results common to both papers.

\emph{Acknowledgements.} The author is extremely thankful to Carles Broto, whose inspiring ideas originated and shaped this paper, and to Andy Chermak, for several discussions and a seminar that helped consolidate our results. The author would like to thank also Oihana Garaialde for a careful reading of the notes and many fruitful conversations.


\section{Background on $p$-local finite groups}\label{Sploc}

In this section we review some basic facts about $p$-local finite groups that we will use throughout this paper. The reader is referred to \cite{BLO2} for further reference.

\begin{defi}

A \emph{fusion system} over a finite $p$-group $S$ is a category $\FF$ whose object set is the collection of all subgroups of $S$ and whose morphism sets satisfy the following conditions:
\begin{enumerate}[(i)]

\item $\Hom_S(P,Q) \subseteq \Hom_{\FF}(P,Q) \subseteq \Inj(P,Q)$  for all $P, Q \in \Ob(\FF)$; and

\item every morphism in $\FF$ factors as an isomorphism in $\FF$ followed by an inclusion.

\end{enumerate}

\end{defi}

Given a fusion system $\FF$, we say that $P,Q \in \Ob(\FF)$ are \emph{$\FF$-conjugate} if they are isomorphic as objects in $\FF$. The $\FF$-conjugacy class of an object $P$ is denoted by $P^{\FF}$.

\begin{defi}\label{defisat}

Let $\FF$ be a fusion system over a finite $p$-group $S$, and let $P \leq S$.
\begin{itemize}

\item $P$ is \emph{fully $\FF$-centralized} if $|C_S(P)| \geq |C_S(Q)|$ for all $Q \in P^{\FF}$.

\item $P$ is \emph{fully $\FF$-normalized} if $|N_S(P)| \geq |N_S(Q)|$ for all $Q \in P^{\FF}$.

\end{itemize}
The fusion system $\FF$ is \emph{saturated} if the following three conditions hold.
\begin{enumerate}[(I)]

\item For each $P \leq S$ which is fully $\FF$-normalized, $P$ is fully $\FF$-centralized, $\Out_{\FF}(P)$ is finite and $\Out_S(P) \in \Syl_p(\Out_{\FF}(P))$.

\item If $P \leq S$ and $f \in \Hom_{\FF}(P,S)$ are such that $f(P)$ is fully $\FF$-centralized, and if we set
$$
N_f = \{g \in N_S(P) \,\, | \,\, f \circ c_g \circ f^{-1} \in \Aut_S(f(P))\},
$$
there there is $\widetilde{f} \in \Hom_{\FF}(N_f, S)$ such that $\widetilde{f}|_P = f$.

\end{enumerate}

\end{defi}

\begin{rmk}

The above definition of saturation is taken from \cite{BLO2}, although several other definitions of saturation exist in the literature, for example see \cite{AKO} or \cite{Chermak}. All these definitions have been proved to be equivalent, and the reader can switch to her or his favorite definition, since this does not make any difference in this paper.

\end{rmk}

\begin{defi}\label{definormalA}

Let $\FF$ be a saturated fusion system over a finite $p$-group $S$.
\begin{itemize}

\item A subgroup $P \leq S$ is \emph{$\FF$-centric} if $C_S(Q) = Z(Q)$ for all $Q \in P^{\FF}$.

\item A subgroup $P \leq S$ is \emph{$\FF$-radical} if $\Out_{\FF}(P)$ contains no nontrivial normal $p$-subgroup.

\item A subgroup $A \leq S$ is \emph{weakly $\FF$-closed} if $f(A) = A$ for all $f \in \Hom_{\FF}(A, S)$.

\item A subgroup $A \leq S$ is \emph{strongly $\FF$-closed} if, for all $P \leq S$ and all $f \in \Hom_{\FF}(P,S)$, $f(P \cap A) \leq A$.

\item A subgroup $A \leq S$ is \emph{$\FF$-normal} if, for all $P \leq S$ and all $f \in \Hom_{\FF}(P,S)$, there is $\gamma \in \Hom_{\FF}(P\cdot A, S)$ such that $\gamma|_P = f$ and $\gamma|_A \in \Aut_{\FF}(A)$.

\item A subgroup $A \leq S$ is \emph{$\FF$-central} if $A$ is $\FF$-normal and $\Aut_{\FF}(A) = \{\Id\}$.

\end{itemize}
The \emph{center} of $\FF$, denoted by $Z(\FF)$, is the maximal subgroup of $Z(S)$ that is $\FF$-central.

\end{defi}

Given a saturated fusion system $\FF$ over a finite $p$-group $S$, we denote by $\FF^c$ and $\FF^r$ the full subcategories of $\FF$ with object sets the collections of $\FF$-centric and $\FF$-radical subgroups, respectively. We also set $\FF^{cr} \subseteq \FF$ for the full subcategory of $\FF$-centric $\FF$-radical subgroups. It is an easy exercise to check that an $\FF$-normal subgroup is always strongly $\FF$-closed, and a strongly $\FF$-closed subgroup is also weakly $\FF$-closed. In particular, if $A \leq S$ satisfies any of these three properties then $A$ is a normal subgroup of $S$. Regarding the definition of $\FF$-central subgroup, note that if $A$ is $\FF$-central then $A$ must be abelian, since $\Aut_{\FF}(A) = \{\Id\}$.

Next we review Aschbacher's definition of normal subsystem, as stated in \cite{Aschbacher}.

\begin{defi}\label{normalF}

Let $\FF$ be a saturated fusion system over a finite $p$-group $S$, and let $\varepsilon \subseteq \FF$ be a subsystem over a subgroup $R \leq S$. Then, $\varepsilon$ is a \emph{normal subsystem} of $\FF$ if the following conditions are satisfied.
\begin{itemize}

\item[(N1)] $\varepsilon$ is a saturated fusion system over $R$.

\item[(N2)] $R$ is strongly $\FF$-closed.

\item[(N3)] For each $P \leq Q \leq R$ and each $\gamma \in \Hom_{\FF}(Q,S)$, the map sending $f \in \Hom_{\varepsilon}(P,Q)$ to $\gamma \circ f \circ \gamma^{-1}$ is a bijection between the sets $\Hom_{\varepsilon}(P,Q)$ and $\Hom_{\varepsilon}(\gamma(P), \gamma(Q))$.

\item[(N4)] For each $f \in \Aut_{\varepsilon}(R)$ there is some extension $\widetilde{f} \in \Aut_{\FF}(R \cdot C_S(R))$ such that
$$
[\widetilde{f}, C_S(R)] \defin \{\widetilde{f}(x) \cdot x^{-1} \, \big| \, x \in C_S(R)\} \leq Z(R).
$$

\end{itemize}

\end{defi}

The concept of transporter system associated to a fusion system was originally introduced in \cite{OV}. Let $G$ be a group and let $\hh$ be a family of subgroups of $G$ which is invariant under $G$-conjugacy and overgroups. The transporter category of $G$ with respect to $\hh$ is the category $\TT_{\hh}(G)$ with object set $\hh$ and morphism sets
$$
\Mor_{\TT_{\hh}(G)}(P,Q) = \{x \in G \,\, | \,\, x \cdot P \cdot x^{-1} \leq Q\}
$$
for each pair of subgroups $P,Q \in \hh$.

\begin{defi}\label{defitransporter}

Let $\FF$ be a fusion system over a finite $p$-group $S$. A \emph{transporter system} associated to $\FF$ is a nonempty category $\TT$ such that $\Ob(\TT) \subseteq \Ob(\FF)$ is closed under $\FF$-conjugacy and overgroups, together with a pair of functors
$$
\TT_{\Ob(\TT)}(S) \Right4{\varepsilon} \TT \qquad \mbox{and} \qquad \TT \Right4{\rho} \FF
$$
satisfying the following conditions.
\begin{itemize}

\item[(A1)] The functor $\varepsilon$ is the identity on objects and an inclusion on morphism sets, and the functor $\rho$ is the inclusion on objects and a surjection on morphism sets.

\item[(A2)] For each $P, Q \in \Ob(\TT)$, the kernel
$$
E(P) \defin \Ker \big[\rho_P \colon \Aut_{\TT}(P) \Right2{} \Aut_{\FF}(P) \big]
$$
acts freely on $\Mor_{\TT}(P,Q)$ by right composition, and $\rho_{P,Q}$ is the orbit map of this action. Also, $E(Q)$ acts freely on $\Mor_{\TT}(P,Q)$ by left composition.

\item[(B)] For each $P,Q \in \Ob(\TT)$, $\varepsilon_{P,Q} \colon N_S(P,Q) \to \Mor_{\TT}(P,Q)$ is injective, and the composite $\rho_{P,Q} \circ \varepsilon_{P,Q}$ sends $g \in N_S(P,Q)$ to $c_g \in \Hom_{\FF}(P,Q)$.

\item[(C)] For all $\varphi \in \Mor_{\TT}(P,Q)$ and all $g \in P$, the diagram
$$
\xymatrix{
P \ar[r]^{\varphi} \ar[d]_{\varepsilon_P(g)} & Q \ar[d]^{\varepsilon_Q(\rho(\varphi)(g))} \\
P \ar[r]_{\varphi} & Q
}
$$
commutes in $\TT$.

\item[(I)] Each $\FF$-conjugacy class of subgroups in $\Ob(\TT)$ contains a subgroup $P$ such that $\varepsilon_P(N_S(P)) \in \Syl_p(\Aut_{\TT}(P))$; that is, such that $[\Aut_{\TT}(P)\colon \varepsilon(N_S(P))]$ is finite and prime to $p$.

\item[(II)] Let $\varphi \in \Iso_{\TT}(P,Q)$, $P \lhd \widetilde{P} \leq S$ and $Q \lhd \widetilde{Q} \leq S$ be such that $\varphi \circ \varepsilon_P(\widetilde{P}) \circ \varphi^{-1} \leq \varepsilon_Q(\widetilde{Q})$. Then there is some $\widetilde{\varphi} \in \Mor_{\TT}(\widetilde{P}, \widetilde{Q})$ such that $\widetilde{\varphi} \circ \varepsilon_{P, \widetilde{P}}(1) = \varepsilon_{Q, \widetilde{Q}}(1) \circ \varphi$.

\end{itemize}

A \emph{centric linking system} associated to a saturated fusion system $\FF$ is a transporter system $\LL$ such that $\Ob(\LL)$ is the collection of all $\FF$-centric subgroups of $S$ and $E(P) = Z(P)$ for all $P \in \Ob(\LL)$.

\end{defi}

\begin{defi}

A \emph{$p$-local finite group} is a triple $\g = \ploc$, where $S$ is a finite $p$-group, $\FF$ is a saturated fusion system over $S$, and $\LL$ is a centric linking system associated to $\FF$. The \emph{classifying space} of a $p$-local compact group $\g$ is the $p$-completed nerve of $\LL$, denoted by $B\g = |\LL|^{\wedge}_p$. The \emph{center} of $\g$ is the center of the fusion system $\FF$, and is denoted by $Z(\g)$.

\end{defi}

\begin{rmk}

Given a saturated fusion system $\FF$ over a finite $p$-group $S$, Chermak and Oliver proved, in \cite{Chermak} and \cite{Oliver} respectively, that there is an essentially unique centric linking system $\LL$ associated to $\FF$.

\end{rmk}

Eventually, we deal in this paper with fibrations involving classifying spaces of $p$-local finite groups. To this end, we recall the description of the topological monoid of self-equivalences of the classifying space of a $p$-local finite group.

\begin{defi}

Let $\g = \ploc$ be a $p$-local finite group. An automorphism $\Psi \colon \LL \Right2{\cong} \LL$ is \textit{isotypical} if $\Psi(\varepsilon_P(P)) = \varepsilon_{\Psi(P)}(\Psi(P))$ for each $P \in \Ob(\LL)$.

\end{defi}

Let $\Aut_{\typ}^I(\LL)$ be the collection of isotypical automorphisms of $\LL$ which send inclusions to inclusions. That is, $\Psi \in \Aut_{\typ}^I(\LL)$ if $\Psi(\varepsilon_{P,Q}(1)) = \varepsilon_{\Psi(P), \Psi(Q)}(1)$ whenever $P \leq Q$. This collection turns out to be a group by \cite[Lemma 1.14]{AOV}.

The elements of $\Aut_{\LL}(S)$ induce isotypical automorphisms of $\LL$ by conjugation, as follows. Fix $\varphi \in \Aut_{\LL}(S)$, and define $c_{\varphi}$ by
$$
c_{\varphi}(P) = \rho(\varphi)(P) \qquad \mbox{and} \qquad c_{\varphi}(\psi) = (\varphi|_{Q, c_{\varphi}(Q)}) \circ \psi \circ (\varphi^{-1}|_{c_{\varphi}(P), P})
$$
for each $P,Q \in \Ob(\LL)$ and all $\psi \in \Mor_{\LL}(P,Q)$. Notice that $c_{\varphi} \in \Aut_{\typ}^I(\LL)$ by construction. Actually, $\{c_{\varphi} \,| \, \varphi \in \Aut_{\LL}(S)\}$ is a normal subgroup of $\Aut_{\typ}^I(\LL)$, so we can define $\Out_{\typ}(\LL) \defin \Aut_{\typ}^I(\LL)/\{c_{\varphi} \,| \, \varphi \in \Aut_{\LL}(S)\}$. The following is a simplification of \cite[Theorem 8.1]{BLO2}.

\begin{prop}\label{auttyp}

Let $\g = \ploc$ be a $p$-local finite group. Then,
$$
\pi_i(\underline{\aut}(B\g)) \cong \left\{
\begin{array}{ll}
\Out_{\typ}(\LL) & i = 0 \\
Z(\g) & i = 1 \\
\{0\} & i \geq 2 \\
\end{array}
\right.
$$

\end{prop}


\section{Background on partial groups and localities}\label{partial}

In this section we review the definitions of partial groups and localities as they appeared in the seminal paper \cite{Chermak}. Roughly speaking, a partial group is a set with an associative operation (multiplication), a unit and an inversion (just like a group), but where some of the possible products are undefined. From this point of view, a partial group is formed by four pieces of data: a set of basic elements $\mm$, a subset $\dd(\mm)$ of the free monoid on $\mm$, a multiplication function $\Pi \colon \dd(\mm) \to \mm$, and an inversion $(-)^{-1}$.

A note of warning about this section. The convention in \cite{Chermak} is that maps apply \emph{from the right}, whereas our convention is that maps apply \emph{from the left}. For the sake of coherence, we present partial groups with the original notation and convention from \cite{Chermak}. In later sections we will adapt notation to make both conventions compatible with each other.

Let us start by fixing some notation. Let $\mm$ be a set, not necessarily finite, and let $\ww(\mm)$ be the free monoid on the set $\mm$. An element in $\ww(\mm)$ (also called a \emph{word in} $\ww(\mm)$) is a finite sequence of elements of $\mm$. The word $u \in \ww(\mm)$ formed by the letters $x_1, \ldots, x_n \in \mm$ will be represented by the symbol
$$
u = (x_1, \ldots, x_n).
$$
More generally, given words $u = (x_1, \ldots, x_n)$ and $v= (y_1, \ldots, y_m)$ in $\ww(\mm)$, its concatenation will be abbreviated by $u \circ v = (x_1, \ldots, x_n, y_1, \ldots, y_m)$.

Associated to $\ww(\mm)$ there is a \emph{length function} $l\colon \ww(\mm) \to \N$ which sends each element $w \in \ww(\mm)$ to the length of the sequence of elements of $\mm$ which forms $w$. The \emph{empty word}, represented by $(\emptyset)$, is the unique word of length $0$. When there is no place for confusion we may abbreviate the notation by writing $\ww$ instead of $\ww(\mm)$.

\begin{defi}\label{pgroup}

Let $\mm$ be a set, and let $\dd(\mm) \subseteq \ww(\mm)$ be a subset such that
\begin{enumerate}[(i)]

\item $\mm \subseteq \dd(\mm)$;

\item $u \circ v \in \dd(\mm) \Longrightarrow u, v \in \dd(\mm)$

\end{enumerate}
(in particular, $(\emptyset) \in \dd(\mm)$). A mapping $\Pi \colon \dd(\mm) \to \mm$ is a \emph{product} if
\begin{itemize}

\item[(P1)] $\Pi$ restricts to the identity map on $\mm$;

\item[(P2)] if $u \circ v \circ w \in \dd(\mm)$ then $u \circ \Pi(v) \circ w \in \dd(\mm)$ and
$$
\Pi(u \circ v \circ w) = \Pi(u \circ \Pi(v) \circ w).
$$

\end{itemize}
The \emph{unit} of a product $\Pi$ is defined as $1 = \Pi(\emptyset)$. A \emph{partial monoid} is a triple $(\mm, \dd(\mm), \Pi)$, where $\Pi$ is a product defined on $\dd(\mm)$.

An \emph{inversion} on $\mm$ is an involutory bijection $x \mapsto x^{-1}$ on $\mm$ together with the mapping $u \mapsto u^{-1}$ on $\ww(\mm)$ given by
$$
u = (x_1, \ldots, x_n) \mapsto (x_n^{-1}, \ldots, x_1^{-1}) = u^{-1}.
$$
A \emph{partial group} is a tuple $(\mm, \dd(\mm), \Pi, (-)^{-1})$ where $\Pi$ is a product on $\dd(\mm)$, and $(-)^{-1}$ is an inversion on $\mm$ satisfying
\begin{itemize}

\item[(I1)] if $u \in \dd(\mm)$ then $(u^{-1},u) \in \dd(\mm)$ and $\Pi(u^{-1}, u) = 1$.

\end{itemize}

\end{defi}

To simplify the notation, we will use $\mm$ to refer to a partial group $(\mm, \dd(\mm), \Pi, (-)^{-1})$ (or partial monoid) if the rest of the data is understood. The following is a summary of the basic properties of partial groups, as stated in \cite{Chermak}.

\begin{lmm}\label{2.2Ch}

Let $\mm$ be a partial group. Then,
\begin{enumerate}[(i)]

\item $\Pi$ is $\dd(\mm)$-multiplicative. That is, if $u \circ v \in \dd(\mm)$, then $\Pi(u)\circ \Pi(v)\in \dd(\mm)$, and
$$
\Pi(u \circ v) = \Pi(\Pi(u) \circ \Pi(v));
$$

\item $\Pi$ is $\dd(\mm)$-associative. That is, if $u \circ v \circ w \in \dd(\mm)$, then
$$
\Pi(\Pi(u \circ v) \circ w) = \Pi(u \circ \Pi(v \circ w));
$$

\item if $u \circ v \in \dd(\mm)$, then $u \circ 1 \circ v \in \dd(\mm)$ and
$$
\Pi(u \circ 1 \circ v) = \Pi(u \circ v);
$$

\item if $u \circ v \in \dd(\mm)$, then $w = u^{-1} \circ u \circ v, w' = u \circ v \circ v^{-1} \in \dd(\mm)$, and
$$
\begin{array}{ccc}
\Pi(w) = \Pi(v) & \mbox{ and } & \Pi(w') = \Pi(u);\\
\end{array}
$$

\item the left cancellation rule (similarly for right cancellation): if $u \circ v, u \circ w \in \dd(\mm)$ and we have $\Pi(u \circ v) = \Pi(u \circ w)$, then
$$
\Pi(v) = \Pi(w);
$$

\item if $u \in \dd(\mm)$, then $u^{-1} \in \dd(\mm)$, and $\Pi(u^{-1}) = \Pi(u)^{-1}$. In particular, $(1)^{-1} = 1$;

\item the left uncancellation rule (similarly for right uncancellation): let $u,v,w \in \dd(\mm)$ and suppose that $a = u \circ v, b = u \circ w \in \dd(\mm)$ with $\Pi(v) = \Pi(w)$. Then,
$$
\Pi(a) = \Pi(b).
$$

\end{enumerate}

\end{lmm}

\begin{rmk}\label{2.3Ch}

Let $\mm$ be a partial group. If $\dd(\mm) = \ww(\mm)$, then $\mm$ is an actual group via the binary operation $(u,v) \mapsto \Pi(u \circ v)$.

\end{rmk}

Let $\mm$ be a partial group, and let $\hh \subseteq \mm$ be a non-empty subset. Set also $\dd(\hh) = \dd(\mm) \cap \ww(\hh)$. Then, $\hh$ is called a \emph{partial subgroup} of $\mm$ if $\hh$ is closed under inversion and with respect to products: if $w \in \dd(\hh)$, then $\Pi(w) \in \hh$. If in addition, $\dd(\hh) = \ww(\hh)$, then $\hh$ is a \emph{subgroup} of $\mm$. We will use calligraphic letters $\hh, \kk, \ldots$ to denote partial subgroups, and straight letters $H, K, \ldots$ to denote subgroups of a given partial group $\mm$.

Let $u \in \mm$, and set $\dd(u) \defin \{x \in \mm \spa | \spa (u^{-1}, x, u) \in \dd(\mm)\}$. There is an obvious mapping
\begin{equation}\label{conjugation}
\xymatrix@R=1mm{
\dd(u) \ar[rr]^{c_u} & & \mm \\
x \ar@{|->}[rr] & & (x)c_u = \Pi(u^{-1},x, u) \\
}
\end{equation}
We also adopt the following convention: given elements $u, x \in \mm$, the symbol $x^u$ stands for $(x)c_u = \Pi(u, x, u^{-1})$, and in particular includes the assumption that $x \in \dd(u)$. Notice that even if $x, y \in \dd(u)$ and $(x, y) \in \dd(\mm)$, it is possible that $(u^{-1},x \cdot y, u) \notin \dd(\mm)$.

\begin{lmm}\label{2.5Ch}

Let $\mm$ be a partial group and let $u \in \mm$. Then,
\begin{enumerate}[(i)]

\item $1 \in \dd(u)$ and $(1)c_u = 1$;

\item $\dd(u)$ is closed under inversion and $(x^{-1})c_u = ((x)c_u)^{-1}$ for all $x \in \dd(u)$;

\item $c_u$ is a bijection $\dd(u) \to \dd(u^{-1})$, and $c_{u^{-1}} = (c_u)^{-1}$;

\item $\mm = \dd(1)$ and $(x)c_1 = x$ for all $x \in \mm$.

\end{enumerate}

\end{lmm}

If $H \leq \mm$ is a subgroup and $H \subseteq \dd(u)$ for some $u \in \mm$, we write $H^u = \{(x)c_u \, | \, x \in H\}$. More generally, given subgroups $H, K \leq \mm$, define
$$
\begin{array}{l}
N_{\mm}(H, K) \defin \{u \in \mm \, | \, H \subseteq \dd(u) \mbox{ and } H^u \leq K\} \\
N_{\mm}(H) \defin \{u \in \mm \, | \, H \subseteq \dd(u) \mbox{ and } H^u \leq H\} \\
C_{\mm}(H) \defin \{u \in \mm \, | \, x^u = x \mbox{ for all } x \in H\}
\end{array}
$$
Notice that, $H^u$ need not be a group with respect to $\Pi$, even if $H$ is! Furthermore, the induced map $c_u \colon H \to K$, mapping $h \mapsto h^u$, is not in general a group homomorphism.

\begin{defi}\label{morphpg}

Let $\mm, \mm'$ be partial groups, with $\dd(\mm), \dd(\mm)'$, multiplications $\Pi, \Pi'$ and units $1, 1'$ respectively. A mapping $\beta\colon \mm \Right1{} \mm'$ is called a \emph{morphism of partial groups} if
\begin{enumerate}[(i)]

\item $(\dd(\mm))\beta^{\ast} \subseteq \dd(\mm)'$; and

\item $(\Pi(u))\beta = \Pi'((u)\beta^{\ast})$ for all $u \in \dd(\mm)$,

\end{enumerate}
where $\beta^{\ast}\colon \ww(\mm) \to \ww(\mm')$ is the map induced by $\beta$. The morphism $\beta$ is an \emph{isomorphism} if there is a homomorphism $\beta'\colon \mm' \to \mm$ such that $\beta \circ \beta' = \Id_{\mm}$ and $\beta' \circ \beta = \Id_{\mm'}$.

\end{defi}

Homomorphisms of partial groups satisfy similar properties to those of group homomorphisms (see \cite[Lemma 3.2]{Chermak} for further details). Namely, if $\beta\colon \mm \to \mm'$ is a homomorphism of partial groups, then
\begin{enumerate}[(a)]

\item $(1)\beta = 1'$; and

\item $(u^{-1})\beta = ((u)\beta)^{-1}$ for all $u \in \mm$.

\end{enumerate}
With this notion of morphism, partial groups form a category.

\begin{defi}\label{npg}

Let $\mm$ be a partial group, and let $\nn \leq \mm$ be a partial subgroup. We say that $\nn$ is a \emph{partial normal subgroup} of $\mm$ if $x^g \in \nn$ whenever $x \in \nn \cap \dd(g)$. 

\end{defi}

\begin{lmm}\label{npg1}

Let $\beta \colon \mm \to \mm'$ be a morphism of partial groups, and define the \emph{kernel of $\beta$} as $\Ker(\beta) = \{x \in \mm \, | \, (x)\beta = 1'\}$. Then $\Ker(\beta)$ is a partial normal subgroup of $\mm$.

\end{lmm}

\begin{proof}

This is \cite[Lemma 3.3]{Chermak}.
\end{proof}

\begin{defi}\label{opgroup}

Let $\mm$ be a partial group, and let $\Delta$ be a collection of subgroups of $\mm$. Define also
$$
\dd_{\Delta} \defin \{w = (u_1, \ldots, u_n) \in \ww(\mm) \, | \, \exists X_0, \ldots, X_n \in \Delta \colon X_{i-1}^{u_i} = X_i, \, i = 1, \ldots, n\}.
$$
The pair $(\mm, \Delta)$ is an \emph{objective partial group} if
\begin{itemize}

\item[(O1)] $\dd(\mm) = \dd_{\Delta}$; and

\item[(O2)] if $X,Z \in \Delta$, $Y \leq Z$, and $u \in \mm$ are such that $X^u \subseteq Y$, then $N_Y(X^u) \in \Delta$. In particular, $X^u \in \Delta$.

\end{itemize}

\end{defi}

\begin{defi}\label{asscat}

Given an objective partial group $(\mm, \Delta)$, we can form the following associated categories.
\begin{itemize}

\item The \emph{transporter category} of $(\mm, \Delta)$, $\cat = \cat_{\Delta}(\mm)$, whose set of objects is $\Delta$ and with morphism sets $\Mor_{\cat}(X,Y) = N_{\mm}(X,Y)$ for each pair $X,Y \in \Delta$. Whenever $X \leq Y$, the morphism $1 \in \Mor_{\cat}(X,Y)$ is called an \emph{inclusion morphism}. Note that every morphism in $\cat$ factors uniquely as an inclusion morphism followed by an isomorphism.

\item The \emph{fusion category} of $(\mm, \Delta)$, $\FF = \FF_{\Delta}(\mm)$, whose objects are the groups $U$ such that $U \leq X$ for some $X \in \Delta$, and whose morphisms are compositions of restrictions of conjugation homomorphisms $c_u\colon X \to Y$ between objects in $\Delta$ (note that morphisms in $\FF$ are group homomorphisms).

\end{itemize}

\end{defi}

We are ready now to define localities. A (finite) group $G$ is of \emph{characteristic $p$} if $C_G(O_p(G)) \leq O_p(G)$. Roughly speaking, a locality is a partial group with a Sylow $p$-subgroup.

\begin{defi}\label{locality}

A \emph{locality} is a triple $\loc$, where $\LL$ is a finite partial group (i.e., $\LL$ is a finite set and has the structure of a partial group), $S$ is a finite $p$-subgroup of $\LL$, and $\Delta$ is a collection of subgroups of $S$, with $S \in \Delta$, and subject to the following conditions.
\begin{itemize}

\item[(L1)] $(\mm, \Delta)$ is an objective partial group; and

\item[(L2)] $S$ is maximal in the poset (ordered by inclusion) of $p$-subgroups of $\mm$.

\end{itemize}
A locality $\loc$ is \emph{proper} if it satisfies the following properties.
\begin{itemize}

\item[(PL1)] $\Delta$ contains all the $\FF$-centric $\FF$-radical subgroups.

\item[(PL2)] For each $P \in \Delta$, the group $N_{\LL}(P)$ is of characteristic $p$.

\end{itemize}
More generally, a locality is \emph{saturated} if it satisfies condition (PL1) above.

\end{defi}

Note that, in particular, if $\loc$ is a saturated locality, then $\FF_{\Delta}(\locl)$ is a saturated fusion system fby \cite[Theorem A]{BCGLO1}. In \cite[Appendix A]{Chermak} the author describes a bijective correspondence between isomorphism classes of localities and isomorphism classes of transporter systems, which  specializes to a correspondence between proper localities and linking systems. The reader is referred there for further details.

\begin{lmm}

Let $(\LL, \Delta, S)$ be a locality. Then, for each $(x_1, \ldots, x_n) \in \dd(\LL)$ and each $(s_1, \ldots, s_m) \in \ww(N_{\LL}(S))$, we have
\begin{enumerate}[(i)]

\item $(x_1, \ldots, x_n, s_1, \ldots, s_m) \in \dd(\LL)$; and

\item $(s_1, \ldots, s_m, x_1, \ldots, x_n) \in \dd(\LL)$.

\end{enumerate}
In particular, $\LL$ is an $(S,S)$-biset.

\end{lmm}

\begin{proof}

We prove (i) and leave (ii) to the reader. Since $(x_1, \ldots, x_n) \in \dd(\LL)$, there exist $H_0, \ldots, H_n \in \Delta$ such that $H_{i-1}^{x_i} = H_i$ for all $i = 1, \ldots, n$. Note also that $(s_1, \ldots, s_m) \in \dd(\LL)$ via the subgroup $S \in \Delta$, since $S^{s_j} = S$ for all $j = 1, \ldots, m$. Thus, $(x_1, \ldots, x_n, s_1, \ldots, s_m) \in \dd(\LL)$ via the sequence $H_0, \ldots, H_n, H_{n+1}, \ldots, H_{n+m} \in \Delta$, where $H_{n+j} = ((H_n^{s_1})^{s_2} \ldots)^{s_j}$ for all $j = 1, \ldots, m$.
\end{proof}

Let $\loc$ be a locality, and let $\omega = (x_1, \ldots, x_n) \in \ww(\LL)$ be a word. Define
\begin{equation}\label{SwSu0}
\begin{array}{l}
R_{\omega} = \{g_0 \in S \, | \, \exists g_0, \ldots, g_n \in S \mbox{ such that } (g_{i-1})^{x_i} = g_i \mbox{ for all } i = 1, \ldots, n\} \\
L_{\omega} = \{(((g_0)^{x_1})^{x_2} \ldots)^{x_n} \, | \, g_0 \in R_w\}.
\end{array}
\end{equation}
Notice that $L_{\omega} = R_{\omega^{-1}}$, where $\omega^{-1} = (x_n^{-1}, \ldots, x_1^{-1})$. By \cite[Lemma 2.14]{Chermak}, $R_{\omega}$ and $L_{\omega}$ are subgroups of $S$, and $\omega \in \dd(\LL)$ if and only if $R_{\omega}, L_{\omega} \in \Delta$. A note the the reader, the original notation for $R_{\omega}$ in \cite{Chermak} was $S_{\omega}$. We opted for altering this in order to accommodate the left/right conjugation conventions mentioned at the beginning of this section. This way, $R_{\omega}$ stands for the biggest subgroup of $S$ that can be \emph{right conjugated} by $\omega$ into $S$, while $L_{\omega}$ stands for the biggest subgroup of $S$ that can be \emph{left conjugated} by $\omega$ into $S$.

\begin{lmm}\label{SwSu}

Let $\loc$ be a locality, and let $\omega \in \ww(\LL)$, with $\omega = u \circ v$. Then,
\begin{enumerate}[(i)]

\item $R_{\omega} = R_u \cap (L_u \cap R_v)^{u^{-1}}$; and

\item if $\omega \in \dd(\LL)$, then $R_{\omega} \leq R_{\Pi(\omega)}$.

\end{enumerate}

\end{lmm}

\begin{proof}

This is immediate by definition.
\end{proof}


\section{Background on simplicial sets}\label{simplicial}

Let $\Delta$ denote the category of finite ordered sets. A \emph{simplicial set} is a functor
$$
X\colon \Delta \op \longrightarrow \Setsc
$$

Usually, we will replace $\Delta$ by the skeletal subcategory which whose objects are the sets $[n] = \{0, 1, \ldots, n\}$, for $n \geq 0$, and we will refer to this subcategory also as $\Delta$. Hence, we can think of a simplicial set as a sequence of sets $X_n = X([n])$, together with maps among them induced by non-decreasing functions $\varphi\colon [n] \to [m]$. Simplicial sets form a category whose morphisms are natural transformations among functors.

If $\Delta_n$ denotes the standard Euclidean $n$-simplex, then the assignment $[n] \in \Delta \mapsto \Delta_n \in \topo$ forms a functor by assigning to each morphism in $\Delta$ the corresponding sequence of inclusions and/or collapsing of faces in the usual way. The geometric realization of a simplicial set is the CW-complex defined as
$$
|X| = \left(\coprod_n \Delta_n \times X_n\right)/\sim
$$
where for each non-decreasing function $\varphi\colon [n] \to [m]$, if $t \in \Delta_m$ and $x \in X_n$, we identify $(\varphi(t),x)$ with $(t, X(\varphi(x)))$. Conversely, the set of singular simplices $\sigma: \Delta_n \to T$ of a topological space $T$ forms a simplicial set $\Sing(T)$. The geometric realization and the singular simplicial set form a pair of adjoint functors defining an equivalence between the homotopy categories of topological spaces and of simplicial sets. In this way, simplicial sets become convenient combinatorial models for topological spaces.

Finally, let us fix some notation. Given simplicial sets $X$ and $Y$,
\begin{itemize}

\item $\underline{\mathrm{hom}}(X,Y)$ is the simplicial set of maps from $X$ to $Y$: $(\underline{\mathrm{hom}}(X,Y))_n$ is the set of simplicial maps $\Delta[n] \times X \to Y$;

\item $\underline{\aut}(X)$ is the maximal subgroup inside the function complex $\underline{\mathrm{hom}}(X,X)$; and

\item $\Out(X) \defin \pi_0(\underline{\aut}(X))$.

\end{itemize}
The face and degeneracy operators on $\Delta^{\op}$ naturally induce on $\underline{\mathrm{hom}}(X,Y)$ the structure of a simplicial set. Furthermore, when $X = Y$, we can define in addition a multiplication in $\underline{\mathrm{hom}}(X,Y)$
$$
(\Delta[n] \times X \stackrel{\sigma} \to X) \cdot (\Delta[n] \times X \stackrel{\tau} \to X) = (\Delta[n] \times X \stackrel{\mathrm{pr}_1 \times \sigma} \longrightarrow \Delta[n] \times X \stackrel{\tau} \to X).
$$
There is an obvious group homomorphism $\Aut(X) \to \underline{\aut}(X)$ which sends $f \in \Aut(X)$ to the corresponding vertex in $\underline{\aut}(X)$, so in fact $\Aut(X)$ is a discretization of $\underline{\aut}(X)$.

Let us formalize certain operations that one can perform on the collection of simplices of any simplicial set, such as ``extracting'' the $r$-th front face of the $s$-th back face, or listing the edges. These operations will play a crucial role when studying partial groups from a simplicial point of view. To simplify notation, $d_m^k$ will denote the $k$-th iteration of the face operator $d_m$ on a given simplicial set.

\begin{defi}\label{frontbackop}

Let $X$ be a simplicial set, and let $r \in \N$. The \emph{$r$-front face operator} on $X$, $F_r$, is the collection of all set maps $F_{r,n}$, $n \geq r$, defined by 
$$
\xymatrix@R=1mm{
X_n \ar[rr]^{F_{r,n}} & & X_r \\
(\Delta[n] \stackrel{\sigma} \to X) \ar@{|->}[rr] & & (\Delta[r] \stackrel{F^r} \to \Delta[n] \stackrel{\sigma} \to X)
}
$$
where $F^r(i) = i$ for all $i \in \Delta[r]$. Similarly, for $s \in \N$, the \emph{$s$-back face operator} on $X$, $B_s$, is the collection of all set maps $B_{s,n}$, $n \geq s$, defined by
$$
\xymatrix@R=1mm{
X_n \ar[rr]^{B_{s,n}} & & X_s \\
(\Delta[n] \stackrel{\sigma} \to X) \ar@{|->}[rr] & & (\Delta[s] \stackrel{B^s} \to \Delta[n] \stackrel{\sigma} \to X)\\
}
$$
where $B^s(j) = n-s+j$ for all $j \in \Delta[s]$.

\end{defi}

Clearly, these operators are meaningless when applied to simplices of too small dimension. One could define them to be the identity in these situations, but this will not make a difference here.

Whenever $r + s = n$, these operators combine as $D_{r,s} = (F_r, B_s)\colon X_n \to X_r \times X_s$. This can be iterated in different ways to define operators $D_{r_1, r_2, \ldots, r_l}\colon X_n \to X_{r_1} \times X_{r_2} \times \ldots \times X_{r_l}$, which are of not much relevance in this paper in general. There is, however, a particular case which is of interest to us.

\begin{lmm}\label{ident1}

Let $X$ be a simplicial set, and let $\omega \in X_n$. Then,
\begin{enumerate}[(i)]

\item for all $j = 0, \ldots, n$,
$$
\begin{array}{ccc}
F_{n-j}(\omega) = d_n^j(\omega) & \mbox{ and } & B_{n-j}(\omega) = d_0^j(\omega) \\
\end{array}
$$

\item for all $i = 1, \ldots, n-1$,
$$
F_k(d_i(\omega)) = \left\{
\begin{array}{ll}
d_i(F_k(\omega)) & \mbox{, if } k> i\\
d_i(F_{k+1}(\omega)) & \mbox{, if } k = i\\
\end{array}
\right.
$$

\item for all $i = 1, \ldots, n-1$,
$$
B_k(d_i(\omega)) = \left\{
\begin{array}{ll}
d_i(B_k(\omega)) & \mbox{, if } k< i\\
d_i(B_{k+1}(\omega)) & \mbox{, if } k = i\\
\end{array}
\right.
$$

\end{enumerate}

\end{lmm}

\begin{defi}\label{enumop}

Let $X$ be a simplicial set. The \emph{enumerating operator}, $E$, is the collection of set maps $E_n$, $n \geq 1$, defined by
\begin{equation}\label{En}
E_n = D_{1, 1, \ldots, 1} \colon X_n \Right8{(F_{n-1},B_1)} X_{n-1} \times X_1 \Right8{(F_{n-2}, B_1)\times \Id} \ldots \Right8{(F_1, B_1) \times \Id} \underbrace{X_1 \times \ldots \times X_1}_{\mbox{$n$ times}}.
\end{equation}

\end{defi}

The name of the above operator comes from the fact that, for each $n$, the operator $E_n$ enumerates the edges of an arbitrary $n$-simplex in $X$. To simplify the notation we will write $E(\sigma)$ instead of $E_n(\sigma)$ for an $n$-simplex $\sigma \in X$ whenever its dimension is understood.

\begin{rmk}\label{enumop2}

Notice that, for each $n$, the operator $E_n$ can actually be defined in several ways, all giving the same result. For instance, for $n = 3$ we have
$$
\xymatrix{
 & & X_2 \times X_1 \ar[rrd]^{(F_1, B_1) \times \Id} & & \\
X_3 \ar[rrrr]|{E_3} \ar[rru]^{(F_2, B_1)} \ar[rrd]_{(F_1, B_2)} & & & & X_1 \times X_1 \times X_1 \\
 & & X_1 \times X_2 \ar[rru]_{\Id \times (F_1, B_1)} & & \\
}
$$
We will not enumerate here all the different decompositions of $E_n$ for a general $n$ since this would require an unnecessary amount of combinatorics, but we will use any equivalent definition of $E$ without further mention.

\end{rmk}

\begin{lmm}\label{propE}

The operator $E$ satisfies the following properties.
\begin{enumerate}[(i)]

\item Let $X$ be a simplicial set. Then, for each $n$ and each $0 \leq i \leq n$ there is a commutative diagram
$$
\xymatrix@R=15mm@C=15mm{
X_n \ar[r]^{s_i} \ar[d]_{E} & X_{n+1} \ar[d]^{E} \\
\hbox{$\underbrace{{X_1\times\dots \times X_1}}_{\text{$n$ times}}$} \ar[r]_{s_i} & \hbox{$\underbrace{{X_1\times\dots \times X_1}}_{\text{$n+1$ times}}$} \\
}
$$
where the bottom horizontal map is defined as
$$
s_i(x_1,\ldots, x_i, \ldots, x_n) = \left\{
\begin{array}{ll}
(s_0(d_1(x_1)), x_1, \ldots, x_n) & i = 0;\\
(x_1, \ldots, x_i, s_0(d_0(x_i)), x_{i+1}, \ldots, x_n) & 1 \leq i \leq n.
\end{array}
\right.
$$

\item Let $f\colon X \to Y$ be a simplicial map. Then, for each $n$ there is a commutative diagram
$$
\xymatrix@R=15mm@C=15mm{
X_n \ar[r]^{f} \ar[d]_{E} & Y_n \ar[d]^{E} \\
\hbox{$\underbrace{{X_1\times\dots \times X_1}}_{\text{$n$ times}}$} \ar[r]_{f \times \ldots \times f} & \hbox{$\underbrace{{Y_1\times\dots \times Y_1}}_{\text{$n$ times}}$}
}
$$

\end{enumerate}

\end{lmm}

\begin{defi}\label{prodop}

Let $X$ be a simplicial set. The \emph{product operator}, $\Pi$, is the collection of maps $\Pi_n$, $n \geq 0$, defined by
$$
\xymatrix@R=1mm{
X_n \ar[rr]^{\Pi_n} & & X_1 \\
(\Delta[n] \stackrel{f} \to X) \ar@{|->}[rr] & & (\Delta[1] \stackrel{\pi^n} \to \Delta[n] \stackrel{f} \to X)
}
$$
where $\pi^n(0) = 0$ and $\pi^n(1) = n$.

\end{defi}

\begin{rmk}

Note that, for each $n$, the map $\Pi_n$ is $d_1^{n-1}$, the face operator $d_1$ iterated $(n-1)$ times.

\end{rmk}

Eventually we may also use a variation of the product operator, which we call the \emph{$(r+s)$-product operator}, $\Pi_{r,s}$. This is just the collection of maps $\Pi_{r,s,n}$, $n \geq 0$, defined by
$$
\xymatrix@R=1mm{
X_n \ar[rr]^{\Pi_{r,s,n}} & & X_2 \\
(\Delta[n] \stackrel{f} \to X) \ar@{|->}[rr] & & (\Delta[2] \stackrel{\pi^{r,s}} \to \Delta[n] \stackrel{f} \to X)
}
$$
where $\pi^{r,s}(0) = 0$, $\pi^{r,s}(1) = r$ and $\pi^{r,s}(2) = n$.

\begin{lmm}\label{simprop}

Let $X$ be a simplicial set. Then, for each $n = r + s$, $\Pi = \Pi \circ \Pi_{r,s}$, and there is a commutative diagram
$$
\xymatrix@R=15mm@C=15mm{
X_n \ar@/^2pc/[rr]^\Pi \ar[r]^{\Pi_{r,s}} \ar[d]_{D_{r,s}} & X_2 \ar[r]^{\Pi} \ar[d]^{E} & X_1 \\
X_r \times X_s \ar[r]_{\Pi \times \Pi} & X_1 \times X_1
}
$$

\end{lmm}

\begin{proof}

This is immediate by definition.
\end{proof}

The rest of this section consists of a brief review of the necessary definitions and results about fibre bundles. For further reference the reader is directed to \cite{BGM}, \cite{BLO6}, \cite{Curtis} and \cite{GoJa}.

\begin{defi}\label{FB}

A simplicial map $\tau\colon X \to B$ is called a \emph{fibre bundle} if $\tau$ is onto and, for each simplex $b \in B$ there is an equivalence $\alpha(b)$ between the fibre over $b$ and a fixed simplicial set $F$ (which does not depend on $b$), making the following diagram strictly commutative
$$
\xymatrix{
F \times \Delta[n] \ar[r]^{\alpha(b)} \ar[d] & X^b \ar[r] \ar[d]^{\tau^b} & X \ar[d]^{\tau} \\
\Delta[n] \ar[r]_{\Id} & \Delta[n] \ar[r]_{b} & B
}
$$
The space $F$ is called the \emph{fibre of $\tau$}, and the collection $\{\alpha(b)\}_{b \in B}$ is called an \emph{atlas of $\tau$}.

\end{defi}

Let $\tau \colon X \to B$ be a fibre bundle with fibre $F$. Note that there may be more than one choice of an atlas for $\tau$. However, given atlases $\{\alpha(b)\}$ and $\{\beta(b)\}$, it is immediate that, for each $b \in B_n$,
$$
\alpha(b)^{-1} \cdot \beta(b) \in \underline{\aut}_n(F).
$$

\begin{defi}\label{TCP}

Given simplicial sets $B$ and $F$, a twisting function amounts to a collection $\{\phi_n \colon B_n \to \underline{\aut}_{n-1}(F)\}_{n \geq 1}$ such that, for all $n \geq 1$ and all $b \in B_n$, the following equations hold
\begin{itemize}

\item[($\tau_1$)] $\phi_{n-1}(d_i(b)) = d_{i-1}(\phi_n(b))$ for $2 \leq i \leq n$;

\item[($\tau_2$)] $\phi_{n-1}(d_1(b)) = d_0(\phi_n(b)) \cdot \phi_{n-1}(d_0(b))$;

\item[($\tau_3$)] $\phi_{n+1}(s_i(b)) = s_{i-1}(\phi_n(b))$ for $i \geq 1$; and

\item[($\tau_4$)] $\phi_{n+1}(s_0(b)) = 1 \in \underline{\aut}_n(F)$.

\end{itemize}
Given a twisting function $\{\phi_n \colon B_n \to \underline{\aut}_{n-1}(F)\}_{n \geq 1}$, the associated \emph{regular twisted Cartesian product} (RTCP for short) $X = F \times_{\phi} B$ is the simplicial set with $X_n = F_n \times B_n$, $n \geq 0$, and such that, for all $g \in F_n$, $b \in B_n$,
\begin{enumerate}[(i)]

\item $d_i(g,b) = (d_i(g), d_i(b))$, for all $n>0$ and all $0 < i \leq n$; and

\item $d_0(g,b) = (\phi(g,b) \cdot d_0(g), d_0(b))$.

\end{enumerate}

\end{defi}

By \cite[Proposition IV.5.3]{BGM}, an RTCP is a fibre-bundle via the projection map onto the second coordinate. Every RTCP in this paper will be considered as a fibre bundle this way, even if the projection map is not explicitly mentioned.

\begin{defi}

A \emph{map of fibre bundles} from $\tau \colon X \to B$ to $\chi \colon Y \to C$ is a pair of maps $f \colon X \to Y$ and $g \colon B \to C$ such that
$$
g \circ \tau = \chi \circ f.
$$
This way, an \emph{equivalence} of fibre bundles is a pair $(f,g)$ where both $f$ and $g$ are equivalences, and the pair $(f,g)$ is a \emph{strong equivalence} if in addition $g = \Id$.

\end{defi}

Let $\tau \colon X \to B$ be a fibre bundle with fibre $F$. Let also $\{\alpha(b)\}_{b \in B}$ be a choice of an atlas  for $\tau$. For each $b \in B_n$ and each $i = 1, \ldots, n$, there is an isomorphism $\alpha_i(b) \colon F \times \Delta[n-1] \to X^{d_i(b)}$ making the following diagram commutative
$$
\xymatrix{
F \times \Delta[n] \ar[r]^{\alpha(b)} & X^b \ar[r] & X \\
F \times \Delta[n-1] \ar[r]_{\alpha_i(b)} \ar[u]^{\Id \times s_i} & X^{d_i(b)} \ar[r] \ar[u] & X \ar@{=}[u]\\
}
$$
However, it is not true in general that $\alpha_i(b) = \alpha(d_i(b))$, and hence we can define the \emph{transformation elements}
$$
\psi_i(b) = (\alpha(d_i(b)))^{-1} \cdot \alpha_i(b).
$$
It is easy to check that $\psi_i(b) \in \underline{\aut}_{n-1}(F)$ for all $b \in B_n$ and all $i = 1, \ldots, n$.

Let $\Gamma \leq \underline{\aut}(F)$. An atlas $\{\alpha(b)\}$ for $\tau$ is a \emph{$\Gamma$-atlas} if all of its transformation elements live in $\Gamma$. Two atlases $\{\alpha(b)\}, \{\beta(b)\}$ for $\tau$ are \emph{$\Gamma$-equivalent} if, for each $b \in B$,
$$
\beta(b) = \alpha(b) \circ \gamma(b)
$$
for some $\gamma(b) \in \Gamma \leq \underline{\aut}(F)$. Clearly this is an equivalence relation on the set of all atlases of $\tau$. The following is \cite[Definition IV.2.4]{BGM}.

\begin{defi}\label{GFB}

Let $\Gamma \leq \underline{\aut}(F)$. A fibre bundle with fibre $F$ together with a given $\Gamma$-equivalence class of $\Gamma$-atlases is called a \emph{$\Gamma$-bundle}.

\end{defi}

Similarly, we may talk about $\Gamma$-RTCPs. In particular, if $\Gamma = \underline{\aut}(F)$ then we will simply refer to fibre bundles and RTCPs. Let $\tau \colon X \to B$ and $\chi \colon Y \to C$ be $\Gamma$-bundles. A map $(f, g) \colon (X,B) \to (Y,C)$ is a \emph{$\Gamma$-map} if for each $b \in B$
$$
(\beta(g(b)))^{-1} \circ f \circ \alpha(b) \in \Gamma
$$
provided that $\{\alpha(b)\}$ and $\{\beta(b)\}$ belong to the given $\Gamma$-equivalence classes of atlases of $\tau$ and $\chi$ respectively. In particular we may talk about \emph{$\Gamma$-equivalences} and \emph{strong $\Gamma$-equivalences}.

\begin{prop}\label{strongequiv}

Every $\Gamma$-bundle is strongly equivalent to a $\Gamma$-RTCP.

\end{prop}

\begin{proof}

This is \cite[Proposition IV.3.2]{BGM}.
\end{proof}

\begin{rmk}

Let $\tau \colon X \to B$ be a $\Gamma$-bundle with fibre $F$. By \cite[Theorem B and Corollary 0.6]{Weingram}, the realization of $\tau$ is a Serre fibration $|\tau| \colon |X| \to |B|$ with fibre $|F|$ and structural group $|\Gamma|$. We will use this property without any further mention.

\end{rmk}


\section{Partial groups from a simplicial point of view}\label{Spgsimp}

In this section we present a rather natural construction that allows one to consider partial groups as simplicial sets. The main goal of this section is to characterize all simplicial sets associated to partial groups. Recall that a simplicial set $X$ is \emph{reduced} if it has a single vertex.

\begin{defi}\label{BNsimpsets}

Let $X$ be a simplicial set. We say that $X$ is an \emph{$N$-simplicial set} if $X$ satisfies
\begin{itemize}

\item[(N)] the enumerating operator $E \colon X_n \to X_1 \times \ldots \times X_1$ is injective for all $n \geq 1$.

\end{itemize}
If in addition $X$ is a reduced $N$-simplicial set, with $X_0 = \{v\}$, then the \emph{unit} of $X$ is the $1$-simplex $1_X =s_0(v)$ obtained by applying the degeneracy operator to $v$.

\end{defi}

Notice that the nerve of every small category is an $N$-simplicial set, hence the name. A particular example is the nerve of a group (i.e., a category with one object all of whose morphisms are invertible), whose nerve is the so-called \emph{bar construction}. We borrow from this example the notation for simplices in an $N$-simplicial set. More precisely, let $X$ be an $N$-simplicial set, and let $\omega \in X_n$ be a simplex. By definition $\omega$ is determined by its image through the enumerating operator, $E(\omega) = (x_1, \ldots, x_n)$, and we adopt the following notation
$$
\omega = [x_1|\ldots|x_n]
$$
for a simplex $\omega \in X_n$ with $E(\omega) = (x_1, \ldots, x_n)$. With this notation, the action of the product operator is expressed by
$$
\Pi(\omega) = [x_1 \cdot \ldots \cdot x_n].
$$

\begin{rmk}

Note that not every tuple $(x_1, \ldots, x_n) \in X_1 \times \ldots \times  X_1$ defines an $n$-simplex.

\end{rmk}

Next we define a notion of \emph{inversion} for reduced $N$-simplicial sets. Consider the mapping $\chi_n\colon [n] \to [n]$ determined by $i \mapsto n-i$, and also let $\mathrm{op}: \Delta \to \Delta$ be the identity on objects. On morphisms,
$$
(f\colon [n] \to [m]) \mapsto \big((\chi_m \circ f)\colon [n] \to [m]\big).
$$
Given a simplicial set $X: \Delta\op \to \Setsc$, define the \emph{opposite simplicial set} $X\op$ as the composition $X\op: \Delta\op \Right1{\op} \Delta\op \Right1{X} \Setsc$. By definition, $(X\op)_n = X_n$, and $(d\op)_i = d_{n-i}, (s\op)_i = s_{n-i}$ for $0 \leq i \leq n$.

An \emph{anti-involution} of a simplicial set $X$ is a simplicial map $\nu: X \to X\op$ such that $\nu^2 = \nu \circ \nu\op = \Id_X$. Note that for each $n$, $(X\op)_n = X_n$, and $\nu_n = (\nu\op)_n$ as maps of sets. In particular, if $\nu$ as above exists, it has to be an isomorphism of simplicial sets.

\begin{defi}\label{inversion}

Let $X$ be a reduced $N$-simplicial set. An \emph{inversion} in $X$ is an anti-involution $\nu: X \to X\op$ such that for any $n \geq 1$ and any $x \in X_n$, $L(x) = [\nu(x)|x]$ is a simplex in $X_{2n}$ and $\Pi(L(x)) = 1_X$.

\end{defi}

The rest of this section is devoted to show that the category of partial groups is equivalent to the full subcategory of $\Sset$ of reduced $N$-simplicial sets with inversion. This is done in Theorem \ref{isocat} below.

\begin{lmm}

Let $X$ be a reduced $N$-simplicial set, with unit $1_X$, and let $\omega = [x_1|\ldots|x_n] \in X_n$.
\begin{enumerate}[(i)]

\item For all $0 \leq r \leq n$ we have $\sigma_r = [x_1\cdot \ldots \cdot x_r| x_{r+1} \cdot \ldots \cdot x_n] \in X_2$ and $\Pi(\sigma_r) = \Pi(\omega)$.

\item For all $i = 0, \ldots, n$ we have $\omega_i = [x_1|\ldots|x_i|1_X|x_{i+1}|\ldots|x_n] \in X_{n+1}$ and $\Pi(\omega_i) = \Pi(\omega)$.

\end{enumerate}

\end{lmm}

\begin{proof}

Both (i) and (ii) follow easily from the simplicial structure of $X$ and the definition of (reduced) $N$-simplicial set. In particular, (i) follows from Lemma \ref{simprop}, and (ii) follows since $\omega_i = s_i(\omega)$.
\end{proof}

\begin{rmk}

In particular, for each $[x_1|x_2|x_3] \in X_3$ and each $[y] \in X_1$, we have the following identities
$$
\Pi[x_1 \cdot x_2|x_3] = \Pi[x_1|x_2|x_3] = \Pi[x_1|x_2 \cdot x_3] \qquad \mbox{and} \qquad \Pi[1_X|y] = [y] = \Pi[y|1_X].
$$

\end{rmk}

\begin{lmm}\label{maps}

Let $Y$ be a reduced $N$-simplicial set (with inversion), and let $f\colon X \to Y$ be a simplicial map. Then, $f$ is determined by its restriction to the $1$-skeleton of $X$.

\end{lmm}

\begin{proof}

First note that since $Y$ is reduced, $Y_0 = \{e\}$, and thus $f(v) = e$ for any vertex $v \in X_0$. Assume also that we know the values of $f(x)$, for all $x \in X_1$, and let $\omega \in X_n$, with $E(\omega) = (x_1, \ldots, x_n) \in X_1 \times \ldots \times X_1$. Since $f$ is simplicial, it follows from Lemma \ref{propE} (ii) that
$$
E(f(\omega)) = f^n(E(\omega)) = (f(x_1), \ldots, f(x_n)) \in Y_1 \times \ldots \times Y_1.
$$
Thus, $f(\omega) = [f(x_1)|\ldots |f(x_n)]$. Finally, since $Y$ is and $N$-simplicial set, there can be a most one $n$-simplex $\sigma \in Y_n$ with $E(\sigma) = (f(x_1), \ldots, f(x_n))$. Hence, $f^n(E(\omega))$ determines $f(\omega)$.
\end{proof}

\begin{rmk}

In the situation above, a map $f_1\colon X^1 \to Y^1$ between the corresponding $1$-skeleton does not necessarily extend to a map $f\colon X \to Y$.

\end{rmk}

\begin{thm}\label{isocat}

The category of partial groups is equivalent to the full subcategory of $\Sset$ of reduced $N$-simplicial sets with inversion.

\end{thm}

\begin{proof}

Let $(\mm,\dd(\mm), \Pi, (-)^{-1})$ be a partial group with identity $1$, and define a simplicial set $\locm$ as follows. The set $\locm_n$ contains one simplex $\omega = [x_1|\ldots|x_n] \in \locm$ for each word $(x_1, \ldots, x_n) \in \dd(\mm)$ of length $n$. The face operators are defined by
$$
d_i(\omega) = \left\{\begin{array}{ll}
[x_2|\ldots|x_n] & i = 0\\

[x_1| \ldots | x_i \cdot x_{i+1} | x_{i+2} | \ldots |x_n] & 1 \leq i \leq n-1 \\

[x_1|\ldots|x_{n-1}] & i = n\\
\end{array}
\right.
$$
while the degeneracy operators are defined, for $i = 0, \ldots, n$, by
$$
s_i(\omega) = [x_1|\ldots|x_i|1|x_{i+1}|\ldots |x_n].
$$
This makes $\locm$ into a simplicial set. Furthermore, the inversion in $\mm$ induces naturally an inversion of this simplicial set, and now it is obvious that $\locm$ is a standard $N$-simplicial set with inversion. Conversely, if $X$ is a standard $N$-simplicial set with inversion, then set $\mm = X_1$, the collection of $1$-simplices of $X$. It is clear that the operator $\Pi$ on $X$ defines a product for $\mm$, and the inversion on $X$ induces an inversion on $\mm$.

As a consequence of Lemma \ref{maps}, a morphism between partial groups determines and is determined by a simplicial map between the induced simplicial sets:  a simplicial map $f\colon X \to Y$ between reduced $N$-simplicial sets with inversion determines and is determined by a multiplicative map between the partial monoids in degree $1$. 
\end{proof}

\begin{rmk}\label{autsMautM}

For simplicity, it is convenient to drop the partial group notation from Section \ref{partial}, and work only with simplicial sets. Thus, if $(\mm, \dd(\mm), \Pi, (-)^{-1})$ is a partial group and $\locm$ is the corresponding simplicial set, we will refer to $\locm$ as a partial group, assuming without further mention the equivalence proved above. Actually, this identification goes beyond the partial group itself: the group of automorphisms $\Aut(\mm)$ is clearly isomorphic to the group of invertible simplicial maps from $\locm$ to itself, $\Aut(\locm)$, and we also identify these two groups.

\end{rmk}


\section{Homomorphisms of partial groups}\label{Maps}

The goal of this section is to describe the automorphism simplicial set of a given partial group in terms of the algebraic structure of the partial group. In general, describing the automorphism complex of a simplicial set is a rather unaccessible problem, but in the case of partial groups the algebraic structure provides a beautifully simple description of the automorphism complex. Thus, given a partial group $\locm$, Theorem \ref{AutX0} provides a description of $\underline{\mathrm{aut}}(\locm)$ in terms of $\locm$ and $\Aut(\locm)$, while Theorem \ref{autX} describes some homotopical properties of $\underline{\mathrm{aut}}(\locm)$ again in terms of $\locm$ and $\Aut(\locm)$.

In this section we adopt again the convention that morphisms apply from the left. This affects essentially to conjugation by elements in a partial group, where we talk of \emph{left conjugation} by $[g]$, meaning ``$[g \cdot a \cdot g^{-1}]$'', and \emph{right conjugation} by $[g]$, meaning ``$[g^{-1} \cdot a \cdot g]$'', and which is the convention of \cite{Chermak}. To this end we also introduce the notation
$$
[\9{g}a] = [g \cdot a \cdot g^{-1}],
$$
to avoid confusion with (\ref{conjugation}).

Given a partial group $\locm$, we will denote by $v_{\locm}$ the single vertex of the associated simplicial set. Also, $\Delta[n]$ will denote the standard $n$-simplex. More specifically, $\Delta[n]$ is the nerve of the category $\Delta_n$ with objects $\bullet_0, \bullet_1, \ldots, \bullet_n$, and such that
$$
\Mor_{\Delta_n}(k-1, k) = \{\iota_k \colon k-1 \to k\}
$$
for $k = 1, \ldots, n$. Thus, $\iota_1, \ldots, \iota_n$ also denote the non-degenerate $1$-simplices of $\Delta[n]$. For instance, $\Delta[1]$ will be represented by
$$
\bullet_0 \Right4{\iota_1} \bullet_1.
$$

\begin{lmm}\label{homotopies}

Let $f, g\colon \locm \to \locm'$ be two simplicial maps between partial groups. A simplicial homotopy $F: \locm \times \Delta[1] \to \locm'$ with $f = F|_{\locl \times \{\bullet_0\}}$ and $g = F|_{\locl \times \{\bullet_1\}}$ is completely determined by the element
$$
[\eta] \defin F(v_\locm, \iota_1) \in \locm'_1.
$$

\end{lmm}

As we show in the proof of the statement above, we can interpret the action of $[\eta]$ as conjugating $f(\locm)$ to $g(\locm)$ (in a rather loose sense), so we can think of the homotopy itself as ``(left) conjugation by the element $[\eta]$''.

\begin{proof}

For each simplex $\omega \in \locm_n$, let $F_{\sigma} = F|_{\Delta[n] \times \Delta[1]} = F \circ (\omega \times \Id)$:
$$
F_{\omega}\colon \Delta[n] \times \Delta[1] \to X \times \Delta[1] \to Y.
$$
Since $\locm$ is the colimit of its simplices, $F$ is determined by $\{F_{\omega}\}_{\omega \in \locm}$, subject to the corresponding simplicial restrictions.

For $n = 0$, $F$ is completely determined since $\locm_0 = \{v_{\locm}\}$, $\locm'_0 = \{v_{\locm'}\}$. Indeed, there is essentially one map $v_{\locm}: \Delta[0] \to \locm$, and the map $F_{v_{\locm}}: \Delta[0] \times \Delta[1] \to \locm'$ is determined by the element $[\eta] = F(v_{\locm}, \iota_1)$.

Let now $n = 1$, and let $\omega\colon \Delta[1] \to \locm$. The product $\Delta[1] \times \Delta[1]$ can be thought of as a square, and as such its sets of vertices and edges are, respectively,
$$
V = \{\bullet_{(0,0)}, \bullet_{(0,1)}, \bullet_{(1,0)}, \bullet_{(1,1)}\} \qquad E = \{(0, \iota_1), (1, \iota_1), (\iota_1,0),(\iota_1,1),(\iota_1, \iota_1)\}.
$$
The map $F_{\omega}$ has to send all the vertices of $\Delta[1] \times \Delta[1]$ to the unique vertex $v_{\locm'} \in \locm'_0$. Also, it is easy to see that $F_{\omega}$ has the following effect
$$
\xymatrix@R=1mm{
(0,\iota_1) \ar@{|->}[rr] & & [\eta] & & (\iota_1, 0) \ar@{|->}[rr] & & [f(\omega)] \\
(1,\iota_1) \ar@{|->}[rr] & & [\eta] & & (\iota_1, 1) \ar@{|->}[rr] & & [g(\omega)] \\
}
$$
where the left column follows from the case $n = 0$. The following diagram represents the map $F_{\omega} \colon \Delta[1] \times \Delta[1] \to \locm'$ in dimensions $0$ and $1$.
$$
\xymatrix{
\bullet_{(0,1)} \ar[rr]^{(\iota_1,1)} & & \bullet_{(1,1)} & & & & v_{\locm'} \ar@{-}[dd]_{[\eta]} & & & v_{\locm'} \ar@{-}[lll]_{[g(\omega)]} \ar@{-}[dd]^{[\eta]} \ar@{-}[llldd]|{F(\iota_1,\iota_1)}\\
 & & & \ar@{|->}[rr]^{F_{\omega}} & & & & & & \\
\bullet_{(0,0)} \ar[rr]_{(\iota_1,0)} \ar[uu]^{(0,\iota_1)} \ar[rruu]|{(\iota_1,\iota_1)} & & \bullet_{(0,1)} \ar[uu]_{(1, \iota_1)} & & & & v_{\locm'} & & & v_{\locm'} \ar@{-}[lll]^{[f(\omega)]}\\
}
$$

Let now $a$ and $b$ be the non-degenerate $2$-simplices of $\Delta[1] \times \Delta[1]$, with edges $(0,\iota_1), (\iota_1, 1)$ and $(\iota_1,0), (1, \iota_1)$ respectively. Then, by Lemma \ref{maps},
$$
\xymatrix@R=1mm{
a \ar@{|->}[rr] & & [\eta|g(\omega)]  & & b \ar@{|->}[rr] & & [f(\omega)|\eta] \\
}
$$
and it follows that
\begin{equation}\label{conjeta}
\begin{aligned}\relax
[\eta \cdot g(\omega)] & = \Pi[\eta|g(\omega)] = d_1(F_{\omega}(a)) = F_{\omega}(d_1(a)) = \\
 & = F_{\omega}(\iota_1,\iota_1) = F_{\omega}(d_1(b)) = d_1(F_{\omega}(b)) = \Pi[f(\omega)|\eta] = [f(\omega) \cdot \eta].
\end{aligned}
\end{equation}

Finally, let $\omega = [x_1|\ldots|x_n] \in \locm$. The product $\Delta[n] \times \Delta[1]$ is subdivided into $n+1$ simplices, $a_0, \ldots, a_n$, of dimension $(n+1)$, and subject to obvious simplicial relations. The image of $F_{\omega}$ on each $a_i$ is easily determined by Lemma \ref{maps} as
\begin{equation}\label{homotopyn}
F_{\omega}(a_i) = [f(x_1)|\ldots|f(x_i)|\eta|g(x_{i+1})|\ldots|g(x_n)].
\end{equation}
Furthermore, the set $\{F_{\omega}(a_i)\}_{0 \leq i \leq n}$ satisfies $\Pi(F_{\omega}(a_0)) = \Pi(F_{\omega}(a_1)) = \ldots = \Pi(F_{\omega}(a_n))$, or equivalently
$$
\begin{aligned} \relax
[\eta \cdot g(x_1) \cdot \ldots \cdot g(x_n)] & = [f(x_1) \cdot \eta \cdot g(x_2) \cdot \ldots \cdot  g(x_n)] = \\
 & = [f(x_1) \cdot f(x_2) \cdot \eta \cdot g(x_3) \cdot \ldots \cdot g(x_n)] = \ldots = \\
 & = [f(x_1) \cdot \ldots \cdot f(x_n) \cdot \eta].
\end{aligned}
$$
This finishes the proof.
\end{proof}

It is now evident what we mean when we say that $[\eta]$ conjugates $f(\locm)$ to $g(\locm)$: the equation (\ref{conjeta}) is very close from the conjugation formula $``f(\omega) = \eta \cdot g(\omega) \cdot \eta^{-1}$''. Lemma \ref{homotopies} is a particular case of the following result, where we describe simplicial maps $F\colon \locm \times \Delta[m] \to \locm'$ for any $m$. Let $\bullet_0, \bullet_1, \ldots, \bullet_m$ be the set of vertices of $\Delta[m]$, and let $F\colon \locm \times \Delta[m] \to \locm'$ be a simplicial map. 

\begin{prop}\label{nhom}

Let $\locm$ and $\locm'$ be partial groups, and let $F\colon \locm \times \Delta[m] \to \locm'$ be a simplicial map. Then $F$ is completely determined by the following data
\begin{enumerate}[(i)]

\item the maps $f_i = F|_{\locm \times \{\bullet_i\}} \colon \locm \times \Delta[0] \Right4{\Id \times \{\bullet_i\}} \locm \times \Delta[m] \Right4{F} \locm'$, for $i = 0, \ldots, m$; and

\item the elements $[\eta_k] = F(v_{\locm},\iota_k) \in \locm'_1 , \quad k=1,\ldots,m$.

\end{enumerate}

\end{prop}

\begin{proof}

The map $F$ is determined by the collection of all restrictions $F_{\omega}$, where $\omega$ runs among all the simplices in $\locm$, $F_{\omega}\colon \Delta[n] \times \Delta[m] \Right2{\omega \times \Id} \locm \times \Delta[m] \Right2{F} \locm'$. To avoid confusion between $\Delta[n]$ and $\Delta[m]$, label the vertices of $\Delta[n]$ simply by $0, 1, \ldots, n$, and the non-degenerate $1$-simplices of $\Delta[n]$ by $\mu_1, \ldots, \mu_n$. The vertices of $\Delta[m]$ are denoted by $\bullet_0, \ldots, \bullet_m$, and its non-degenerate $1$-simplices by $\iota_1, \ldots, \iota_m$.

Fix some $\omega = [x_1|\ldots|x_n] \in \locm_n$, and consider $F_{\omega}$. Since $\locm'_0 = \{v\}$, the map $F_{\omega}$ sends all the vertices of $\Delta[n] \times \Delta[m]$ to $v$. Consider now the set of (non-degenerate) $1$-simplices of $\Delta[n] \times \Delta[m]$. Then $F_{\omega}(j,\iota_i) = [\eta_i]$ and $F_{\omega}(\mu_j, \bullet_i) = [f_i(x_j)]$.

Let now $\Gamma$ be the grid of size $n\times m$ determined by the $1$-skeleton of $\Delta[n] \times \Delta[m]$ (represented in the diagram below): the vertices of $\Gamma$ are the pairs $(j, \bullet_i)$, $0 \leq j \leq n$, $0 \leq i \leq m$, the horizontal edges of $\Gamma$ are the pairs $(\mu_j, \bullet_i)$, $1 \leq j \leq n$, $0 \leq i \leq m$, and the vertical edges of $\Gamma$ are the pairs $(j, \iota_i)$, $0 \leq j \leq n$, $1 \leq i \leq m$.

The number of $(n+m)$-simplices subdividing $\Delta[n] \times \Delta[m]$ is the binomial coefficient $\binom{n+m}{n}$, which also corresponds to the number of walks of $n+m$ steps in $\Gamma$ which start on the vertex $(0, \bullet_0)$ and end in the vertex $(n, \bullet_m)$.

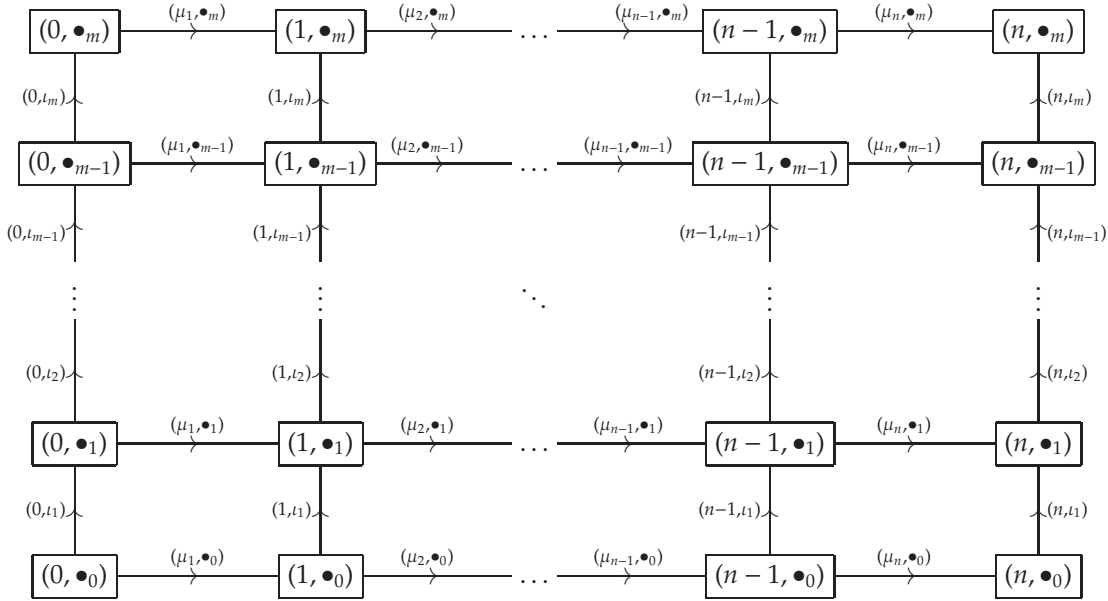
\footnotesize
$$
\xymatrix@R=12mm@C=18mm{
*+[F]{(0,\bullet_m)} \ar@{-}[r]|-{\SelectTips{cm}{}\object@{>}}^{(\mu_1,\bullet_m)} & *+[F]{(1, \bullet_m)} \ar@{-}[r]|-{\SelectTips{cm}{}\object@{>}}^{(\mu_2,\bullet_m)} & \ldots \ar@{-}[r]|-{\SelectTips{cm}{}\object@{>}}^{(\mu_{n-1},\bullet_m)} & *+[F]{(n-1,\bullet_m)} \ar@{-}[r]|-{\SelectTips{cm}{}\object@{>}}^{(\mu_{n},\bullet_m)} & *+[F]{(n,\bullet_m)} \\
*+[F]{(0,\bullet_{m-1})}  \ar@{-}[r]|-{\SelectTips{cm}{}\object@{>}}^{(\mu_1,\bullet_{m-1})} \ar@{-}[u]|-{\SelectTips{cm}{}\object@{>}}^{(0, \iota_m)} & *+[F]{(1,\bullet_{m-1})} \ar@{-}[r]|-{\SelectTips{cm}{}\object@{>}}^{(\mu_2,\bullet_{m-1})} \ar@{-}[u]|-{\SelectTips{cm}{}\object@{>}}^{(1, \iota_m)} & \ldots \ar@{-}[r]|-{\SelectTips{cm}{}\object@{>}}^(.4){(\mu_{n-1}, \bullet_{m-1})} & *+[F]{(n-1,\bullet_{m-1})} \ar@{-}[r]|-{\SelectTips{cm}{}\object@{>}}^{(\mu_n,\bullet_{m-1})} \ar@{-}[u]|-{\SelectTips{cm}{}\object@{>}}^{(n-1, \iota_m)} & *+[F]{(n,\bullet_{m-1})} \ar@{-}[u]|-{\SelectTips{cm}{}\object@{>}}_{(n, \iota_m)}\\
\vdots \ar@{-}[u]|-{\SelectTips{cm}{}\object@{>}}^{(0, \iota_{m-1})} & \vdots \ar@{-}[u]|-{\SelectTips{cm}{}\object@{>}}^{(1, \iota_{m-1})} & \ddots & \vdots \ar@{-}[u]|-{\SelectTips{cm}{}\object@{>}}^{(n-1, \iota_{m-1})} & \vdots \ar@{-}[u]|-{\SelectTips{cm}{}\object@{>}}_{(n, \iota_{m-1})}\\
*+[F]{(0,\bullet_1)}  \ar@{-}[r]|-{\SelectTips{cm}{}\object@{>}}^{(\mu_1,\bullet_1)} \ar@{-}[u]|-{\SelectTips{cm}{}\object@{>}}^{(0, \iota_2)} & *+[F]{(1,\bullet_1)} \ar@{-}[r]|-{\SelectTips{cm}{}\object@{>}}^{(\mu_2,\bullet_1)} \ar@{-}[u]|-{\SelectTips{cm}{}\object@{>}}^{(1, \iota_2)} & \ldots \ar@{-}[r]|-{\SelectTips{cm}{}\object@{>}}^(.4){(\mu_{n-1}, \bullet_1)} & *+[F]{(n-1,\bullet_1)} \ar@{-}[r]|-{\SelectTips{cm}{}\object@{>}}^{(\mu_n,\bullet_1)} \ar@{-}[u]|-{\SelectTips{cm}{}\object@{>}}^{(n-1, \iota_2)} & *+[F]{(n,\bullet_1)} \ar@{-}[u]|-{\SelectTips{cm}{}\object@{>}}_{(n, \iota_2)}\\
*+[F]{(0,\bullet_0)}  \ar@{-}[r]|-{\SelectTips{cm}{}\object@{>}}^{(\mu_1,\bullet_0)} \ar@{-}[u]|-{\SelectTips{cm}{}\object@{>}}^{(0, \iota_1)} & *+[F]{(1,\bullet_0)} \ar@{-}[r]|-{\SelectTips{cm}{}\object@{>}}^{(\mu_2,\bullet_0)} \ar@{-}[u]|-{\SelectTips{cm}{}\object@{>}}^{(1, \iota_1)} & \ldots \ar@{-}[r]|-{\SelectTips{cm}{}\object@{>}}^(.4){(\mu_{n-1}, \bullet_0)} & *+[F]{(n-1,\bullet_0)} \ar@{-}[r]|-{\SelectTips{cm}{}\object@{>}}^{(\mu_n,\bullet_0)} \ar@{-}[u]|-{\SelectTips{cm}{}\object@{>}}^{(n-1, \iota_1)} & *+[F]{(n,\bullet_0)} \ar@{-}[u]|-{\SelectTips{cm}{}\object@{>}}_{(n, \iota_1)}\\
}
$$
\normalsize

Let $\gamma = ((0, \bullet_0) \Right2{y_1} \ldots \Right2{y_{n+m}} (n, \bullet_m))$ be any of these walks, where $y_k$ denotes the $k$-th step in the walk $\gamma$. By definition, $y_k$ is either of the form $(j, \iota_i)$ or $(\mu_j, \bullet_i)$. Let also $a \subseteq \Delta[n] \times \Delta[m]$ be the $(n+m)$-simplex associated to $\gamma$. The image of $F_{\omega}$ is known on all the edges of $\Gamma$, and hence by Lemma \ref{maps},
$$
F_{\omega}(a) = [F_{\omega}(y_1)|\ldots|F_{\omega}(y_{n+m})] \in \locm_{n+m}.
$$
For instance, we have
$$
\begin{aligned}
F_{\omega}\big((0,\bullet_0) \Right2{(\mu_1, \bullet_0)}  (1,\bullet_0) \Right2{(\mu_2, \bullet_0)} & \ldots \Right2{(\mu_n, \bullet_0)} (n,\bullet_0) \Right2{(n, \iota_1)} (n, \bullet_1) \Right2{(n, \iota_2)} \ldots \Right2{(n, \iota_m)} (n, \bullet_m) \big) = \\[6pt] 
 & = [f_0(x_1)|\ldots|f_0(x_n)|\eta_1 |\ldots|\eta_m] \\[6pt]
F_{\omega}\big((0,\bullet_0) \Right2{(0, \iota_1)} (0,\bullet_1) \Right2{(0, \iota_2)}  & \ldots \Right2{(0, \iota_m)} (0,\bullet_m) \Right2{(\mu_1, \bullet_m)} (1, \bullet_m) \Right2{(\mu_2, \bullet_m)} \ldots \Right2{(\mu_n, \bullet_m)} (n, \bullet_m) \big) = \\[6pt]
 & =  [\eta_1|\ldots|\eta_m|f_m(x_1)|\ldots|f_m(x_n)]
\end{aligned}
$$

Clearly, describing the image of $F_{\omega}$ on each of the (non-degenerate) $(n+m)$-simplices of $\Delta[n] \times \Delta[m]$ is enough to describe $F_{\omega}$, and this proves the statement. Notice that for any two $(n+m)$-simplices as above, namely $a$ and $b$, we have $\Pi(a) = \Pi(b)$.
\end{proof}

Given a discrete group, left or right conjugation by any element of the group produces a self-equivalence of its nerve, which in addition is homotopic to the identity. This situation motivates the following definition.

\begin{defi}\label{definorm}

Let $\locm$ be a partial group with nerve $\locm$. The \emph{normalizer} of $\locm$, denoted by $N(\locm)$, is the set of elements $[\eta] \in \locm$ satisfying the following properties.
\begin{enumerate}[(i)]

\item For each $[x] \in \locm$, we have $[\eta|x|\eta^{-1}] \in \locm$, and the mapping defined by $[x] \mapsto [\9{\eta}x] = [\eta \cdot x \cdot \eta^{-1}]$ extends to an automorphism of $\locm$.

\item For each simplex $\omega = [x_1|\ldots|x_n] \in \locm_n$, we have $\omega_i = [\9{\eta}x_1|\ldots|\9{\eta}x_i|\eta|x_{i+1}|\ldots|x_n] \in \locm$, for $i = 0, \ldots, n$, and
$$
\Pi(\omega_0) = \Pi(\omega_1) = \ldots = \Pi(\omega_n).
$$

\end{enumerate}
The \emph{center} of $\locm$ is defined as the subset $Z(\locm) \subseteq N(\locm)$ of elements $[\eta] \in N(\locm)$ such that $[\9{\eta}x] = [x]$ for all $[x] \in \locm$. Finally, we say that $\alpha \in \Aut(\locm)$ is an \emph{inner automorphism} of $\locm$ if there is a homotopy $F$ from $\alpha$ to $\Id_{\locm}$. We will denote the set of all inner automorphisms by $\Inn(\locm)$.

\end{defi}

\begin{lmm}\label{N1homotopies}

Let $\locm$ be a partial group with nerve $\locm$, and let $F\colon \locm \times \Delta[m] \to \locm$ be a simplex in $\underline{\aut}(\locm)_m$. Then, for $1 \leq i \leq m$,
$$
[\eta_i] = F(v_{\locm}, \iota_i) \in N(\locm).
$$

\end{lmm}

\begin{proof}

The proof is straightforward by definition of $N(\locm)$ and Proposition \ref{nhom}.
\end{proof}

\begin{lmm}

Let $\locm$ be a partial group. Then the following holds.
\begin{enumerate}[(i)]

\item $\Psi[\eta] \in N(\locm)$ for each $[\eta] \in N(\locm)$ and each $\Psi \in \Aut(\locm)$.

\item Let $[\eta] \in N(\locm)$ and let $\alpha \in \Aut(\locm)$ be the left conjugation automorphism induced by $[\eta]$. Then $[\eta]$ determines a homotopy from $\alpha$ to $\Id$.

\item $\Inn(\locm)$ is a normal subgroup of $\Aut(\locm)$.

\end{enumerate}

\end{lmm}

\begin{proof}

Let $[\eta] \in N(\locm)$ and let $\Psi \in \Aut(\locm)$. Let also $\alpha \in \Aut(\locm)$ be the left conjugation automorphism induced by $[\eta]$. In particular, for each $[x] \in \locm$, we have $[\eta|x|\eta^{-1}] \in \locm$ and $\alpha[x] = [\eta \cdot x \cdot \eta^{-1}] = [\9{\eta}x]$. Applying $\Psi$ we get $[\Psi(\eta)|\Psi(x)|\Psi(\eta^{-1})] \in \locm$ and
$$
\Psi(\alpha[x]) = \Pi[\Psi(\eta)|\Psi(x)|\Psi(\eta)^{-1}].
$$
This implies that left conjugation by $\Psi[\eta]$ extends to an automorphism of $\locm$, and $\Psi[\eta]$ satisfies condition (i) in Definition \ref{definorm}. Furthermore, the above also shows that the left conjugation automorphism induced by $\Psi[\eta]$ corresponds to $\Psi \circ \alpha \circ \Psi^{-1}$.

To show that $\Psi[\eta]$ also satisfies condition (ii) in \ref{definorm}, let $\omega = [x_1|\ldots|x_n] \in \locm$. By definition, we have $\omega_i = [\9{\eta}x_1|\ldots|\9{\eta}x_i|\eta|x_{i+1}|\ldots|x_n] \in \locm$, for $i = 0, \ldots, n$, with $\Pi(\omega_0) = \ldots = \Pi(\omega_n)$. By applying $\Psi$ to $\omega_i$, we get
$$
\begin{aligned}
\Psi(\omega_i) & = [\Psi(\9{\eta}x_1)|\ldots|\Psi(\9{\eta}x_i)|\Psi(\eta)|\Psi(x_{i+1})|\ldots|\Psi(x_n)] = \\
 & = [\9{\Psi(\eta)}\Psi(x_1)|\ldots|\9{\Psi(\eta)}\Psi(x_i)|\Psi(\eta)|\Psi(x_{i+1})|\ldots|\Psi(x_n)] \in \locm,
\end{aligned}
$$
with $\Pi(\Psi(\omega_0)) = \ldots = \Pi(\Psi(\omega_n))$. Since $\Psi$ is an automorphism, condition (ii) in \ref{definorm} follows. Thus $\Psi[\eta] \in N(\locm)$.

Part (ii) in the statement follows easily from the properties of $N(\locm)$. Indeed, let $[\eta] \in N(\locm)$, and let $\alpha \in \Aut(\locm)$ be the left conjugation automorphism defined by $[\eta]$. Then, using property (ii) in Definition \ref{definorm}, it is clear that we can inductively construct a homotopy $F \colon \locm \times \Delta[1] \to \locm$ from $\alpha$ to $\Id$, following the reverse process that we followed in the proof of Lemma \ref{homotopies}.

Finally, part (iii) of the statement is a straightforward consequence of parts (i) and (ii). Let $\alpha \in \Inn(\locm)$ and let $\Psi \in \Aut(\locm)$. By definition of $\Inn(\locm)$, there exists a homotopy $F \colon \locm \times \Delta[1] \to \locm$ from $\alpha$ to $\Id$. As discussed above, the composition $\beta = \Psi \circ \alpha \circ \Psi^{-1}$ corresponds to conjugation $\Psi[\eta] \in N(\locm)$, and part (ii) implies that $\Psi[\eta]$ determines a homotopy from $\beta$ to $\Id$.
\end{proof}

\begin{defi}

The outer automorphism group of a partial group $\locm$ is the quotient $\Out(\locm) = \Aut(\locm)/\Inn(\locm)$.

\end{defi}

\begin{lmm}\label{norm-cent-inn}

Let $\locm$ be a partial group. Then, the following holds.
\begin{enumerate}[(i)]

\item $N(\locm)$ is a subgroup of $\locm$;

\item $Z(\locm)$ is an abelian subgroup of $\locm$; and

\item $\Inn(\locm) \cong N(\locm)/Z(\locm)$.

\end{enumerate}

\end{lmm}

\begin{proof}

Let us start by (i). By definition of partial group, $1_{\locm} \in N(\locm)$. We start by showing that, for each sequence $\eta_1, \ldots, \eta_m \in N(\locm)$, $\omega = [\eta_1|\ldots|\eta_m] \in \locm_m$ and $\Pi(\omega) \in N(\locm)$.

Let $\eta_1, \eta_2, \ldots, \eta_m \in N(\locm)$. By condition (ii) in Definition \ref{definorm}, we see inductively that $[\eta_1|\ldots|\eta_m]$ is a simplex in $\locm$, and thus we can consider $\sigma = \Pi[\eta_1|\ldots|\eta_m] \in \locm$. Let us check now that $\sigma \in N(\locm)$. For each $[x] \in \locm_1$ let
$$
[\9{\sigma}x] \defin \big[\, \9{\eta_1}(\ldots\9{\eta_{m-1}}(\9{\eta_m}x))\big].
$$
This way, it is clear that left conjugation by $\sigma$ induces an automorphism of $\locm$, which actually corresponds to $\alpha_1 \circ \ldots \circ \alpha_m$, where $\alpha_i$ is the left conjugation automorphism induced by $[\eta_i]$. Thus, $\sigma$ satisfies property (i) in Definition \ref{definorm}.

Fix now $\omega = [x_1|\ldots|x_n] \in \locm_n$. Since $\eta_1, \ldots, \eta_m \in N(\locm)$, by successively applying property (ii) in Definition \ref{definorm} we get simplices
$$
\begin{array}{l}
\omega_i^{(m)} = [\9{\eta_m}x_1|\ldots|\9{\eta_m}x_i|\eta_m|x_{i+1}|\ldots|x_n] \in \locm_{n+1} \\[10pt]
\omega_i^{(m-1)} = [\9{\eta_{m-1}}(\9{\eta_m}x_1)|\ldots|\9{\eta_{m-1}}(\9{\eta_m}x_i)|\eta_{m-1}|\eta_m|x_{i+1}|\ldots|x_n] \in \locm_{n+2} \\
\qquad \vdots \\
\omega_i^{(1)} = [\9{\sigma}x_1|\ldots|\9{\sigma}x_i|\eta_1|\ldots|\eta_m|x_{i+1}|\ldots|x_n] \in \locm_{n+m}
\end{array}
$$
and the appropriate application of $\Pi$ produces the simplex $\omega_i = [\9{\sigma}x_1|\ldots|\9{\sigma}x_i|\sigma|x_{i+1}|\ldots|x_n] \in \locm_{n+1}$. Furthermore, we have $\Pi(\omega_0) = \ldots = \Pi(\omega_m)$ by construction, and thus $\sigma \in N(\locm)$.

To finish the proof of part (i) in the statement, we have to show that $N(\locm)$ is closed by inversion. Fix some $[\eta] \in N(\locm)$, and let $\alpha \in \Aut(\locm)$ be the left conjugation automorphism induced by $[\eta]$. By property (i) in Definition \ref{definorm}, for each $[y] \in \locm$ we have $[\eta|y|\eta^{-1}] \in \locm$, and thus we also have $[\eta^{-1}|y^{-1}|\eta] \in \locm$. In particular, left conjugation by $[\eta^{-1}]$ induces the automorphism $\alpha^{-1}$, and thus $[\eta^{-1}]$ satisfies property (i) in \ref{definorm}.

To show that $[\eta^{-1}]$ also satisfies property (ii) in \ref{definorm}, fix some simplex $\omega = [x_1|\ldots|x_n] \in \locm$, and set
$$
\mu = [y_1|\ldots|y_n] \defin \alpha^{-1}(\omega^{-1}).
$$
In particular, $[y_i] = [\9{\eta^{-1}}(x_{n-i}^{-1}] = [(\9{\eta^{-1}}x_{n-1})^{-1}]$ for $i = 1, \ldots, n$. Now, since $[\eta] \in N(\locm)$, we can apply property (ii) in \ref{definorm} to get the simplices
$$
\mu_i = [\9{\eta}y_1|\ldots|\9{\eta}y_{n-i+1}|\eta|y_{n-i}|\ldots|y_n] = [x_n^{-1}|\ldots|x_{i+1}^{-1}|\eta|(\9{\eta^{-1}}x_i)^{-1}|\ldots |\9{\eta^{-1}}x_1)^{-1}] \in \locm_{n+1}.
$$
This way, we have $\omega_i = \mu_i^{-1} = [\9{\eta^{-1}}x_1|\ldots|\9{\eta^{-1}}x_i|\eta^{-1}|x_{i+1}|\ldots|x_n] \in \locm_{n+1}$ for each $i$. Moreover, it follows by construction that $\Pi(\omega_0) = \ldots = \Pi(\omega_n)$, and thus $[\eta^{-1}] \in N(\locm)$.

Restricting the above to elements of $N(\locm)$ which induce the identity on $\locm$ is a closed operation, and thus it follows that $Z(\locm)$ is a subgroup of $N(\locm)$. It is left to check that it is also abelian. Let $[\eta], [\sigma] \in Z(\locm)$ be any two elements. By condition (ii) in Definition \ref{definorm}, $[\eta| \sigma], [\9{\eta}\sigma| \eta]$ are both simplices in $\locm$, with $\Pi[\eta|\sigma] = \Pi[\9{\eta}\sigma|\eta]$. On the other hand, since $\eta \in Z(\locm)$, we have $[\9{\eta}\sigma] = [\sigma]$, and thus $\Pi[\sigma|\eta] = \Pi[\eta|\sigma]$. It follows that $Z(\locm)$ is abelian.

Finally, let us check (iii). First, we have to check that each inner automorphism of $\locm$ is the left conjugation automorphism induced by some $[\eta] \in N(\locm)$. Let $\alpha$ be an inner automorphism. By definition, there is a homotopy $F$ from $\alpha$ to $\Id_{\locm}$, and this homotopy is determined by a distinguished element $[\eta] \in N(\locm)$, by Lemma \ref{homotopies} and Lemma \ref{N1homotopies}. We claim that $\alpha$ is the left conjugation automorphism induced by $[\eta]$.

By Lemma \ref{N1homotopies}, for each $[x] \in \locm$ we have $[\alpha(x)|\eta], [\eta|x] \in \locm$, and $\Pi[\alpha(x)|\eta] = \Pi[\eta|x]$. On the other hand, since $[\eta] \in N(\locm)$, we have $[\9{\eta}x|\eta] \in \locm$, and $\Pi[\9{\eta}|x] = \Pi[\eta|x]$. Combining these two multiplications, and by the cancellation rule (Lemma \ref{2.2Ch} (v)) we deduce that $\alpha[x] = [\9{\eta}x]$ for all $[x] \in \locm_1$, and this proves the claim.

It is now easy to finish the proof of (iii). Let $[\eta] \in N(\locm)$, let $[z] \in Z(\locm)$, and let $[\sigma] = [z \cdot \eta] \in N(\locm)$. Let also $\alpha$ and $\beta$ be the left conjugation automorphisms induced by $[\eta]$ and $[\sigma]$ respectively. It is clear that $\alpha = \beta$. Conversely, if $[\sigma] \in N(\locm)$ induces the same left conjugation automorphism $\alpha$, then it is easy to check that $[\sigma \cdot \eta^{-1}] \in Z(\locm)$.
\end{proof}

We now focus on describing $\underline{\aut}(\locm)$ for a partial group $\locm$. Let $F \in \underline{\aut}(\locm)_n$ be an $n$-simplex. Let also $\Psi_0, \ldots, \Psi_n \in \Aut(\locm)$ be the automorphisms of $\locm$ obtained by restricting $F$ to the vertices of $\Delta[n]$, and let $[\eta_1], \ldots, [\eta_n] \in N(\locm)$ be the elements determined by $F$ by Lemma \ref{N1homotopies}. Notice that $F$ is completely determined by this data.

\begin{lmm}

For $i = 1,2$, let $F_i$ be a homotopy from $\Psi_i$ to $\Psi_{i-1}$, determined by an element $[\eta_i] \in N(\locm)$. Then, $[\eta_1 \cdot \eta_2] \in N(\locm)$ determines a homotopy $F$ from $\Psi_2$ to $\Psi_0$.

\end{lmm}

\begin{proof}

For each simplex $\omega = [x_1|\ldots|x_n] \in \locm_n$, let
$$
F_{1,\omega} \colon \Delta[n] \times \Delta[1] \Right3{\omega \times \Id} \locm \times \Delta[1] \Right3{F_1} \locm.
$$
If we denote by $a_0, \ldots, a_n$ the $(n+1)$-simplices subdividing $\Delta[n] \times \Delta[1]$, then $F_{1,\omega}$ maps the simplex $a_i$ to
$$
\omega_i = [\Psi_0(x_1)|\ldots|\Psi_0(x_i)|\eta_1|\Psi_1(x_{i+1})|\ldots|\Psi_1(x_n)].
$$
If we apply now $\Psi_1^{-1}$ to $\omega_i$ we get
$$
\Psi_1^{-1}(\omega_i) = [(\Psi_1^{-1} \circ \Psi_0)(x_1)| \ldots | (\Psi_1^{-1} \circ \Psi_0)(x_i)| \Psi_1^{-1}(\eta_1)|x_{i+1}|\ldots|x_n].
$$
Finally, the homotopy $F_2$ applied on $\Psi_1^{-1}(\omega_i)$ produces, in particular, the simplex
$$
[\Psi_0(x_1)|\ldots|\Psi_0(x_i)|\eta_1|\eta_2|\Psi_2(x_{i+1})|\ldots|\Psi_2(x_n)],
$$
and thus $\sigma_i = [\Psi_0(x_1)|\ldots|\Psi_0(x_i)|\eta_1\cdot \eta_2|\Psi_2(x_{i+1})|\ldots|\Psi_2(x_n)] \in \locm$. Notice that $\Pi(\sigma_0) = \ldots = \Pi(\sigma_n)$ by definition. Using the formulae (\ref{homotopyn}) we see how $[\eta_1 \cdot \eta_2] \in N(\locm)$ determines a homotopy from $\Psi_2$ to $\Psi_0$.
\end{proof}

\begin{defi}\label{calautM}

Define the category $\autcat(\locm)$ as follows. Its object set is $\Ob(\autcat(\locm)) = \Aut(\locm)$, and its morphism sets are
$$
\Mor_{\autcat(\locm)}(\Phi, \Psi) = \{\eta \,\, \big| \,\, [\eta] = F(v, \iota_1) \in N(\locm) \mbox{ for a homotopy } F \mbox{ from } \Psi \mbox{ to } \Phi\}.
$$
A morphism in this category will be represented by $(\Phi \Left1{\eta} \Psi)$. With this notation the composition of morphisms in this category is given by multiplication in $\locm$:
$$
\big(\Psi_0 \Left3{\eta_1} \Psi_1\big) \circ \big(\Psi_1 \Left3{\eta_2} \Psi_2\big) = \big(\Psi_0 \Left6{\eta_1 \cdot \eta_2} \Psi_2\big).
$$

\end{defi}

Actually, the category $\autcat(\locm)$ is strict monoidal, as we describe next. Let $(\Psi_0 \Left1{\eta} \Psi_1)$ be a morphism in $\autcat(\locm)$. Seen as a homotopy $F \colon \locm \times \Delta[1] \to \locm$, the morphism $(\Psi_0 \Left1{\eta} \Psi_1)$ has the following effect on a $1$-simplex $[x] \in \locm$:
\begin{equation}\label{action1}
(\Psi_0 \Left3{\eta} \Psi_1) \cdot [x] = [\Psi_0(x) \cdot \eta] = [\eta \cdot \Psi_1(x)],
\end{equation}
where the rightmorst equality follows from Lemma \ref{homotopies}. If $(\Phi_0 \Left1{\omega} \Phi_1)$ is another morphism in $\autcat(\locm)$, the above formula implies
$$
\begin{aligned}
(\Phi_0 \Left3{\omega} \Phi_1) \cdot & \big((\Psi_0 \Left3{\eta}  \Psi_1) \cdot [x] \big) = \\
 & = (\Phi_0 \Left3{\omega} \Phi_1) \cdot [\Psi_0(x) \cdot \eta] = (\Phi_0 \Left3{\omega} \Phi_1) \cdot [\eta \cdot \Psi_1(x)] = \\
 & = [\Phi_0(\Psi_0(x) \cdot \eta) \cdot \omega] = [\omega \cdot \Phi_1(\eta \cdot \Psi_1(x))] = \\
 & = (\Phi_0 \circ \Psi_0 \LEFT6{\Phi_0(\eta) \omega = }{= \omega \Phi_1(\eta)} \Phi_1 \circ \Psi_1) \cdot [x].
\end{aligned}
$$
Using this formula we can now define an operation on $\autcat(\locm)$ as follows.
\begin{equation}\label{productdim1}
\xymatrix@R=1mm{
\autcat(\locm) \times \autcat(\locm) \ar[rr]^{\otimes} & & \autcat(\locm) \\
(\Phi, \Psi) \ar@{|->}[rr] & & \Phi \circ \Psi \\
\big((\Phi_0 \Left1{\omega} \Phi_1), (\Psi_0 \Left1{\eta} \Psi_1) \big) \ar@{|->}[rr] & & (\Phi_0 \circ \Psi_0 \LEFT6{\Phi_0(\eta) \omega = }{= \omega \Phi_1(\eta)} \Phi_1 \circ \Psi_1)
}
\end{equation}
The unit for this operation is $(\Id_{\locm} \Left1{1_{\locm}} \Id_{\locm})$, and the inverse is given by
\begin{equation}\label{inverse1}
(\Psi_0 \Left3{\eta} \Psi_1)^{-1} = (\Psi_0^{-1} \LEFT8{\Psi_0^{-1}(\eta^{-1}) = }{ = \Psi_1^{-1}(\eta^{-1})} \Psi_1^{-1}).
\end{equation}

In later sections we will need generalized versions of some of the formulas showed above, in particular of (\ref{action1}) and (\ref{inverse1}). Let $(\Psi_0 \Left1{\eta_1} \ldots \Left1{\eta_n} \Psi_n)$ be an $n$-simplex in $N \autcat(\locm)$. Formula (\ref{action1}) generalizes to
\begin{equation}\label{actionn}
\begin{aligned}
(\Psi_0 \Left1{\eta_1} \Psi_1 \Left1{\eta_2} \ldots \Left1{\eta_n} \Psi_n) \cdot [x_1|\ldots|x_n] & = [\Psi_0(x_1) \cdot \eta_1 | \Psi_1(x_2) \eta_2 |\ldots |\Psi_{n-1}(x_n) \cdot \eta_n] = \\
& = [\eta_1 \cdot \Psi_1(x_1)| \eta_2 \cdot \Psi_2(x_2)| \ldots |\eta_n \Psi_n(x_n)].
\end{aligned}
\end{equation}
while formula (\ref{inverse1}) generalizes as follows
\begin{equation}\label{inversen}
\begin{aligned}
(\Psi_0 \Left3{\eta_1} \Psi_1 & \Left3{\eta_2} \ldots \Left3{\eta_n} \Psi_n)^{-1} = \\
 & = (\Psi_0^{-1} \LEFT8{\Psi_1^{-1}(\eta_1^{-1}) = }{= \Psi_0^{-1}(\eta_1^{-1})} \Psi_1^{-1} \LEFT8{\Psi_2^{-1}(\eta_2^{-1}) = }{= \Psi_1^{-1}(\eta_2^{-1})} \ldots \LEFT8{\Psi_n^{-1}(\eta_n^{-1}) = }{= \Psi_{n-1}^{-1}(\eta_n^{-1})} \Psi_n^{-1}).
\end{aligned}
\end{equation}
For the sake of completeness we also include the formula for the product in $N_n \autcat(\locm)$, although it plays no relevant role in these notes.
$$
\begin{aligned}
(\Phi_0 \Left3{\omega_1} \Phi_1 & \Left3{\omega_2} \ldots \Left3{\omega_n} \Phi_n) \otimes (\Psi_0 \Left3{\eta_1} \Psi_1 \Left3{\eta_2} \ldots \Left3{\eta_n} \Psi_n) = \\[6pt]
 & = (\Phi_0 \circ \Psi_0 \LEFT8{\omega_1 \Phi_1(\eta_1) = }{ = \Phi_0(\eta_1) \omega_1} \Phi_1 \circ \Psi_1 \LEFT8{\omega_2 \Phi_2(\eta_2) = }{ = \Phi_1(\eta_2) \omega_2} \ldots \LEFT8{\omega_n \Phi_n(\eta_n) = }{ = \Phi_{n-1}(\eta_n) \omega_n} \Phi_n \circ \Psi_n).
\end{aligned}
$$

\begin{thm}\label{AutX0}

Let $\locm$ be a partial group. Then, there is an equivalence of simplicial groups
$$
N\autcat(\locm) \cong \underline{\mathrm{aut}}(\locm).
$$

\end{thm}

\begin{proof}

Consider the evaluation map $N\autcat(\locm) \times \locm \Right2{\kappa} \locm$, whose effect on $1$-simplices is given by formula (\ref{action1}):
$$
\kappa\big((\Psi_0 \Left3{\eta} \Psi_1), [x] \big) = [\Psi_0(x) \cdot \eta] = [\eta \cdot \Psi_1(x)].
$$
Recall that this is enough to determine $\kappa$ by Lemma \ref{maps}. One checks easily that this action of $N\autcat(\locm)$ on $\locm$ is associative, and hence induces by adjunction a simplicial map
$$
N\autcat(\locm) \Right2{\tilde{\kappa}} \underline{\aut}(\locm).
$$
By definition of $N\autcat(\locm)$ and $\underline{\aut}(\locm)$, together with Lemma \ref{homotopies}, it is clear that $\tilde{\kappa}$ is injective. By Proposition \ref{nhom} it also follows easily that $\tilde{\kappa}$ is surjective.
\end{proof}

For simplicity, we will identify from now on the simplicial groups $\underline{\aut}(\locm)$ and $N \autcat(\locm)$. Next we give a description of the homotopy type of $N \autcat(\locm)$, or rather of its \emph{classifying space} $B\autcat(\locm)$, in terms of the partial group $\locm$. This classifying space is obtained via the so-called \emph{$W$-construction}, which we omit in these notes for the sake of brevity. The interested reader is referred to \cite[Section IV.5]{BGM} for further details. 

\begin{thm}\label{autX}

Let $\locm$ be a partial group. Then, the following holds
\begin{enumerate}[(i)]

\item for each connected component $\cala_i \subseteq N\autcat(\locm)$ there is an equivalence $\cala_i \simeq BZ(\locm)$;

\item there is a fibration $B^2Z(\locm) \longrightarrow B\autcat(\locm) \longrightarrow B\Out(\locm)$;

\item there is an exact sequence of groups $\{1\} \to Z(\locm) \to N(\locm) \to \Aut(\locm) \to \Out(\locm) \to \{1\}$.

\end{enumerate}

\end{thm}

\begin{proof}

The connected components of $\autcat(\locm)$ correspond to $\Out(\locm)$. Also, for each automorphism $\Psi \in \Aut(\locm)$, $\Aut_{\autcat(\locm)}(\Psi) = Z(\locm)$ by definition of $\autcat(\locm)$. Thus, the nerve $N\autcat(\locm)$ is a simplicial group with
$$
\Out(\locm) = \pi_0(N\autcat(\locm)),
$$
and each component is equivalent to $BZ(\locm)$. This proves points (i). Furthermore, the classifying space of $N\autcat(\locm)$, $B\autcat(\locm)$, has fundamental group $\Out(\locm)$ and universal cover $B^2Z(\locm)$, that is, there is a fibration
$$
B^2Z(\locm) \longrightarrow B\autcat(\locm) \longrightarrow B\Out(\locm)
$$
and this proves point (ii). Finally, the exact sequence in (iii) is a direct consequence of Lemma \ref{norm-cent-inn} (iii), since $\Inn(\locm) \cong N(\locm)/Z(\locm)$.
\end{proof}

To finish this section we analyze the automorphism complex of a locality $(\locl, \Delta, S)$. By definition, $\locl$ contains the $p$-group $S$ as a subgroup, and we can define
$$
\Aut_(\locl;S) \defin \{\Psi \in \Aut(\locl) \,\, \big| \,\, \Psi(S) = S\} \leq \Aut(\locl).
$$
We can define a category $\autcat(\locl;S)$ by means of $\Aut(\locl;S)$ and $N_{\locl}(S)$ as we did with $\autcat(\locl)$.

\begin{lmm}\label{autloc0}

Let $\underline{\aut}(\locl;S) \subseteq \underline{\aut}(\locl)$ be the simplicial subset with vertex set $\Aut(\locl;S)$. Then there is an equivalence $N \autcat(\locl;S) \cong \underline{\aut}(\locl;S)$. In particular,
$$
\pi_0(\underline{\aut}(\locl;S)) \cong \Out(\locl;S) \quad \mbox{and} \quad \pi_i(\underline{\aut}(\locl;S)) \cong \left\{
\begin{array}{ll}
Z(\locl) & \mbox{if } i = 1;\\
0 & \mbox{if } i \geq 2.
\end{array}
\right.
$$

\end{lmm}

If we set $\Out(\locl;S) \defin \pi_0(\underline{\aut}(\locl;S))$ then it is an easy exercise to check that there is an isomorphism $\Out(\locl;S) \cong \Aut(\locl;S)/(N_{\locl}(S)/Z(\locl))$, and hence there is an exact sequence
$$
\{1\} \to Z(\locl) \Right4{} N_{\locl}(S) \Right4{} \Aut(\locl;S) \Right4{} \Out(\locl;S) \to \{1\}.
$$
The reader can compare this exact sequence with that in Theorem \ref{autX} (iii) and that in \cite[Lemma 1.14]{AOV}.

\begin{lmm}\label{autloc1}

There is a commutative diagram of exact sequences
$$
\xymatrix{
\{1\} \ar[r] & Z(\locl) \ar[r] &N(\locl) \ar[r] & \Aut(\locl) \ar[r] & \Out(\locl) \ar[r] & \{1\} \\
\{1\} \ar[r] & Z(\locl) \ar@{=}[u] \ar[r] & N_{\locl}(S) \ar[u]^{\incl} \ar[r] & \Aut(\locl;S) \ar[u]_{\incl} \ar[r] & \Out(\locl;S) \ar[u] \ar[r] & \{1\}
}
$$
where all the vertical arrows are monomorphisms.

\end{lmm}

\begin{lmm}

Let $\g = \ploc$ be a $p$-local finite group and let $(\locl, \Delta, S)$ be the associated locality. Then, there are isomorphisms
$$
\Aut(\locl;S) \cong \Aut_{\typ}^{I}(\LL) \qquad\mbox{and} \qquad \Out(\locl;S) \cong \Out_{\typ}(\LL).
$$

\end{lmm}

The following is an immediate consequence of Lemma \ref{autloc0} and \cite[Theorem 7.1]{BLO3}.

\begin{cor}\label{autloc2}

Let $\g = \ploc$ be a $p$-local finite group and let $(\locl, \Delta, S)$ be its associated locality. Then there are homotopy equivalences of topological monoids
$$
|\underline{\aut}(\locl;S)| \simeq |\autcat(\locl;S)| \simeq \Aut(B\g).
$$

\end{cor}


\section{Extensions of partial groups}\label{FBPG}

In the previous sections we have seen how partial groups can be treated as simplicial objects. Following this point of view, we now show that the total space of a fibre bundle where both base and fibre are partial groups is always a partial group. This will lead to the definition of extension of partial groups.

By \cite[Proposition IV.3.2]{BGM} we know that every fibre bundle is strongly equivalent to an RTCP. Thus, we start by studying RTCPs of partial groups, and in particular we start by describing twisting functions in this setting. Let $\locm'$ and $\locm''$ be partial groups. Recall that a twisting function is a collection
$$
\{\phi_n \colon \locm''_n \to \underline{\aut}_{n-1}(\locm') = N_{n-1} \autcat(\locm')\}_{n \geq 1},
$$
of functions satisfying conditions ($\tau_1$)-($\tau_4$) in Definition \ref{TCP}.

\begin{defi}\label{twistpair}

A \emph{twisting pair} is a pair of functions $t \colon \locm''_1 \to \Aut(\locm')$ and $\eta \colon \locm''_2 \to N(\locm')$, satisfying the following conditions:
\begin{enumerate}[(a)]

\item $t(1) = \Id$ and $\eta(1,g) = [1] = \eta(g,1)$ for all $[g] \in \locm''_1$; and

\item (the cocycle formula) for each $[g|h|k] \in \locm''_3$,
\begin{equation}\label{coboundary2}
\Psi_{g}(\eta(h,k)) \cdot \eta(g, hk) = \eta(g, h) \cdot \eta(gh,k).
\end{equation}

\end{enumerate}

\end{defi}

\begin{lmm}\label{twistingn}

A twisting pair $(t, \eta)$ determines a a twisting function, $\{\phi_n \colon \locm''_n \to \underline{\aut}_{n-1}(\locm') = N_{n-1} \autcat(\locm')\}_{n \geq 1}$ with
$$
\begin{array}{l}
\phi_1 \colon [g] \in \locm''_1 \longmapsto t(g) \in \Aut(\locm') = \underline{\aut}_0(\locm')\\
\phi_2 \colon [g|h] \in \locm''_2 \longmapsto \big(t(g) \Left3{\eta(g,h)} t(gh) \circ t(h)^{-1} \big) \in \Mor(\autcat(\locm')) = \underline{\aut}_1(\locm')\\
\end{array}
$$
Moreover, every twisting function $\{\phi_n \colon \locm''_n \to \underline{\aut}_{n-1}(\locm')\}_{n \geq 1}$ is of this form.

\end{lmm}

\begin{proof}

Let $(t, \eta)$ be a twisting pair. For each $n \geq 1$ and each $[g_1|\ldots|g_n] \in \locm''$, define
$$
\phi_n[g_1|\ldots|g_n] = \big(\alpha_0 \Left2{\eta_1} \ldots \Left2{\eta_{n-1}} \alpha_{n-1} \big)
$$
as follows. Set $\alpha_0 = t(g_1)$, $\eta_1 = \eta(g_1, g_2)$, and set
\begin{itemize}

\item[(T1)] $\alpha_j = t(g_1\ldots g_j) \circ t(g_2 \ldots g_j)^{-1}$, for $j = 2, \ldots, n-1$; and

\item[(T2)] $\eta_k = \eta(g_1, g_2 \ldots g_k)^{-1} \cdot \eta(g_1, g_2 \ldots g_{k+1})$, for $k = 2, \ldots n-1$.

\end{itemize}
We claim that these formulas define a twisting function. Fix $\omega = [g_1|\ldots|g_n] \in \locm''_n$, with,
$$
\phi_n(\omega) = \big(\alpha_0 \Left2{\eta_1} \ldots \Left2{\eta_{n-1}} \alpha_{n-1} \big)
$$
as above. We check conditions ($\tau_1$)-($\tau_4$) of twisting functions.

\textbf{Condition ($\tau_1$).} We have to check that $\phi_{n-1}(d_i(\omega)) = d_{i-1}(\phi_n(\omega))$. Note that
$$
\begin{aligned}
d_{i-1}(\phi_n(\omega)) & = d_{i-1}\big(\alpha_0 \Left2{\eta_1} \ldots \Left2{\eta_{n-1}} \alpha_{n-1} \big) = \\
 & = \big(\alpha_0 \Left2{\eta_1} \ldots \Left2{} \alpha_{i-2} \Left3{\eta_{i-1} \eta_i} \alpha_i \Left2{} \ldots \Left2{\eta_{n-1}} \alpha_{n-1} \big).
\end{aligned}
$$
We may distinguish two cases, namely $i = 2$ and $i \geq 3$. When $i = 2$, we have $\eta_1 \eta_2 = \eta(g_1, g_2g_3)$, and formulas (T1) and (T2) imply that
$$
\begin{aligned}
\big(\alpha_0 \Left3{\eta_1 \eta_2} \alpha_2 \Left2{\eta_3} \ldots \Left2{\eta_{n-1}} \alpha_{n-1}\big) & = d_1\big(\alpha_0 \Left2{\eta_1} \ldots \Left2{\eta_{n-1}} \alpha_{n-1} \big) = \\
 & = \phi_{n-1}[g_1g_2|g_3|\ldots|g_n].
\end{aligned}
$$
When $i \geq 3$, we have $\eta_{i-1} \eta_i = \eta(g_1, g_2\ldots g_{i-1})^{-1} \eta(g_1, g_2\ldots g_{i+1})$, and condition ($\tau_1$) follows easily again from the formulas (T1) and (T2).

\textbf{Condition ($\tau_2$).} We check the case $n = 3$, with the general case being an obvious generalization. Given $[g|h|k] \in \locm''_3$, we have by definition
$$
\phi_3[g|h|k] = \big(t(g) \Left3{\eta(g,h)} t(gh) \circ t(h)^{-1} \Left9{\eta(g,h)^{-1} \eta(g, hk)} t(ghk) \circ t(hk)^{-1}\big).
$$
We have to show that $\phi_1[gh|k] = d_0(\phi_3[g|h|k]) \otimes \phi_1[h|k]$. We have
\begin{itemize}

\item $\phi_1[gh|k] = (t(gh) \Left3{\eta(gh, k)} t(ghk) \circ t(k)^{-1})$;

\item $d_0(\phi_3[g|h|k]) = (t(gh) \circ t(h)^{-1} \Left9{\eta(g,h)^{-1} \eta(g, hk)} t(ghk) \circ t(hk)^{-1})$; and

\item $\phi_1[h|k] = (t(h) \Left3{\eta(h,k)} t(hk) \circ t(k)^{-1})$.

\end{itemize}
Applying (\ref{productdim1}), we get
$$
d_0(\phi_3[g|h|k]) \otimes \phi_1[h|k] = (t(gh) \Left3{x} t(ghk) \circ t(k)^{-1}),
$$
where $x = (t(gh) \circ t(h)^{-1})(\eta(h,k)) \cdot \eta(g,h)^{-1} \cdot \eta(g, hk)$. To check that $x = \eta(gh, k)$, recall that $\eta(g,h)$ determines a homotopy $(t(g) \Left2{\eta(g,h)} t(gh) \circ t(h)^{-1})$, which implies that
$$
(t(gh) \circ t(h)^{-1})(y) = \eta(g,h)^{-1} \cdot t(g)(y) \cdot \eta(g,h)
$$
for all $y \in \locm'_1$. Thus, we have
$$
\begin{aligned}
X & = (t(gh) \circ t(h)^{-1})(\eta(h,k)) \cdot \eta(g,h)^{-1} \cdot \eta(g, hk) = \\
 & = \eta(g,h)^{-1} \cdot t(g)(\eta(h,k)) \cdot \eta(g,h) \cdot \eta(g,h)^{-1} \cdot \eta(g,hk) = \\
 & = \eta(g,h)^{-1} \cdot t(g)(\eta(h,k)) \cdot \eta(g,hk) = \eta(gh, k)
\end{aligned}
$$
where the last equality follows from property (b) of twisting pairs.

\textbf{Conditions ($\tau_3$) and ($\tau_4$).} These two conditions follow easily from the defining formulas (T1) and (T2), together with property (a) of twisting pairs.

To finish the proof we have to show that every twisting function is defined by formulas (T1) and (T2) for some twisting pair $(t, \eta)$. Thus, let $\{\phi_n \colon \locm''_n \to \underline{\aut}_{n-1}(\locm')\}_{n \geq 1}$ be a twisting function, and define functions
$$
t \colon \locm''_1 \Right2{} \Aut(\locm') \qquad \mbox{and} \qquad \eta \colon \locm''_2 \Right2{} N(\locm')
$$
as follows:
\begin{itemize}

\item $t(g) = \phi_1[g]$ for each $[g] \in \locm''_1$; and

\item $\eta(g,h)$ is given by $\phi_2[g|h] = (\alpha \Left2{\eta(g,h)} \beta)$ for each $[g|h] \in \locm''_2$.

\end{itemize}
We have to check that these two functions satisfy conditions (a) and (b) of twisting pairs, and that the original twisting function is determined by the pair $(t, \eta)$ according to formulas (T1) and (T2).

\textbf{Step 1.} The twisting function for $n = 2$ and property (a). Let $[g|h] \in \locm''_2$, and set
$$
\phi_2[g|h] = (\alpha \Left3{\eta(g,h)} \beta).
$$
We seek to determine $\alpha$ and $\beta$ in terms of the function $t$ and the simplex $[g|h]$. To do so, apply properties ($\tau_1$) and ($\tau_2$) of twisting functions to get
$$
\phi_1(d_0[g|h]) = \phi_1[h] = t(h) \qquad \phi_1(d_1[g|h]) = \phi_1[gh] = t(gh) \qquad \phi_1(d_2[g|h]) = \phi_1[g] = t(g).
$$
This way, $\alpha = d_1(\alpha \Left3{\eta(g,h)} \beta) = t(g)$ and $\beta = d_0(\alpha \Left3{\eta(g,h)} \beta) = t(gh) \circ t(h)^{-1}$, and
$$
\phi_2[g|h] = (t(g) \Left3{\eta(g,h)} t(gh) \circ t(h)^{-1}).
$$
Furthermore, by properties ($\tau_3$) and ($\tau_4$) of twisting functions, we have
\begin{align*}
\phi_2[g|1] = \phi_2(s_1[g]) = & s_0(\phi_1[g]) = s_0(t(g)) = (t(g) \Left2{1} t(g)) \\
\phi_2[1|g] = \phi_2(s_0[g]) = & (\Id \Left2{1} \Id)
\end{align*}
Thus, $\eta(g,1) = 1 = \eta(1,g)$. To see that $t(1) = \Id$, it is enough to apply property ($\tau2$) of twisting functions to the simplex $[1|1] \in \locm''_2$.

\textbf{Step 2.} Property (b) of twisting pairs and the formulae (T1) and (T2). To check that property (b) holds, we have to show that
$$
t(g)(\eta(h,k)) \cdot \eta(g,hk) = \eta(g,h) \cdot \eta(gh, k)
$$
for all $[g|h|k] in \locm''_3$. Fix elements $[g|h|k] in \locm''_3$, and set $\phi_3[g|h|k] = (\alpha_0 \Left2{\eta_1} \alpha_1 \Left2{\eta_2} \alpha_2)$. Applying property ($\tau_1$) with $i = 3, 2$ respectively, we get
\begin{align*}
(t(g) \Left3{\eta(g,h)} t(gh) \circ t(h)^{-1}) & = (\alpha_0 \Left2{\eta_1} \alpha_1) \\
(t(g) \Left3{\eta(g, hk)} t(ghk) \circ t(hk)^{-1}) & = (\alpha_0 \Left2{\eta_1 \eta_2} \alpha_2)\\
\end{align*}
Thus, $\phi_3[g|h|k] = (t(g) \Left3{\eta(g,h)} t(gh) \circ t(h)^{-1} \Left9{\eta(g,h)^{-1} \eta(g,hk)} t(ghk) \circ t(hk)^{-1})$.

On the other hand, applying property ($\tau_2$) of twisting functions, together with (\ref{productdim1}), we get the following.
\begin{align*}
(\alpha_1 \Left2{\eta_2} \alpha_2) & = d_0(\phi_3[g|h|k]) = \phi_2[gh|k] \otimes \phi_2[h|k]^{-1} = \\
 & = (t(gh) \Left3{\eta(gh,k)} t(ghk) \circ t(k)^{-1}) \otimes (t(h) \Left3{\eta(h,k)} t(hk) \circ t(k)^{-1})^{-1} = \\
 & = (t(gh) \circ t(h)^{-1} \,\, \Left9{(t(gh) \circ t(h)^{-1})(\eta(h,k)^{-1}) \eta(gh,k)} \,\, t(ghk) \circ t(hk)^{-1}).\\
\end{align*}
Note that $\eta(g,h)$ determines a homotopy $(t(g) \Left2{\eta(g,h)} t(gh) \circ t(h)^{-1})$, and in particular this means that
$$
(t(gh) \circ t(h)^{-1})(y) = \eta(g,h)^{-1} \cdot t(g)(y) \cdot \eta(g,h)
$$
for all $y \in \locm$. Combining this with the above description of $\phi_3[g|h|k]$, we deduce the following.
\begin{align*}
\eta(g,h)^{-1} \cdot \eta(g,hk) & = \eta_2 = (t(gh) \circ t(h)^{-1})(\eta(h,k)^{-1}) \eta(gh,k) = \\
 & = \eta(g,h)^{-1} \cdot t(g)(\eta(h,k)^{-1}) \cdot \eta(g,h) \cdot \eta(gh,k).
\end{align*}
Property (b) of twisting pairs follows immediately.

To finish the proof we have to show that $\phi_n[g_1|\ldots|g_n]$ is given by the formulae (T1) and (T2) above, and this is actually easily checked inductively by applying property ($\tau1$) with $i = n-1, n$ to $\phi_n[g_1|\ldots|g_n]$.
\end{proof}

\begin{rmk}\label{etagh}

Notice that $\eta(g_1,g_2) \in N(\locm)$ determines a homotopy from $\Psi_{g_1 g_2} \circ \Psi_{g_2}^{-1}$ to $\Psi_{g_1}$. By Lemma \ref{homotopies}, this implies the following equality for all $[x] \in \locm'$
$$
[(\Psi_{g_1} \circ \Psi_{g_2})(x)] = [\eta(g_1,g_2) \cdot \Psi_{g_1 g_2}(x) \cdot \eta(g_1,g_2)^{-1}]
$$
This useful interpretation of the elements $\eta(g_1, g_2)$ will be tacitly applied from now on.

\end{rmk}

The simplex $\phi_n[g_1|\ldots|g_n]^{-1}$ plays a crucial role in the simplicial structure of a TCP and it is interesting to have an explicit formula for it. In the notation above we have
\begin{equation}\label{twistingninv}
\big(\phi_n[g_1|\ldots|g_n]\big)^{-1} = (\varphi_1^{-1} \Left6{\omega_{1,2}} \varphi_2^{-1} \Left6{\omega_{2,3}} \ldots \Left6{\omega_{n-1,n}} \varphi_n^{-1}),
\end{equation}
where $\omega_{1,2} = \varphi_1^{-1}(\eta(g_1,g_2)^{-1})$ and $\omega_{k, k+1} = \varphi_k^{-1}(\eta(g_1\ldots g_k, g_{k+1})^{-1}) \cdot \eta(g_2 \ldots g_k, g_{k+1})$ for every $k = 2, \ldots, n-1$.

\begin{thm}\label{extension1}

The category of partial groups is closed by fibre bundles.

\end{thm}

\begin{proof}

We have to show that every fibre bundle of partial groups is again a partial group. By Theorem \ref{isocat} and by \cite[Proposition IV.3.2]{BGM}, it is enough to show that every RTCP of partial groups is a reduced $N$-simplicial set with inversion. Let $\locm'$ and $\locm''$ be partial groups, and let $X = (\locm' \times_{\phi} \locm'')$ be a RTCP.

Note first that, since $\locm'$ and $\locm''$ are reduced, then so is $\locm' \times_{\phi} \locm''$. Next let us check that the enumerating operator on $\locm' \times_{\phi} \locm''$ is injective for all $n$. Let
$$
\omega_n = \big([y_1|\ldots|y_n],[g_1|\ldots|g_n] \big),
$$
where $[y_1|\ldots|y_n] \in \locm'_n$ and $[g_1|\ldots|g_n] \in \locm''_n$, and recall that the enumerating operator $E_n$ is defined recursively as $E_1 = \Id$ and
$$
E_n \colon X_n \Right8{(F_1, d_0)} X_1 \times X_{n-1} \Right8{\Id \times E_{n-1}} \underbrace{X_1 \times \ldots \times X_1}_{\mbox{$n$ times}},
$$
where $F_1 = d_2 \circ \ldots \circ d_n$. For $n = 1$, the enumerating operator is the identity, so there is nothing to check. For $n = 2$, we have
$$
E_2(w_2) = \Big(([y_1],[g_1]), \big([\Psi_{g_1}^{-1}(y_2 \cdot \eta(g_1,g_2)^{-1})],[g_2]\big) \Big).
$$

Describing $E_3$ will be enough to deduce the general case. Using Lemma \ref{twistingn} and the formula (\ref{twistingninv}), we have the following expression of $\phi_3[g_1|g_2|g_3]^{-1}$:
$$
\xymatrix{
(\Psi_{g_1}^{-1} & & & \Psi_{g_2} \circ \Psi_{g_1g_2}^{-1} \ar[lll]_{\Psi_{g_1}^{-1}(\eta(g_1,g_2)^{-1})} & & & & & \Psi_{g_2g_3} \circ \Psi_{g_1g_2g_3}^{-1}), \ar[lllll]_{(\Psi_{g_2} \circ \Psi_{g_1g_2}^{-1})(\eta(g_1g_2,g_3)^{-1}) \cdot \eta(g_2,g_3)}
}
$$
from where we deduce the following, using formula (\ref{actionn}):
$$
\phi_3[g_1|g_2|g_3]^{-1} \cdot [y_2|y_3] = [\Psi_{g_1}^{-1}(y_2 \cdot \eta(g_1,g_2)^{-1})|(\Psi_{g_2} \circ \Psi_{g_1g_2}^{-1})(y_3 \cdot \eta(g_1g_2,g_3)^{-1}) \cdot \eta(g_2,g_3)].
$$
Thus, given $\omega_3 = ([y_1|y_2|y_3], [g_1|g_2|g_3]) \in X_3$, the enumerating operator $E_3$ has the following effect:
\small
$$
\begin{aligned}
E(\omega_3) & = \big((\Id \times E_2) \circ (F_1, d_0)\big)(\omega_3) \\[4pt]
 & = (\Id \times E_2)\bigg(\big([y_1], [g_1]\big), \big([\Psi_{g_1}^{-1}(y_2 \cdot \eta(g_1,g_2)^{-1})|(\Psi_{g_2} \circ \Psi_{g_1g_2}^{-1})(y_3 \cdot \eta(g_1g_2,g_3)^{-1}) \cdot \eta(g_2,g_3)], [g_2|g_3]\big) \bigg) \\[4pt]
 & = \Big(([y_1],[g_1])\, , \, ([\Psi_{g_1}^{-1}(y_2 \cdot \eta(g_1,g_2)^{-1})],[g_2])\, , \, ([\Psi_{g_1g_2}^{-1}(y_3 \cdot \eta(g_1g_2,g_3)^{-1})],[g_3]) \Big).
\end{aligned}
$$

\normalsize

In general, the enumerating operator $E_n$ is given by the formula
$$
E_n([y_1|\ldots|y_n],[g_1|\ldots|g_n]) = \Big(([z_1],[g_1]), ([z_2],[g_2]), \ldots, ([z_n],[g_n]) \Big),
$$
where $z_1 = y_1$ and, for $k = 2, \ldots, n$, $z_k = \Psi_{g_1 \ldots g_{k-1}}^{-1}(y_k \cdot \eta(g_1 \ldots g_{k-1}, g_k)^{-1})$. It is clear now that $E_n$ is injective on $\locm' \times_{\phi} \locm''$, since multiplication on the right by $\eta(g_1 \ldots g_{k-1}, g_k)^{-1}$ is a free action, and $\Psi_{g_1 \ldots g_{k-1}}^{-1}$ is an automorphism of partial groups.

Finally, we have to check that $\locm' \times_{\phi} \locm''$ has an inversion. This follows immediately since the face operator $d_0$ is not involved in the definition of the product operator $\Pi$, and hence we can use the inversions on $\locm'$ and $\locm''$ to define an inversion on $\locm' \times_{\phi} \locm''$.
\end{proof}

Let now $\locm', \locm''$ be partial groups, and let $\Gamma \leq \underline{\aut}(\locm')$. Let also $\locm' \times_{\phi} \locm''$ be a $\Gamma$-RTCP. We can define the associated \emph{outer action} of $\locm' \times_{\phi} \locm''$ by
\begin{equation}\label{oaction}
\xymatrix@R=1mm@C=15mm{
\varepsilon \colon \locm'' \ar[r] & \pi_0(\Gamma) \ar[r]^{\incl_{\ast}} & \Out(\locm')\\
[g] \ar@{|->}[rr] & & \overline{\phi[g]}
}
\end{equation}

\begin{lmm}

The map $\varepsilon$ in (\ref{oaction}) is a morphism of partial groups.

\end{lmm}

\begin{proof}

This is immediate by definition.
\end{proof}

\begin{defi}\label{extBI}

Let $\locm', \locm''$ be partial groups and let $\Gamma \leq \underline{\aut}(\locm')$.
\begin{enumerate}[(i)]

\item An \emph{outer action} of $\locm''$ on $\locm'$ is a morphism of partial groups $\varepsilon \colon \locm'' \to \Out(\locm')$.

\item Given a $\Gamma$-RTCP $\locm' \times_{\phi} \locm''$, the \emph{associated outer action} is the morphism of partial groups $\varepsilon \colon \locm'' \to \Out(\locm')$ defined in (\ref{oaction}).

\item Let $\varepsilon \colon \locm'' \to \Out(\locm')$ be an outer action that satisfies
$$
\Im(\varepsilon) \subseteq \incl_{\ast}(\pi_0(\Gamma)) \subseteq \pi_0(\underline{\aut}(\locm')) = \Out(\locm').
$$
A \emph{$\Gamma$-extension of $\locm''$ by $\locm'$ with action $\varepsilon$}, or simply a $\Gamma$-extension of $(\locm', \locm'', \varepsilon)$, is a $\Gamma$-RTCP $\locm' \times_{\phi} \locm''$ whose associated outer action is $\varepsilon$.

\end{enumerate}

\end{defi}

Let $X = \locm' \times_{\phi} \locm''$ be a $\Gamma$-extension of $(\locm', \locm'', \varepsilon)$. The partial group structure of $X$ is encrypted in the description of $X$ as an RTCP. For this reason, we now introduce a simplicial set $\locm$ which is strongly equivalent to $X$ and where the partial group structure is evident.

 \begin{prop}\label{TCP1}

Let $X = \locm' \times_{\phi} \locm''$ be a $\Gamma$-extension of $(\locm', \locm'', \varepsilon)$, where the twisting function $\phi$ is determined by
$$
\big\{\phi[g] = \Psi_g \in \Aut(\locm')\big\}_{[g] \in \locm''} \qquad \mbox{and} \qquad \big\{\phi[g_1|g_2] = (\Psi_{g_1} \Left3{\eta(g_1,g_2)} \Psi_{g_1g_2} \circ \Psi_{g_2}^{-1})\big\}_{[g_1|g_2] \in \locm''},
$$
and by Lemma \ref{twistingn}. Let also $\locm = \{[(v_{\locm'}, v_{\locm''})]\} \coprod_{n \geq 1} \locm_n$, where, for each $n \geq 0$ the set $\locm_n$ is the collection of symbols
$$
[(x_1,g_1)|(x_2,g_2)|\ldots|(x_n,g_n)] \in (\locm'_1 \times \locm''_1)^n,
$$
subject to the following conditions
\begin{enumerate}[(a)]

\item $[x_1| \Psi_{g_1}(x_2) \cdot \eta(g_1,g_2)| \Psi_{g_1g_2}(x_3) \cdot \eta(g_1g_2,g_3)|\ldots |\Psi_{g_1\ldots g_{n-1}} (x_n) \cdot \eta(g_1\ldots g_{n-1},g_n)] \in \locm'_n$;\\[1pt]

\item $[g_1|\ldots|g_n] \in \locm''_n$.

\end{enumerate}
Set also $s_0[(v_{\locm'}, v_{\locm''})] = [(1,1)]$ and, for each $n \geq 1$ and each $[(x_1,g_1)|\ldots|(x_n,g_n)] \in \locm_n$, set
$$
d_i(w_n) = \left\{
\begin{array}{ll}
[(x_2,g_2)|\ldots|(x_n,g_n)] & i = 0 \\[2pt]
[(x_1,g_1)|\ldots |(x_i \cdot \Psi_{g_i}(x_{i+1}) \cdot \eta(g_i,g_{i+1}),g_i g_{i+1})| \ldots |(x_n,g_n)] & 1 \leq i \leq n-1 \\[2pt]
[(x_1,g_1)|\ldots|(x_{n-1},g_{n-1})] & i = n
\end{array}
\right.
$$
$$
s_i[(x_1,g_1)|(x_2,g_2)|\ldots|(x_n,g_n)] = [(x_1,g_1)|\ldots|(x_i,g_i)|(1,1)|(x_{i+1}, g_{i+1})|\ldots |(x_n,g_n)]. \phantom{AAA}
$$
Finally, let $\tau \colon \locm \to \locm''$ be defined by $[(x_1,g_1)|(x_2,g_2)|\ldots|(x_n,g_n)] \mapsto [g_1|g_2|\ldots|g_n]$. Then, the following holds.
\begin{enumerate}[(i)]

\item $\locm$ is a simplicial set.

\item $\tau \colon \locm \to \locm''$ is a $\Gamma$-bundle.

\item $\locm' \times_{\phi} \locm''$ is strongly equivalent to $\tau \colon \locm \to \locm''$.

\end{enumerate}
In particular, $\locm$ is a partial group.

\end{prop}

\begin{proof}

For each $n$ let $\alpha_n \colon \locm_n \to (\locm' \times_{\phi} \locm'')_n = \locm' \times \locm''$ be the map defined by
$$
\begin{aligned}
\alpha_n[(x_1,g_1)| & \ldots |(x_n,g_n)] =\\
 & = \big([x_1| \Psi_{g_1}(x_2) \cdot \eta(g_1,g_2)| \ldots |\Psi_{g_1\ldots g_{n-1}} (x_n) \cdot \eta(g_1\ldots g_{n-1},g_n)] \, , \, [g_1|\ldots|g_n] \big).
\end{aligned}
$$
The maps $\alpha_n$ are clearly bijective for all $n$, and furthermore they are defined to commute with the simplicial operators on $\locm' \times_{\phi} \locm''$ and the operators $d_i$, $s_i$ defined above. This makes $\locm$ into a simplicial set. It is also clear that the map $\tau$is a $\Gamma$-fibre bundle, and it is strongly equivalent to $\locm' \times_{\phi} \locm''$ by construction. Since $X$ is a partial group, then so is $\locm$.
\end{proof}

The following is an alternative condition for (i) above. It will be useful in later sections.

\begin{lmm}\label{alternativecond}

Let $\locm', \locm''$ be partial groups and let $\Gamma \leq \underline{\aut}(\locm')$ be a subgroup. Let also $\locm' \times_{\phi} \locm''$ be a $\Gamma$-TCP, and let $[(x_1,g_1)], \ldots, [(x_n, g_n)] \in \locm'_1 \times \locm''_1$ be such that $[g_1|\ldots |g_n] \in \locm''_n$. Then, the following statements are equivalent.
\begin{itemize}

\item[(i)] $[x_1| \Psi_{g_1}(x_2) \cdot \eta(g_1,g_2)| \Psi_{g_1g_2}(x_3) \cdot \eta(g_1g_2,g_3)|\ldots |\Psi_{g_1\ldots g_{n-1}} (x_n) \cdot \eta(g_1\ldots g_{n-1},g_n)] \in \locm'_n$.\\[2pt]

\item[(i')] $[x_1|\Psi_{g_1}(x_2)| (\Psi_{g_1} \circ \Psi_{g_2})(x_3)|\ldots |(\Psi_{g_1} \circ \ldots \Psi_{g_{n-1}})(x_n)] \in \locm'_n$.

\end{itemize}

\end{lmm}

\begin{proof}

This follows since $[\eta(g_1 \ldots g_k, g_{k+1})] \in N(\locm)$ for all $k = 1, \ldots, n-1$ and because, by definition, $[\eta(g_1\ldots g_k, g_{k+1})]$ defines a homotopy from $\Psi_{g_1 \ldots g_{k+1}} \circ \Psi_{g_{k+1}}^{-1}$ to $\Psi_{g_1\ldots g_k}$.
\end{proof}

\begin{rmk}\label{represext}

Let $\locm' \times_{\phi} \locm''$ and $\tau \colon \locm \to \locm''$ be as defined in Proposition \ref{TCP1}. The obvious inclusion of $\locm'$ into $\locm$, defined by $[x] \mapsto [(x,1)]$, is actually a map of simplicial sets, namely $\iota \colon \locm' \to \locm$. We will represent the extension $\locm' \times_{\phi} \locm''$ by
$$
\locm' \Right4{\iota} \locm \Right4{\tau} \locm'',
$$
with multiplication and inversion rules in $\locm$ given by the formulas
\begin{equation}\label{prodext}
\begin{array}{l}
\Pi[(x_1,g_1)|(x_2,g_2)] = [(x_1 \cdot \Psi_{g_1}(x_2) \cdot \eta(g_1, g_2), g_1 \cdot g_2)] \\[2pt] \relax
[(x,g)]^{-1} = [(\eta(g^{-1},g)^{-1} \cdot \Psi_{g^{-1}}(x^{-1}), g^{-1})].
\end{array}
\end{equation}
These generalize the usual multiplication and inversion formulae in group extensions, see \cite[Chapter IV]{MacLane} or \cite[Chapter IV]{Brown} for further details.

\end{rmk}

Finally, we describe the cohomological obstructions to the existence and uniqueness of extensions of partial groups (given some initial data). Unsurprisingly, the result generalizes the existing obstructions for existence and uniqueness of extensions of finite groups.

Let us be more precise. Let $\locm', \locm''$ be partial groups, and let $\varepsilon \colon \locm'' \to \overline{\Gamma} \leq \Out(\locm')$ be an outer action. Let also $\Gamma^{\ast} \leq \underline{\aut}(\locm')$ be the pre-image of $\overline{\Gamma}$, so $\pi_0(\Gamma^{\ast}) = \overline{\Gamma}$. With this initial data, we define obstructions to the existence and uniqueness of $\Gamma^{\ast}$-extensions of $(\locm', \locm'', \varepsilon)$.

Such obstructions are described as cohomology classes of the space $\locm''$ with \emph{local coefficients} in $Z(\locm')$. The precise description of the cohomology groups in Theorem \ref{classext} below is in \cite{DK2}, and the reader is referred to this source for further details.

\begin{thm}\label{classext}

Let $\locm', \locm''$ be partial groups, and let $\varepsilon\colon \locm'' \to \overline{\Gamma} \leq \Out(\locm')$ be an outer action. Then, the following holds.
\begin{enumerate}[(i)]

\item There is an obstruction class $[\kappa] \in H^3(\locm'';Z(\locm'))$ to the existence of $\Gamma^{\ast}$-extensions of $(\locm', \locm'', \varepsilon)$: such extensions exist if and only if the class $[\kappa] = 0$.

\item If there is any, the set of isomorphism classes of $\Gamma^{\ast}$-extensions of $(\locm', \locm'', \varepsilon)$ is in one-to-one correspondence with the set $H^2(\locm'';Z(\locm'))$.

\end{enumerate}

\end{thm}

\begin{proof}

By \cite[\S 5]{BGM}, existence and uniqueness of $\Gamma^{\ast}$-extensions of $(\locm', \locm'', \varepsilon)$ correspond to existence and uniqueness of liftings of the induced map $\locm'' \to B\overline{\Gamma} \leq B\Out(\locm')$ to $B\Gamma^{\ast} \subseteq B\underline{\aut}(\locm')$: each lifting
$$
\widetilde{\varepsilon} \colon \locm'' \Right3{} B \Gamma^{\ast} \subseteq B \underline{\aut}(\locm')
$$
corresponds to an isomorphism class of such extensions. Now, by Theorem \ref{autX} and by definition of $\Gamma^{\ast}$, there is a fibration $B^2Z(\locm') \to B\Gamma^{\ast} \to B\Out(\locm')$, and the proof is finished by classical obstruction theory \cite{DK2}.
\end{proof}

\begin{rmk}

Let $(\locl', \Delta', S')$ be a locality and let $\overline{\Gamma} = \Out(\locl';S') \leq \Out(\locl')$. Then it follows from Lemma \ref{autloc0} that $\Gamma^{\ast} = \underline{\aut}(\locl';S')$.

\end{rmk}

Let us briefly discuss the classification of extensions of partial groups attained in Theorem \ref{classext}. The reader should keep in mind that we have not defined (and we shall not in this paper) modules of partial groups, which should be the first step towards developing homological algebra for partial groups.

The idea behind the proof of Theorem \ref{classext} is to construct the twisting function $\{\phi_n \colon \locm''_n \to \Gamma^{\ast} \leq \underline{\aut}_{n-1}(\locm')\}$ directly out of $\varepsilon$, by choosing liftings of the elements $\varepsilon[g]$ in $\Aut(\locm')$, and then proceeding to higher dimensions. Filling in the gaps in this idea would lead to ``group theoretical'' obstructions to the existence and uniqueness of extensions of partial groups, as done in \cite[Chapter IV]{MacLane} or \cite[Chapter IV]{Brown} for extensions of groups (more exactly, this way one would construct cocycles in the cochain complex for the simplicial set $\locm''$ with local coefficients in the abelian group $Z(\locm')$, which produce the desired obstructons). This is an outline of the process that the interested reader should follow.
\begin{enumerate}

\item Starting from $\varepsilon$, one may choose a representative $\Psi_g \in \Aut(\locm')$ for each $[g] \in \locm''_1$ (in particular, choose $\Psi_1 = \Id$).

\item The above choices are random, and so in general the equality $\Psi_g \circ \Psi_{h} = \Psi_{gh}$ does not hold. However, for each word $[g|h] \in \locm''_2$ the exact sequence of Theorem \ref{autX} (iii) implies the existence of some $\eta(g,h) \in N(\locm)$ such that
$$
\Psi_g \circ \Psi_{h} = c_{\eta(g,h)} \circ \Psi_{gh}.
$$
Since $\Psi_1 = \Id$, one may choose $\eta(g,1) = 1 = \eta(1,g)$ for all $[g] \in \locm''_1$.

\item Using the above choices, one may define $\locm$ as in Proposition \ref{TCP1} and a multiplication using the formula (\ref{prodext}). This done, one has to check that this multiplication is associative (associativity only needs to be checked on the elements of $\locl''_3$), and this produces the class $[\kappa]$ of Theorem \ref{classext} (i).

\item The vanishing of the class $[\kappa]$ in the previous step implies associativity of the product defined by the formula (\ref{prodext}). To check part (ii) of Theorem \ref{classext} from this group theoretical point of view, one just has to check that modifying the choices in steps (1) and (2) by a $2$-dimensional cocycle produces a different multiplication rule (and hence a different extension).

\end{enumerate}


\section{Extensions of localities}\label{isoext}

We focus our study on extensions of saturated localities, our goal being to give conditions for such an extension to give rise to a new saturated locality. In order to give a clear introduction to this section, let $(\locl', \Delta', S')$ and $(\locl'', \Delta'', S'')$ be (saturated) localities, and let $\varepsilon \colon \locl'' \to \Out(\locl')$ be an outer action. Let also
$$
\locl' \Right3{} X \Right3{\tau} \locl''
$$
be an extension of $(\locl', \locl'', \varepsilon)$. Ultimately, we want conditions for $X$ to be equivalent to the classifying space of a $p$-local finite group (up to $p$-completion). Given such an extension, one should not expect $X$ to be a locality in general, although we show some examples where this is the case.


\subsection{Isotypical extension of localities}

In order to achieve our goals for this section, we start with a rather more general setup, in which we do not require any locality to be saturated. This way we are able to prove some general results which we later specialize to saturated localities. Intuitively, it is necessary to put some conditions on such an extension for the locality structures to play any role.

\begin{defi}

Let $(\locl', \Delta', S')$ and $(\locl'', \Delta'', S'')$ be localities. An extension of $\locl''$ by $\locl'$ is \emph{isotypical} if the following conditions hold:
\begin{enumerate}[(i)]

\item the action $\varepsilon \colon \locl'' \to \Out(\locl')$ factors through $\Out(\locl';S') \leq \Out(\locl')$; and

\item $\Delta'$ is $\Aut(\locl';S')$-invariant, i.e. $\Psi(P) \in \Delta'$ for all $P \in \Delta'$ and all $\Psi \in \Aut(\locl';S')$.

\end{enumerate}

\end{defi}

Thus, roughly speaking an isotypical extension of $(\locl'', \Delta'', S'')$ by $(\locl', \Delta', S')$ is an $\underline{\aut}(\locl';S')$-bundle over $\locl''$, and with fibre $\locl'$.

\begin{hyp}\label{hyp1}

Fix an isotypical extension $\locl' \Right1{} \locl \Right1{\tau} \locl''$, where
\begin{enumerate}[(a)]

\item $(\locl', \Delta', S')$ is a locality, with associated fusion system $\FF' = \FF_{\Delta'}(\locl')$.

\item $(\locl'', \Delta'', S'')$ is a locality, with associated fusion system $\FF'' = \FF_{\Delta''}(\locl'')$.

\end{enumerate}
The above extension is determined by some twisting function $\{\phi_n \colon \locl''_n \to (\underline{\aut})_{n-1}(\locl';S')\}$, which in turn is determined by the data
$$
\{\Psi_g \in \Aut(\locl';S') \,\, \big| \,\, [g] \in \locl''_1\} \qquad \mbox{and} \qquad \{[\eta(g,h)] \in N_{\locl'}(S') \,\, \big| \,\, [g|h] \in \locl''_2\},
$$
satisfying the following conditions (see Lemma \ref{twistingn})
\begin{enumerate}[(1)]

\item $\overline{\Psi_g} = \varepsilon[g]$ for all $[g] \in \locl''_1$;

\item $\Psi_1 = \Id$ and $\eta(1,h) = [(1,1)] = \eta(g,1)$ for all $[g], [h] \in \locl''_1$;

\item $[\eta(g,h) \cdot \Psi_{gh}(x) \cdot \eta(g,h)^{-1}] = [(\Psi_g \circ \Psi_h)(x)]$ for all $[g|h] \in \locl''_2$ and all $[x] \in \locl'_1$; and

\item $[\Psi_g(\eta(h,k)) \cdot \eta(g, hk)] = [\eta(g,h) \cdot \eta(gh, k)]$ for all $[g|h|k] \in \locl''_3$.

\end{enumerate}
We will consider these choices fixed, in case we have to perform explicit calculations with elements in $\locl$. Notice that the fusion systems $\FF'$ and $\FF''$ are not assumed to be saturated.

\end{hyp}

\begin{lmm}\label{propext1}

The partial group $\locl$ contains a $p$-subgroup $S \leq \locl$ which makes the following diagram of extensions commutative
$$
\xymatrix{
BS' \ar[rr]^{\iota} \ar[d]_{\incl} & & BS \ar[rr]^{\tau} \ar[d]_{\incl} & & BS'' \ar[d]^{\incl} \\
\locl' \ar[rr]_{\iota} & & \locl \ar[rr]_{\tau} & & \locl''
}
$$
In particular, $S$ is maximal in the poset of $p$-subgroups of $\locl$.

\end{lmm}

\begin{proof}

Write $N' = N_{\locl'}(S')$ and $N'' = N_{\locl''}(S'')$ for short, which are finite groups since $S' \in \Delta'$ and $S'' \in \Delta''$, and let $N = \{[(x,g)] \in \locl \, | \, [x] \in N' \mbox{ and } [g] \in N''\}$. We claim that $N$ is a (finite) subgroup of $\locl$. To prove the claim we have to check that $[(x_1,g_1)|\ldots|(x_n,g_n)] \in \locl$ for every sequence $[(x_1,g_1)], \ldots, [(x_n,g_n)] \in N$.

Since the subgroup $N' \leq \locl'$ is invariant under the action of $\Aut(\locl';S')$, it follows that $[x_1|\Psi_{g_1}(x_2)|\ldots|(\Psi_{g_1} \circ \ldots \circ \Psi_{g_{n-1}})(x_n)] \in BN' \leq \locl'$. Also, $[g_1|\ldots|g_n] \in BN'' \leq \locl''$, since $N''$ is a group. By Lemma \ref{alternativecond} this means that $[(x_1,g_1)|\ldots|(x_n,g_n)] \in \locl$, and thus $N'$ is a group.

It is clear now that $N$ is an extension of $N''$ by $N'$, and there is a commutative diagram of extensions 
$$
\xymatrix{
BN' \ar[r] \ar[d] & BN \ar[r] \ar[d] & BN'' \ar[d] \\
\locl' \ar[r] & \locl \ar[r] & \locl'' \\
}
$$
In particular, $N$ has Sylow $p$-subgroups, and we may choose $S \in \Syl_p(N)$ completing the diagram in the statement.
\end{proof}

Fix a choice of a subgroup $S \leq \locl$ satisfying the properties described in Lemma \ref{propext1}. For a subgroup $P \leq S$, we use the following notation: $P' \defin P \cap S' \leq S'$ and $P'' \defin \tau(P) \leq S''$. Fix also the collection
\begin{equation}\label{collectionS}
\Delta = \{P \leq S \,\, \big| \,\, P' = P \cap S' \in \Delta' \mbox{ and } P'' = \tau(P) \in \Delta''\}.
\end{equation}
Notice that $\Delta$ depends heavily on the choices of $\Delta'$ and $\Delta''$ (and these choices are usually not unique!).

\begin{rmk}\label{prodHi}

Let $[(x,g)] \in \locl$, and let $L_{[(x,g)]} \leq S$ be the biggest subgroup of $S$ that is left conjugated by $[(x,g)]$ to a subgroup of $S$ (notice that we do not know whether $(\locl, \Delta, S)$ is a locality or not, and thus we cannot use the definition in (\ref{SwSu0})). This means that, for each $[(y,h)] \in L_{[(x,g)]}$, we have $u = [(x,g)|(y,h)|(x,g)^{-1}] \in \locl$ and $\Pi(u) \in S$. More specifically, we have
$$
\begin{aligned}
\Pi(u) & = \Pi\big[(x, g)|(y, h)|(\eta(g^{-1}, g)^{-1} \cdot \Psi_{g^{-1}}(x^{-1}), g^{-1})\big] = \\[4pt]
 & = \Pi\big[(x \cdot \Psi_{g}(y) \cdot \eta(g, h), g \cdot h)|(\eta(g^{-1}, g)^{-1} \cdot \Psi_{g^{-1}}(x^{-1}), g^{-1})\big] = \\[4pt]
 & = \big[(x \cdot \Psi_{g}(y) \cdot \eta(g, h) \cdot \Psi_{g h}(\eta(g^{-1}, g)^{-1}) \cdot \Psi_{g h}(\Psi_{g^{-1}}(x^{-1})) \cdot \eta(gh, g^{-1}), g \cdot h \cdot g^{-1}) \big] = \\[4pt]
 & = \big[(x \cdot \Psi_{g}(y) \cdot \eta(g, h) \cdot \eta(ghg^{-1}, g) \cdot (\9{ \eta(gh, g^{-1})^{-1}}((\Psi_{gh} \circ \Psi_{g^{-1}})(x^{-1}))), g \cdot h \cdot g^{-1})\big] = \\[4pt]
 & = \big[(x \cdot \Psi_{g}(y) \cdot \eta(g, h) \cdot \eta(ghg^{-1}, g) \cdot \Psi_{ghg^{-1}}(x^{-1}), g \cdot h \cdot g^{-1})\big],
\end{aligned}
$$
where the equality between lines three and four follows by an application of the cocycle formula (\ref{coboundary2}), and the equality between lines four and five follows from property (2) in Hypothesis \ref{hyp1}. In particular, if $[(y,h)] \in S'$, then $h = 1$ and the above formula implies that
$$
\Pi[(x, g)|(y, 1)|(x, g)^{-1}] = [(x \cdot \Psi_{g}(y) \cdot x^{-1}, 1)].
$$

\end{rmk}

We start by analyzing the triple $(\locl, \Delta, S)$. As we show below, this is not a locality in general, but it is rather close from being one. This requires the application of the formulae (\ref{prodext}), and the cocycle condition (\ref{coboundary2}), which we use without any further mention. Let $\ww(\locl_1)$ be the free monoid on $\locl_1$. To simplify the notation, we denote the words in $\ww(\locl_1)$ by $[(x_1, g_1)| \ldots| (x_n,g_n)]$. Let $\dd_{\Delta}$ be the collection of words $\omega = [(x_1, g_1)| \ldots| (x_n,g_n)] \in \ww(\locl_1)$ for which there exists $H_0, \ldots, H_n \in \Delta$ such that $\9{(x_i, g_i)}H_i = H_{i-1}$ for all $i = 1, \ldots, n$ (note that this is equivalent to the definition in \ref{opgroup}).

\begin{prop}\label{propext3}

The collection $\Delta$ defined in (\ref{collectionS}) satisfies the following properties:
\begin{enumerate}[(i)]

\item $S \in \Delta$;

\item $(\locl, \Delta)$ satisfies condition (O2) of objective partial groups;

\item there is an inclusion $\dd_{\Delta} \subseteq \locl$.

\end{enumerate}

\end{prop}

\begin{proof}

Property (i) is immediate by definition of $S$ and $\Delta$. To prove property (ii), note that $\Delta$ is clearly closed by overgroups, and we have to show that if $K \in \Delta$ and $[(x,g)] \in \locl$ are such that $\9{(x,g)}K = H \leq S$, then $H \in \Delta$. By definition of $\Delta$, we need to check that $H' \in \Delta'$ and $H'' \in \Delta''$. Note that $H'' = [g] \cdot K'' \cdot [g^{-1}] \leq S''$, and thus $H'' \in \Delta''$ since $(\locl'', \Delta'', S'')$ is a locality. Regarding $H'$, the conjugation formula in Remark \ref{prodHi} implies that
$$
H' = [x] \cdot \Psi_g(K') \cdot [x^{-1}].
$$
Since $K' \in \Delta'$ and every isotypical automorphism of $\locl'$ preserves $\Delta'$, it follows that $\Psi_g(K') \in \Delta'$, and it follows that $H' \in \Delta$ since $(\locl', \Delta', S')$ is a locality and $\Delta'$ is closed by conjugation.

To prove property (iii), let $\omega = [(x_1, g_1)| \ldots| (x_n,g_n)] \in \dd_{\Delta}$, via the sequence $H_0, \ldots, H_n \in \Delta$, so $\9{[(x_i,g_i)]}H_i = H_{i-1}$ for $i = 1, \ldots, n$. In particular, for each $i$ and each $[(y_i, h_i)] \in H_i$ we have $\omega_i = [(x_i, g_i)|(y_i, h_i)|(x_i, g_i)^{-1}] \in \locl$ and $\Pi(\omega_i) \in H_{i-1}$. By Lemma \ref{alternativecond}, in order to show that $\omega \in \locl$ we have to check that
\begin{enumerate}[(1)]

\item $\omega'' = [g_1|\ldots|g_n] \in \locl''$; and

\item $\omega' = [x_1|\Psi_{g_1}(x_2)|(\Psi_{g_1} \circ \Psi_{g_2})(x_3)|\ldots |(\Psi_{g_1} \circ \ldots \Psi_{g_{n-1}})(x_n)] \in \locl'$.

\end{enumerate}
Since $(\locl'', \Delta'', S'')$ is a locality, we see that $\omega'' \in \locl''$ via the sequence $H_0'', \ldots, H_n'' \in \Delta''$. To check that $\omega' \in \locl'$, use the formula in Remark \ref{prodHi} above: we see that $\omega'$ conjugates the sequence $H_0', \Psi_{g_1}(H_1'), (\Psi_{g_1} \circ \Psi_{g_2})(H_2'), \ldots, (\Psi_{g_1} \circ \ldots \Psi_{g_n})(H_n') \in \Delta'$. Since $(\locl', \Delta', S')$ is a locality, it follows that $\omega' \in \locl'$.
\end{proof}

The partial group $\locl$ determines a locality as follows. Recall that $\dd_{\Delta} \subseteq \locl$ by Proposition \ref{propext3}, where $\dd_{\Delta}$ is the set of all words $\omega \in \ww(\locl_1)$ that conjugate a sequence in $\Delta$.

\begin{prop}\label{propext3-1}

Set $\loct = \dd_{\Delta}$. Then, the triple $(\loct, \Delta, S)$ is a locality.

\end{prop}

\begin{proof}

It is clear that $S \leq \loct$. Let $\omega = [(x_1, g_1)|\ldots|(x_n,g_n)] \in \loct$. By definition of $\loct$, there is a sequence $H_0, \ldots, H_ n \in \Delta$ such that $\9{[(x_i,g_i)]}H_i = H_{i-1}$ for each $i = 1, \ldots, n$. Clearly, it follows that $\9{\Pi(\omega)}H_n = H_0$, so $\loct$ is closed by products, and we can see that $\loct$ is closed by inversion since we have $\9{[(x_i,g_i)]^{-1}}H_{i-1} = H_i$ for each $i = 1, \ldots, n$. This shows that $\loct$ is a partial group, with multiplication and inversion induced by those in $\locl$.

Let us prove now that $(\loct, \Delta, S)$ is a locality. Since $S$ is maximal in $\locl$, it must be maximal in $\loct$ too. We have to check that $(\loct, \Delta)$ is an objective partial group, where $\Delta$ is the collection defined in (\ref{collectionS}). Note that $\Delta$ is closed by overgroups and conjugation in $\loct$ since the same holds with respect to $\locl$ by Proposition \ref{propext3} (ii). Thus, $(\loct, \Delta)$ satisfies condition (O2) of objective partial groups.

Let $\dd'_{\Delta} \subseteq \ww(\loct_1)$ be the subset of all words $[(x_1,g_1)|\ldots|(x_n,g_n)]$ for which there exists some sequence $H_0, \ldots, H_ n \in \Delta$ such that $\9{[(x_i,g_i)]}H_i = H_{i-1}$ for each $i = 1, \ldots, n$. We have to check that $\dd'_{\Delta} = \loct$, and this is immediate by definition of $\loct$.
\end{proof}

\begin{cor}\label{propext3-2}

The subgroup $S' \leq S$ is strongly closed in the fusion system $\FF_{\Delta}(\loct)$.

\end{cor}

\begin{proof}

This follows immediately from the conjugation formula in Remark \ref{prodHi}.
\end{proof}

The locality $(\loct, \Delta, S)$ will play an important role in the next subsections, where we analyze its relation to $\locl$.


\subsection{Examples of isotypical extensions}\label{Sexpl}

We now present some examples of isotypical extensions, all of which produce localities. These examples have already been studied in the context of $p$-local finite groups, in \cite{OV} and in \cite{BCGLO2} respectively. Among other reasons, we include these examples here to show that our results and constructions are independent from the aforementioned papers.

Let us start by analyzing isotypical extensions where the fibre is a $p$-group. This situation was first studied in \cite{OV}(in the context of transporter systems), where the authors showed (among other results) that every extension of a transporter system by a $p$-group is again a transporter system, and the reader is referred to this paper for further details.

\begin{expl}\label{explOV}

Let $\locl'$ be a finite $p$-group. In other words, $\locl'_1 = S'$, and $\locl' = BS'$. Let $\Delta' = \{S'\}$. Then $(BS', \Delta', S')$ is a proper locality. In this example we show that an isotypical extension of a locality $(\locl'', \Delta'', S'')$ by $(BS', \Delta', S')$ always gives rise to a locality. More specifically, fix such an extension $BS' \to \locl \to \locl''$, and let $(\loct, \Delta, S)$ be the locality associated to the extension, as shown in Proposition \ref{propext3-1}. We show that $\loct = \locl$.

By Propositions \ref{propext3} and \ref{propext3-1}, we only have to show that $\locl \subseteq \loct ( = \dd_{\Delta})$. In other words, given $[(x_1, g_1)|\ldots|(x_n,g_n)] \in \locl$, we have to find a sequence of subgroups $H_0, \ldots, H_n \in \Delta$ such that $\9{[(x_i,g_i)]}H_i = H_{i-1}$ for all $i = 1, \ldots, n$.

Since $(\locl'', \Delta'', S'')$ is a locality, we have $[g_1|\ldots|g_n] \in \locl''$ via a sequence $P_0'', \ldots, P_n'' \in \Delta''$, and now an easy calculation shows that it is enough to take $H_i$ to be the pull-back of $S \Right2{} S'' \Left2{} P_i''$ for $i = 0, \ldots, n$. Notice that the fusion system of any such extension need not be saturated. In \cite{OV} the authors give an example of an extension of a transporter system by a $p$-group whose associated fusion system is not saturated (see right after \cite[Proposition 5.8]{OV}), and this same example applies here. We will not analyze the question of saturation of these extensions here. Rather than that, we leave this for a later subsection.

\end{expl}

Next we consider isotypical extensions where the base is a finite group. These extensions were analyzed in \cite{BCGLO2} (in the context of $p$-local finite groups). Here, we consider a more general situation, in which no saturation is assumed. This example in particular will play an important role in the forthcoming subsections.

\begin{expl}\label{particular1}

Let $\locl''$ be a finite group, so $\locl_1'' = G$, and $\locl'' = BG$. Fix some $S'' \in \Syl_p(G)$, and let $\Delta''$ be the collection of all subgroups of $S''$. Then $(BG, \Delta'', S'')$ is a locality (not necessarily proper). Let $(\locl', \Delta', S')$ be a locality, and fix an isotypical extension $\locl' \to \locl \to BG$. Again, let $(\loct, \Delta, S)$ be the locality associated to the extension, as in Proposition \ref{propext3-1}. We show that $\loct = \locl$.

As happened in the previous examples, by Propositions \ref{propext3} and \ref{propext3-1} we only have to show that for every simplex $\omega = [(x_1,g_1)|\ldots|(x_n,g_n)] \in \locl$ there is some sequence $H_0, \ldots, H_n \in \Delta$ such that $\9{[(x_i,g_i)]}H_i = H_{i-1}$ for each $i = 1, \ldots, n$.

By Proposition \ref{TCP1} and Lemma \ref{alternativecond}, we know that $\omega \in \locl$ if and only if
\begin{enumerate}[(i)]

\item $[g_1|\ldots|g_n] \in \locl''_n$; and\\[2pt]

\item $\omega' = [x_1|\Psi_{g_1}(x_2)|(\Psi_{g_1} \circ \Psi_{g_2})(x_3)|\ldots|(\Psi_{g_1} \circ \ldots \circ \Psi_{g_{n-1}})(x_n)] \in \locl'_n$.

\end{enumerate}
Set $y_1 = x_1$ and $y_i = (\Psi_{g_1} \circ \ldots \Psi_{g_{i-1}})(x_i)$ for $i = 2, \ldots, n$ for short, so $\omega' = [y_1|\ldots|y_n]$. Since $(\locl', \Delta', S')$ is a locality, the word $\omega'$ conjugates some sequence of subgroups $K'_0, \ldots, K'_n \in \Delta'$. That is, $\9{[y_i]}(K_i') = K_{i-1}'$ for each $i = 1, \ldots, n$.

Set $H'_0 = K'_0$ and $H'_i = (\Psi_{g_1} \circ \ldots \circ \Psi_{g_i})^{-1}(K'_i)$ for $i = 1, \ldots, n$. Then, $H'_i \in \Delta' \subseteq \Delta$ since every isotypical automorphism of $\locl'$ preserves $\Delta'$. Moreover, it follows that $\omega \in \dd_{\Delta}$ since we have $\9{[(x_i,g_i)]}(H'_i) = H'_{i-1}$ for each $i$ (the reader can compare this with the conjugation formula in Remark \ref{prodHi}).
\end{expl}

We finish this subsection by studying saturation in the example above. We choose to do this in this section since this example will be crucial later on in this section.

\begin{cor}\label{particular2}

Let $G$ be a finite group, and let $\locl' \to \locl \to BG$ be an isotypical extension, where
\begin{enumerate}[(i)]

\item $(\locl', \Delta', S')$ is a saturated locality such that $\Delta'$ contains all the $\FF'$-centric subgroups of $S'$; and

\item $(BG, \Delta'', S'')$ is a locality, with $S'' \in \Syl_p(G)$ and $\Delta'' = \{P'' \leq S''\}$.

\end{enumerate}
Let also $(\locl, \Delta, S)$ be the locality structure described in Example \ref{particular1}, and let $\FF = \FF_{\Delta}(\locl)$ be the associated fusion category. Then, $\Delta$ contains all the $\FF$-centric $\FF$-radical subgroups, and in particular $\FF$ is a saturated fusion system over $S$.

\end{cor}

\begin{proof}

Suppose otherwise that $\Delta$ does not contain all the $\FF$-centric $\FF$-radical subgroups of $S$, and let $P \notin \Delta$ be an $\FF$-centric $\FF$-radical subgroup of $S$. By definition of $\Delta$, this means that $P' = P \cap S' \notin \Delta'$, so in particular $P'$ is not $\FF'$-centric.

We can choose $P$ to be of maximal order among those $\FF$-centric $\FF$-radical subgroups not in $\Delta$. Thus we may assume as well that the saturation axioms hold in $\FF$ for all subgroups $R \leq S$ such that $|P| < |R|$. First we prove the following claim
\begin{itemize}

\item[(\textasteriskcentered)] There is some $Q \in P^{\FF}$ such that $Q'$ is fully $\FF'$-normalized.

\end{itemize}

Clearly, we may assume that $P'$ is not fully $\FF'$-normalized (in particular $P' \lneqq S'$), since otherwise there is nothing to show. Set $P_0' = P'$, and let $P_1' \leq S'$ be $\FF'$-conjugate to $P_0'$ and fully $\FF'$-normalized. Let also $\rho \in \Hom_{\FF'}(N_{S'}(P_0'), N_{S'}(P_1'))$ be such that $\rho(P_0') = P_1'$. Since $P_0' \lneqq S'$, we also have $P_0' \lneqq N_{S'}(P_0')$, and the saturation axioms in $\FF$ hold for $N_{S'}(P_0')$ and $N_{S'}(P_1')$.

Set now $R_0' = N_{S'}(P_0')$ and $K_0 = \{f \in \Aut(R_0') \,\, | \,\, f(P_0') = P_0'\}$, and recall from \cite[Appendix \S A]{BLO2} the notation
$$
N_S^{K_0}(R_0') = \{x \in N_{S'}(R_0') \,\, | \,\, c_x \in K_0\}.
$$
Let also $R_1' = \rho(R_0') \leq N_{S'}(P_1')$ and $K_1 = \{\rho \circ f \circ \rho^{-1} \,\, | \,\, f \in K_0\}$. Finally, let $R_2' \leq S'$ and $\gamma \in \Iso_{\FF}(R_0', R_2')$ be such that $R_2'$ is fully $K_2$-normalized in $\FF$, where
$$
K_2 = \{\gamma \circ f \circ \gamma^{-1} \,\, | \,\, f \in K_0\}.
$$
Set also $P_2' = \gamma(P_0')$.

By \cite[Proposition A.2 (b)]{BLO2} there exist $\chi, \chi' \in \Aut_{\FF}^{K_2}(R_2')$ and morphisms
$$
\alpha \in \Hom_{\FF}(N_S^{K_0}(R_0'), N_S^{K_2}(R_2')) \qquad \beta \in \Hom_{\FF}(N_S^{K_1}(R_1'), N_S^{K_2}(R_2'))
$$
such that $\alpha|_{R_0'} = \chi \circ \gamma$ and $\beta|_{R_1'} = \chi' \circ \gamma \circ \rho^{-1}$. We claim that there is a sequence of inequalities
$$
|N_{S'}(P_0')| \lneqq |N_{S'}^{K_1}(R_1')| \leq |N_{S'}^{K_2}(R_2')| \leq |N_{S'}(P_2'')|.
$$

Indeed, since $P_0'$ is not fully $\FF'$-normalized, we have $R_1' = \rho(N_{S'}(P_0')) \lneqq N_{S'}(P_1')$, and thus
$$
R_1' \lneqq N_{N_{S'}(P_1')}(R_1') = N_{S'}^{K_1}(R_1').
$$
This proves the leftmost inequality. The middle inequality holds immediately since $\beta$ restricts to an inclusion of $N_{S'}^{K_1}(R_1')$ into $N_{S'}^{K_2}(R_2')$. Finally, the rightmost inequality holds since every element in $N_{S'}^{K_2}(R_2')$ normalizes $P_2'$ by definition.

Set $Q = \alpha(P)$, so in particular $Q' = P_2'$, and we have $|N_{S'}(P')| < |N_{S'}(Q')|$. If $Q'$ is not fully $\FF'$-normalized, we can iterate the process, until we find some $Q \in P^{\FF}$ such that $Q'$ is fully $\FF'$-normalized. This proves the claim (\textasteriskcentered).

Suppose now that $P$ is such that $P'$ is fully $\FF'$-centralized, and let $\Aut^0_{\FF}(P) \leq \Aut_{\FF}(P)$ be the subgroup of automorphisms which induce the identity on $P'$ and $P''$. This is a normal subgroup of $\Aut_{\FF}(P)$, and it is also a $p$-group by \cite[Lemma 1.15]{BCGLO2}. Since $P$ is $\FF$-radical, it follows that
$$
\Aut^0_{\FF}(P) \leq \Inn(P).
$$
Now, $P$ acts on $C_{S'}(P')$ since $P$ normalizes $P'$, and thus $P/P'$ acts on $C_{S'}(P')/Z(P')$. If the coset $[(z,1)] \cdot Z(P') \in C_{S'}(P')/Z(P')$ is fixed by the action of $P/P'$, then $c_{(z,1)} \in \Aut^0_{\FF}(P) \leq \Inn(P)$. Thus, $\Pi[(z,1)|(x,g)] \in C_S(P)$ for some $[(x,g)] \in P$. Since $P$ is $\FF$-centric, this implies that $\Pi[(z,1)|(x,g)] \in Z(P) \leq P$, and hence
$$
[(z,1)] \in P \cap C_{S'}(P') = P \cap S' \cap C_S(P') = P' \cap C_S(P') = Z(P').
$$
This means that the action of $P/P'$ on $C_{S'}(P')/Z(P')$ only fixes the trivial element, and since all groups involved are $p$-groups, this implies that $C_{S'}(P') = Z(P')$. Since $P'$ is fully $\FF'$-normalized, this means that $P'$ is $\FF'$-centric, contradicting the assumption that $P' \notin \Delta'$.
\end{proof}

\begin{rmk}

Combining Example \ref{particular1} and Corollary \ref{particular2}, we see that any extension of a finite group by a saturated locality is, up to $p$-completion, the classifying space of a $p$-local finite group. The reader may compare this with \cite[Theorem A]{BLO6} for $p$-local finite groups. Note also that our results do not require $(\locl', \Delta', S')$ to be a proper locality, just as long as $\Delta'$ contains all the centrics.

\end{rmk}


\subsection{Further properties of isotypical extensions}\label{Sfurther}

We have seen earlier in this section how an isotypical extension gives rise to a locality (Proposition \ref{propext3-1}). The purpose of this subsection is to study some further properties of this locality associated to an isotypical extension. This subsection contains a series of technical results, and the reader can skip it in a first reading. Fix an isotypical extension
$$
\locl' \Right3{} \locl \Right3{\tau} \locl'',
$$
where $(\locl', \Delta', S')$ and $(\locl'', \Delta'', S'')$ are localities, as done in Hypothesis \ref{hyp1}. Let also $(\loct, \Delta, S)$ be the locality associated to the above extension.

\begin{lmm}\label{aux-1}

Let $H'' \leq \locl''$ be a subgroup satisfying that $H'' \cap S'' \in \Syl_p(H'')$. Set also
\begin{enumerate}[(i)]

\item $\locl(H'') \defin \{[(x,g)] \in \locl \, | \, [g] \in H''\}$;

\item $S(H'') \defin \locl(H'') \cap S$; and

\item $\Delta(H'') \defin \{P \leq S(H'') \, | \, P \cap S' \in \Delta'\}$.

\end{enumerate}
Then, the triple $(\locl(H''), \Delta(H''), S(H''))$ is a locality.

\end{lmm}

\begin{proof}

By definition of $\locl(H'')$ there is an isotypical extension $\locl' \Right1{} \locl(H'') \Right1{\tau} BH''$, where by abuse of notation $\tau$ denotes the restriction of $\tau \colon \locl \to \locl''$ to $\locl(H'')$. The statement follows by Example \ref{particular1}.
\end{proof}

In particular, if $P'' \in \Delta''$ is fully normalized in $\FF''$ and $H'' = N_{\locl''}(P'')$, then it follows that $H'' \cap S'' = N_{S''}(P'') \in \Syl_p(N_{\locl''}(P''))$, the triple $(\locl(H''), \Delta(H''), S(H''))$ is a locality, with $S(H'') \leq S$. Furthermore, if $P \leq S$ is such that $\tau(P) = P''$, then we have $N_{\locl}(P) = N_{\locl(H'')}(P)$:
\begin{itemize}

\item $N_{\locl}(P) \leq N_{\locl(H'')}(P)$, since the projection map $\tau \colon \locl \to \locl''$ sends $N_{\locl}(P)$ to a subgroup of $N_{\locl''}(P'')$; and

\item $N_{\locl(H'')}(P) = \locl(H'') \cap N_{\locl}(P) \leq N_{\locl}(P)$.

\end{itemize}

Recall that a \emph{section} of $\tau \colon \locl \to \locl''$ is a map (of sets) $\sigma \colon \locl'' \to \locl$ such that $\tau \circ \sigma = \Id_{\locl''}$. The following two lemmas deal with some particular sections of $\tau$.

\begin{lmm}\label{aux0}

The section $\sigma_0 \colon \locl'' \to \locl$ defined by $[g] \mapsto [(1,g)]$ satisfies the following properties.
\begin{enumerate}[(i)]

\item For each $[g] \in \locl''$, left conjugation by $\sigma_0(g)$ induces an element of $\Aut(\locl';S')$.

\item If $[g] \in \locl''$ and $H'', K'' \leq S''$ are such that $\9{[g]}(K'') = H''$, then left conjugation by $\sigma_0(g)$ induces an isomorphism of partial groups form $\locl(K'')$ to $\locl(H'')$.

\item If $[g_1|\ldots|g_n] \in \locl''$, then $[\sigma_0(g_1)|\ldots|\sigma_0(g_n)] \in \locl$.

\end{enumerate}

\end{lmm}

\begin{proof}

To prove parts (i) and (ii), we show that, for each $\omega = [(y_1,h_1)|\ldots|(y_n,h_n)] \in \locl(K'')$, we have
$$
[(1,g)|(y_1,h_1)|(1,g)^{-1}|(1,g)|(y_2,h_2)|(1,g)^{-1}|\ldots|(1,g)|(y_n,h_n)|(1,g)^{-1}] \in \locl.
$$
The conjugation formula of \ref{prodHi} implies that conjugation by $\sigma_0(g)$ induces an automorphism of $\locl'$, as well an isomorphism of partial groups from $\locl(K'')$ to $\locl(H'')$. The conjugation formula of \ref{prodHi} also implies that conjugation by $\sigma_0(g)$ induces an automorphism of $S'$, since the extension is isotypical.

Set $\eta = \eta(g^{-1},g) \in N_{\locl'}(S')$ for short, so $[(1,g)^{-1}] = [(\eta^{-1}, g^{-1})]$. Then, by Lemma \ref{alternativecond}, it is enough to show that the following conditions hold:
\begin{enumerate}[(1)]

\item $[g^{-1}|h_1|g|g^{-1}|h_2|g|\ldots |g^{-1}|h_n|g] \in \locl''$; and

\item $[1|a_1|b_1|1|a_2|b_2|\ldots|1|a_n|b_n] \in \locl'$, where $a_1 = \Psi_g(y_1)$ and $b_1 = (\Psi_g \circ \Psi_{h_1})(\eta^{-1})$, and 
$$
\begin{array}{l}
a_i = (\Psi_{g} \circ \Psi_{h_1} \circ \Psi_{g^{-1}} \circ \ldots \circ \Psi_{g} \circ \Psi_{h_{i-1}} \circ \Psi_{g^{-1}} \circ \Psi_{g})(y_i) \\[4pt]
b_i = (\Psi_{g} \circ \Psi_{h_1} \circ \Psi_{g^{-1}} \circ \ldots \circ \Psi_{g} \circ \Psi_{h_i})(\eta^{-1})
\end{array}
$$
for $i = 2, \ldots, n$.

\end{enumerate}
Condition (1) follows immediately since $H'' = \9{[g]}(K'')$ by assumption. To show that condition (2) also holds, notice that by definition the element $[\eta] \in N_{\locl}(S)$ defines a homotopy from $\Psi_g^{-1}$ to $\Psi_{g^{-1}}$. Recall from Lemma \ref{homotopies} that this is equivalent to
$$
[\Psi_{g^{-1}}(x) \cdot \eta] = [\eta \cdot \Psi_g^{-1}(x)]
$$
for all $[x] \in \locl'$. Furthermore, since $[\eta] \in N_{\locl}(S)$, it follows that $\Psi_{g^{-1}}[x] = [\eta \cdot \Psi_{g^{-1}}(x) \cdot \eta^{-1}]$. Equivalently, for all $[y] \in \locl'$ we have
$$
(\Psi_{g^{-1}} \circ \Psi_g)[y] = [\eta \cdot y \cdot \eta^{-1}].
$$
It is easy to see now that condition (2) above holds if and only if
$$
[\Psi_g(y_1)|(\Psi_g \circ \Psi_{h_1})(y_2)|\ldots|(\Psi_g \circ \Psi_{h_1} \circ \ldots \circ \Psi_{h_{n-1}})(y_n)] \in \locl'.
$$
Since $\omega \in \locl$, it follows that $[y_1|\Psi_{h_2}(y_2)|\ldots|(\Psi_{h_1} \circ \ldots \circ \Psi_{h_{n-1}})(y_n)] \in \locl'$ and the above condition holds.

Finally, we prove property (iii). We have to show that $[\sigma_0(g)|\ldots|\sigma_0(g_n)] \in \locl$ for each $[g_1|\ldots|g_n] \in \locl''$. By Proposition \ref{TCP1} and Lemma \ref{alternativecond}, part (iii) holds if the following conditions are satisfied: $[g_1|\ldots|g_n] \in \locl''$ and $[1|1|\ldots|1] \in \locl'$. Clearly, both conditions hold, and thus part (iii) follows.
\end{proof}

The section $\sigma_0$ has the disadvantage that in general it does not take values in $\loct$. Thus, we now show that there is a section with similar properties and which, in addition, factors through $\loct$.

\begin{lmm}\label{aux2}

There is a section $\sigma \colon \locl'' \to \locl$ satisfying the following properties:
\begin{enumerate}[(i)]

\item $\sigma(g) \in \loct$ for all $[g] \in \locl''$;

\item For each $[g] \in \locl''$, left conjugation by $\sigma(g)$ induces an element of $\Aut(\locl';S')$.

\item If $[g] \in \locl''$ and $H'', K'' \leq S''$ are such that $\9{[g]}(K'') = H''$, then left conjugation by $\sigma(g)$ induces an isomorphism of partial groups form $\locl(K'')$ to $\locl(H'')$.

\item If $[g_1|\ldots|g_n] \in \locl''$, then $[\sigma(g_1)|\ldots|\sigma(g_n)] \in \locl$.

\end{enumerate}
Furthermore, we can choose the above so that $\sigma(1) = [(1,1)]$.

\end{lmm}

\begin{proof}

The proof is divided into several steps for the reader's convenience.

\textbf{Step 1.} Definition of $\sigma$ and properties (i) and (iv). Fix some $[g] \in \locl''$, and let $P'' = L_{[g]}$ and $Q'' = R_{[g]}$, as in (\ref{SwSu0}). In particular, note that $P'', Q'' \in \Delta''$, by \cite[Lemma 2.14]{Chermak}. By Lemma \ref{aux0}, conjugation by $[(1,g)] \in \locl$ defines an isomorphism of partial groups $\locl(P'') \to \locl(Q'')$ by the formula $[(y,h)] \mapsto [(1,g)\cdot (y,h) \cdot (1,g)^{-1}]$. However, in general we have
$$
X \defin \9{[(1,g)]}(S(P'')) \neq S(Q'')
$$
(note that $X$ is a subgroup of $\locl(Q'')$). Since $(\locl(Q''), \Delta(Q''), S(Q''))$ is a locality, there is some $[(y,h)] \in \locl(Q'')$ such that $\9{[(y,h)]}X \leq S(Q'')$, and thus $\9{[(y,h)]}X = S(Q'')$ since $|S(P'')| = |X| = |S(Q'')|$.

By Lemma \ref{alternativecond}, we have $[(y,h)|(1,g)] \in \locl$, and thus
$$
[(y \cdot \eta(h,g), h \cdot g)] = \Pi[(y,h)|(1,g)] \in \loct,
$$
since the above element conjugates the subgroup $S(P'')$ to $S(Q'')$, and these are elements of $\Delta$. In particular, $[h] \in Q''$, since $\tau(\locl(Q'')) = Q''$ by definition. Let $[(z,h)] \in S(Q'') \leq S$ be a preimage of $[h]$. Then,
$$
[(z,h)^{-1}|(y \cdot \eta(h,g), h \cdot g)] \in \loct
$$
since $(\loct, \Delta, S)$ is a locality. Set $\sigma(g) \defin \Pi[(z,h)^{-1}|(y \cdot \eta(g,h), h \cdot g)] \in \loct$. Note that in the case of $1 \in \locl''$, conjugation by $[(1,1)]$ already satisfies the required conditions, so we can choose $\sigma(1)$ to be $[(1,1)]$.

By definition property (i) holds, and property (iv) follows easily too. Indeed, if $[g_1|\ldots|g_n] \in \locl''$, via a sequence $P_0'', \ldots, P_n'' \in \Delta''$, then the word $\omega = [\sigma(g_1)|\ldots|\sigma(g_n)]$ conjugates the sequence $S(P_0''), \ldots, S(P_n'') \in \Delta$, and thus $\omega \in \loct$. 

\textbf{Step 2.} Property (iii). Fix some $[g] \in \locl''$ as above, and let $[(1,g)]$, $[(y,h)]$ and $[(z,h)^{-1}]$ be the elements used in Step 1 to define $\sigma(g)$. Notice that left conjugation by any of these three elements must restrict to an automorphism of $S'$. Indeed,
\begin{enumerate}[(1)]

\item for $[(1,g)]$, it follows from Lemma \ref{aux0};

\item for $[(y,h)]$, it follows because both conjugation by $[(1,g)]$ and by $\Pi[(y,h)|(1,g)]$ do; and

\item for $[(z,h)^{-1}]$, it follows because this is an element of $S(Q'') \leq S$.

\end{enumerate}
Thus, conjugation by $\sigma(g)$ also restricts to an automorphism of $S'$. We have to check that restriction to $\locl'$ also produces an automorphism. For $[(1,g)]$ and $[(z,h)^{-1}]$ this is clear, either by Lemma \ref{aux0}, in the case of $[(1,g)]$, or because $[(z,h)^{-1}] \in S$.

It remains to check that conjugation by $[(y,h)]$ induces an automorphism of $\locl'$. Let $\omega_1 = [u_1|\ldots|u_n] \in \locl'$, via the sequence $P_0', \ldots, P_n' \in \Delta'$, and let $\varphi_g \in \Aut(\locl';S')$ be the automorphism induced by left conjugation by $[(1,g)]$. Then, $\omega_2 \defin \varphi_g(\omega_1) \in \locl'$ via the sequence $\varphi_g(P_0'), \ldots \varphi_g(P_n') \in \Delta'$. Set for short $\omega_2 = [v_1|\ldots|v_n]$, and $Q_i' = \varphi_g(P_i')$ for $i = 0, \ldots, n$.

Recall that $(\locl(Q''), \Delta(Q''), S(Q''))$ is a locality, and that $\Delta' \subseteq \Delta(Q'')$ by definition. Since conjugation by $[(y,h)]$ restricts to an automorphism of $S'$, we have $R_i' \defin \9{[(y,h)]}(Q_i') \in \Delta'$, for $i = 0, \ldots, n$, and it follows that
$$
[(y,h)|(v_1,1)|(y,h)^{-1}|(y,h)|(v_2,1)|(y,h)^{-1}|\ldots|(y,h)|(v_n,1)|(y,h)^{-1}] \in \locl(Q''),
$$
via the sequence $R_0', Q_0', Q_1', R_1', Q_1', Q_2', \ldots, Q_n', R_n' \in \Delta'$. The conjugation formula of Remark \ref{prodHi} now implies that this induces an automorphism of $\locl'$.

\textbf{Step 3.} Property (ii). Again, fix some $[g] \in \locl''$, and suppose that $\9{[g]}(K'') = H''$ for some $H'', K'' \leq S''$. Let also $[(1,g)]$, $[(y,h)]$ and $[(z,h)^{-1}]$ be as in previous steps of the proof. By Lemma \ref{aux0}, conjugation by $[(1,g)]$ induces an isomorphism from $\locl(K'')$ to $\locl(H'')$, and conjugation by $[(z,h)^{-1}] \in S(H'')$ clearly induces an automorphism of $\locl(H'')$. Thus, to prove property (iii), it is enough to show that conjugation by $[(y,h)] \in \locl(H'')$ also induces an automorphism of $\locl(H'')$.

Recall that $(\locl(H''), \Delta(H''), S(H''))$ is a locality by Lemma \ref{aux-1}. Let $[(u_1, k_1)|\ldots |(u_n, k_n)] \in \locl(H'')$, via some sequence $P_0, \ldots, P_n \in \Delta(H'')$. By definition,
$$
\Delta(H'') = \{R \leq S(H'') \, \, | \, \, R \cap S' \in \Delta'\}.
$$
This means that $\Delta' \subseteq \Delta(H'')$, and that we may assume that $P_0, \ldots, P_n \in \Delta'$ (if $P_i \notin \Delta'$, we can replace $P_i$ by $P_i \cap S' \in \Delta'$). Set $V_i = \9{[(y,h)]}U_i$, for $i = 0, \ldots, n$, and note that $V_i \in \Delta'$ by Step 2, since conjugation by $[(y,h)]$ induces an automorphism of $S'$, and $U_i \leq S'$. It follows that
$$
[(y,h)|(u_1,k_1)|(y,h)^{-1}|(y,h)|(u_2, k_2)|(y,h)^{-1}|\ldots|(y,h)|(u_n,k_n)|(y,h)^{-1}] \in \locl(H''),
$$
via the sequence $V_0, U_0, U_1, V_1, U_1, U_2, V_2, \ldots, U_n, V_n$. Property (ii) follows easily now.
\end{proof}

\begin{lmm}\label{aux1}

The subgroup $S$ is a Sylow $p$-subgroup of $\locl$: if $H \leq \locl$ is a $p$-group, then $H$ is conjugate in $\locl$ to a subgroup of $S$.

\end{lmm}

\begin{proof}

Let $H \leq \locl$ be a $p$-subgroup. Then $H''$ is a $p$-subgroup of $\locl''$, and it is conjugate to a subgroup of $S''$ by \cite[Proposition 2.21 (b)]{Chermak} since $(\locl'', \Delta'',S'')$ is a locality.

Suppose first that $H'' \leq S''$. In this case the pull-back $\locl(S'')$ is a locality with Sylow $p$-subgroup $S$ and $H \leq \locl(S'')$, and the statement follows by \cite[Proposition 2.21 (b)]{Chermak}. Suppose now that $H''$ is not a subgroup of $S''$, and let $[g] \in \locl''$ be such that $\9{[g]}(H'') = K'' \leq S''$. By \cite[Proposition 2.21 (a)]{Chermak}, the subgroup $H''$ normalizes some $Q'' \in \Delta''$, and this means that we can choose $[g] \in N_{\locl''}(Q'')$ so that $K'' \leq N_{S''}(Q'')$.

By Lemma \ref{aux0} the element $[(1,g)] \in \locl$ induces an isomorphism of partial groups from $\locl(H'')$ to $\locl(K'')$, and in particular the subgroup $H$ is sent to a $p$-subgroup $K \leq \locl(K'')$. To finish the proof, notice that $(\locl(K''), \Delta(K''), S(K''))$ is a locality by Example \ref{particular1}, and thus by \cite[Proposition 2.21 (b)]{Chermak} $K$ is conjugate in $\locl(K'')$ to a subgroup of $S(K'') \leq S$.
\end{proof}


\subsection{Isotypical extensions and saturated localities}\label{isosat}

Finally, we study isotypical extensions of saturated localities. More specifically, given an isotypical extension $\locl' \to \locl \to \locl''$, with $(\locl', \Delta', S')$ and $(\locl'', \Delta'', S'')$ saturated localities, we want sufficient conditions for $\locl$ to be equivalent to the classifying space of a $p$-local finite group after $p$-completion. Theorem \ref{good1} below motivates our approach to this question.

\begin{thm}\label{good1}

Let $(\locl', \Delta', S')$ and $(\locl'', \Delta'', S'')$ be saturated localities, let $\locl' \Right1{} \locl \Right1{\tau} \locl''$ be an isotypical extension, and let $(\loct, \Delta, S)$ be the locality described in Proposition \ref{propext3-1}, with associated fusion system $\FF_{\Delta}(\loct)$. Suppose in addition that the following conditions hold:
\begin{enumerate}[(a)]

\item $\loct$ contains $\locl'$ as a partial normal subgroup; and

\item $\Delta$ contains all the $\FF_{\Delta}(\loct)$-centric $\FF_{\Delta}(\loct)$-radical subgroups of $S$.

\end{enumerate}
Then the following holds.
\begin{enumerate}[(i)]

\item $(\loct, \Delta, S)$ is a saturated locality.

\item The inclusion $\loct \subseteq \locl$ induces an equivalence after $p$-completion.

\item $\FF'$ is a normal subsystem of $\FF_{\Delta}(\loct)$.

\end{enumerate}
In particular, after $p$-completion $\locl$ is equivalent to the classifying space of a $p$-local finite group, and $\FF_{\Delta}(\loct) = \FF_S(\locl)$.

\end{thm}

\begin{proof}

By definition, $\FF_{\Delta}(\loct)$ is $\Delta$-generated and $\Delta$-saturated. Thus, saturation of $\FF_{\Delta}(\loct)$ follows immediately from \cite[Theorem A]{BCGLO1} or \cite[Proposition 3.6]{OV}. The rest of the proof is divided into steps for the reader's convenience.

\textbf{Step 1.} The $p$-completed space $|\loct|^{\wedge}_p$ is equivalent to the classifying space of a $p$-local finite group. Let $\LL$ be the unique centric linking system associated to the saturated fusion system $\FF_{\Delta}(\loct)$. We show that $|\loct|^{\wedge}_p \simeq |\LL|^{\wedge}_p$. Let $\hh = \Delta \cap \FF_{\Delta}(\loct)^c$. By definition the set $\hh$ contains all the $\FF_{\Delta}(\loct)$-centric $\FF_{\Delta}(\loct)$-radical subgroups of $S$. Set also,
\begin{itemize}

\item $\TT = \TT_{\Delta}(\loct)$, the transporter category associated to $(\loct, \Delta, S)$;

\item $\TT_{\hh} \subseteq \TT$, the full subcategory with object set $\hh$;

\item $\LL$, the centric linking system associated to the saturated fusion system $\FF_{\Delta}(\loct)$;

\item $\LL_{\hh} \subseteq \LL$, the full subcategory with object set $\hh$.

\end{itemize}
Then there is a zigzag
$$
|\loct| \Left2{\proj^{\TT}_{\loct}} |\TT| \Left2{\incl_{\TT}} |\TT_{\hh}| \Right2{\proj_{\LL}} |\LL_{\hh}| \Right2{\incl_{\LL}} |\LL|,
$$
where the map $|\loct| \Left1{\proj^{\TT}_{\loct}} |\TT|$ is a weak equivalence by Theorem \ref{equivnerves}, and all the other maps are equivalences after $p$-completion by \cite[Proposition 4.6]{OV}.

\textbf{Step 2.} Proof of part (ii). The proof follows \cite[Theorem 5.1]{Diaz}, with some small modifications. Consider the commutative diagram of extensions
\begin{equation}\label{diag43}
\xymatrix{
BS' \ar[r] \ar[d] & BS \ar[r] \ar[d] & BS'' \ar[d] \\
\locl' \ar[r] & \locl \ar[r] & \locl''
}
\end{equation}
from Lemma \ref{propext1}. Each row has an associated mod $p$ Lyndon-Hochschild-Serre spectral sequence, whose second pages are
$$
E_{2,S}^{r,s} = H^r(S'; H^s(S''; \F_p)) \qquad \mbox{and} \qquad E_{2, \locl}^{r,s} = H^r(\locl''; H^s(\locl'; \F_p)), 
$$
and converging to $H^{\ast}(S; \F_p)$ and $H^{\ast}(\locl; \F_p)$, respectively. Consider also the fusion system $\FF_{\Delta}(\loct)$ over $S$, and recall that $S' \leq S$ is a strongly $\FF_{\Delta}(\loct)$-closed subgroup by Corollary \ref{propext3-2}. Thus, by \cite[Theorem 1.1]{Diaz} there is a spectral sequence with second page
$$
E_{2, \loct}^{r,s} = H^r(S/S'; H^s(S'; \F_p))^{\FF_{\Delta}(\loct)} = H^r(S''; H^s(S'; \F_p))^{\FF_{\Delta}(\loct)}
$$
and converging to $H^{r+s}(\FF_{\Delta}(\loct); \F_p)$. By part (i), $\loct$ is mod $p$ equivalent to the classifying space of a $p$-local finite group, and thus by \cite[Theorem 5.8]{BLO2} there is an isomorphism
$$
H^{\ast}(\loct; \F_p) \cong H^{\ast}(\FF_{\Delta}(\loct); \F_p).
$$

By the same arguments of the proof of \cite[Theorem 5.1]{Diaz}, there is a cohomological Mackey functor $(A,B) \colon \FF_{\Delta}(\loct) \to CCh^2(\Z_{(p)})$, which produces morphisms of spectral sequences
$$
\xymatrix@C=3cm{
H^r(\locl''; H^s(\locl';\F_p)) \ar@/^2pc/[rr]^{H^{r,s}(A)(\iota)} & & H^r(S''; H^s(S'; \F_p)) \ar@/^2pc/[ll]^{H^{r,s}(B)(\iota)}
}
$$
where $H^{r,s}(A)(\iota)$ coincides with the restriction morphism induced by diagram (\ref{diag43}). The isomorphism between $E_{2,\locl}^{\ast,\ast}$ and $E_{2, \loct}^{\ast, \ast}$ follows now by the same arguments of \cite{Diaz}. To finish the proof, this isomorphism of spectral sequences implies an isomorphism of the $\infty$-pages, and thus there is an isomorphism $H^{\ast}(\locl; \F_p) \cong H^{\ast}(\loct; \F_p)$, which is induced by the inclusion $\loct \to \locl$.

\textbf{Step 3.} The equality $\FF_{\Delta}(\loct) = \FF_S(\locl)$. Set for short $\FF = \FF_S(\locl)$. By part (i) and Step 2 we know that the $p$-completion of $\locl$ is equivalent to the classifying space of a $p$-local finite group. Consider the commutative triangle
$$
\xymatrix{
 & BS \ar[rd]^{\gamma} \ar[ld]_{\gamma'} & \\
|\locl| \ar[rr]_{(-)^{\wedge}_p} & & |\locl|^{\wedge}_p
}
$$
and consider the topological construction introduced in \cite[Section 7]{BLO2}. Given a $p$-group $R$, this construction assigns a fusion system $\ee(\mu)$ to each map $\mu \colon BR \to X$, as follows. For each $P, Q \leq R$,
$$
\Hom_{\ee(\mu)}(P,Q) = \{f \in \Inj(P,Q) \, | \, \mu|_{BP} \simeq \mu|_{BQ} \circ Bf\},
$$
where $\Inj(P,Q)$ is the set of injective group homomorphism from $P$ to $Q$. In the particular case where $X$ is the classifying space of a $p$-local finite group and $\mu \colon BR \to X$ is the inclusion of its Sylow $p$-subgroup, the authors prove in \cite[Proposition 7.3]{BLO2} that $\ee(\mu)$ is (isomorphic to) the fusion system of the original $p$-local finite group.

Applying the topological construction to the triangle above, we get the following sequence of inclusions and isomorphisms
$$
\FF_{\Delta}(\loct) \subseteq \FF \subseteq \ee(\gamma') \subseteq \ee(\gamma) \cong \FF_{\Delta}(\loct).
$$
More specifically, from left to right, the first inclusion is given by definition of $\FF_{\Delta}(\loct)$ and $\FF$; the second inclusion follows from the topological construction, since clearly any conjugation by an element of $\locl$ will induce an morphism in $\ee(\gamma')$; the third inclusion is given again by the topological construction; and the isomorphism follows from \cite[Proposition 7.3]{BLO2}, together with part (ii).

\textbf{Step 4.} Proof of part (iii). By Step 3, it is enough to show that the fusion systems $\FF' \subseteq \FF$ satisfy conditions (N1)-(N4) in Definition \ref{normalF}. By hypothesis the fusion system $\FF'$ is saturated since $(\locl', \Delta', S')$ is a proper locality, and condition (N1) holds. Condition (N2) also holds easily: $\locl'$ is a partial normal subgroup of $\locl$, and thus $S' = S \cap \locl'$ is strongly $\FF$-closed.

To show that condition (N3) holds, let $P' \leq Q' \leq S'$ and $\gamma$ be as in Definition \ref{normalF}. By definition, the morphism $\gamma$ is the left conjugation morphism induced by some word $[(x_1,g_1)|\ldots|(x_n,g_n)] \in \ww(\locl_1)$, so that
$$
\9{[(x_i,g_i)]}(\ldots \9{[(x_{n-1}, g_{n-1})]}(\9{[(x_n,g_n)]}Q')) \leq S
$$
for all $i = 1, \ldots, n$. In particular, it is enough to check the case $n = 1$. That is, $[(x,g)] \in \locl$ such that ${[(x,g)]}Q' \leq S'$. By the conjugation formula in Remark \ref{prodHi}, we have
$$
\Pi[(x,g)|(y,1)|(x,g)^{-1}] = [(x \cdot \Psi_g(y) \cdot x^{-1}, 1)]
$$
for all $[(y,1)] \in Q'$. Thus, $\gamma = \alpha \circ \Psi_g$, where $\alpha$ is the left conjugation morphism induced by $[(x, 1)]$. Clearly, there are bijections
$$
\xymatrix@R=1mm{
\Hom_{\FF'}(P',Q') \ar[rr]^{\cong} & & \Hom_{\FF'}(\Psi_g(P'), \Psi_g(Q')) \ar[rr]^{\cong} & & \Hom_{\FF'}(\gamma(P'), \gamma(Q'))\\
f \ar@{|->}[rr] & & \Psi_g \circ f \circ \Psi_g^{-1} & & \\
 & & f' \ar@{|->}[rr] & & \alpha \circ f' \circ \alpha^{-1}
}
$$
and condition (N3) follows.

Finally, let us show that condition (N4) holds. Let $f \in \Aut_{\FF'}(S')$, and let $[x] \in N_{\locl'}(S')$ be such that $f$ is the left conjugation automorphism induced by $[x]$. Then, for each element $[(z,g)] \in C_S(S')$, we have $[(x,1)|(z,g)|(x,1)^{-1}|(z,g)^{-1}] \in \locl$, and
$$
\alpha = \Pi[(x,1)|(z,g)|(x,1)^{-1}|(z,g)^{-1}] = \Pi[(x,1)|(z \cdot \Psi_g(x^{-1}) \cdot z^{-1},1)] \in S'.
$$
Notice that the conjugation actions on $S'$ of the elements $[x]$ and $[z \cdot \Psi_g(x^{-1}) \cdot z^{-1}]$ are the same since $[(z,g)]$ centralizes $S'$, and hence $\alpha \in C_{S'}(S') = Z(S')$.
\end{proof}

\begin{rmk}

Given a good isotypical extension $\locl' \to \locl \to \locl''$, the associated locality $(\loct, \Delta, S)$ described above need not be proper. For example, consider the split extension $\locl' \to \locl' \times \locl'' \to \locl''$. In this case, $\loct = \locl' \times \locl''$, and $(\loct, \Delta, S)$ is proper if and only if both $\locl'$ and $\locl''$ are proper, which is not necessarily the case.

\end{rmk}

\begin{defi}\label{defigood}

Let $(\locl', \Delta', S')$ and $(\locl'', \Delta'', S'')$ be saturated localities, let $\locl' \Right1{} \locl \Right1{\tau} \locl''$ be an isotypical extension, and let $(\loct, \Delta, S)$ be the locality described in Proposition \ref{propext3-1}, with associated fusion system $\FF_{\Delta}(\loct)$. The above extension is \emph{good} if the following conditions hold:
\begin{enumerate}[(i)]

\item $\loct$ contains $\locl'$ as a partial normal subgroup; and

\item $\Delta$ contains all the $\FF_{\Delta}(\loct)$-centric $\FF_{\Delta}(\loct)$-radical subgroups of $S$.

\end{enumerate}

\end{defi}

\begin{rmk}

Let $(\locl', \Delta', S')$, $(\locl'', \Delta'', S'')$ be saturated localities, and let $\varepsilon \colon \locl'' \to \Out(\locl';S')$ be an outer action. Assuming that some extension of $(\locl', \locl'', \varepsilon)$ is good, it is not clear whether all the extensions of $(\locl', \locl'', \varepsilon)$ are good or not.

\end{rmk}

The rest of this section is devoted to present sufficient (but not necessary) conditions for an isotypical extension $\locl' \to \locl \to \locl''$ to be good. Given such an extension, let $(\loct, \Delta, S)$ be the locality described in Proposition \ref{propext3-1}. Note that in general $\loct$ does not contain $\locl'$ as a partial subgroup, and as a consequence the fusion system $\FF_{\Delta}(\loct)$ may not contain $\FF'$ as a subsystem. The concept of \emph{rigid} extensions is thus introduced to deal with condition (i) in Definition \ref{defigood}. Regarding condition (ii) in \ref{defigood}, our analysis follows the steps of \cite{OV}, where the authors study extensions of transporter systems by $p$-groups. This way we introduce \emph{admissible} extensions, generalizing the same concept introduced in \cite[Definition 5.10]{OV}. The results in the previous subsections, and in particular Corollary \ref{particular2}, impose some mild restrictions to our approach, which we summarize as follows. 

\begin{hyp}\label{hyp3}

Fix an isotypical extension $\locl' \Right1{} \locl \Right1{\tau} \locl''$, where
\begin{enumerate}[(a)]

\item $(\locl', \Delta', S')$ is a saturated locality such that $\Delta'$ contains all $\FF'$-centric subgroups.

\item $(\locl'', \Delta'', S'')$ is a saturated locality.

\end{enumerate}
The twisting function that determines the above extension is assumed to satisfy the same properties (1)-(4) as listed in Hypothesis \ref{hyp1}. Also, the following notation will be tacitly used throughout the rest of this section.
\begin{enumerate}[(i)]

\item $S \leq \locl$ and $\Delta$ are as described in Lemma \ref{propext1} and (\ref{collectionS}) respectively.

\item $\FF$ is the fusion system over $S$ whose morphisms are generated by all the conjugations among subgroups of $S$ induced by elements of $\locl$.

\item $(\loct, \Delta, S)$ is the locality described in Proposition \ref{propext3-1}, with fusion system $\FF_{\Delta}(\loct)$.

\end{enumerate}

\end{hyp}

\begin{lmm}\label{propext5-2}

For all $P, Q \in \Delta$ there is an equality $\Hom_{\FF_{\Delta}(\loct)}(P,Q) = \Hom_{\FF}(P,Q)$.

\end{lmm}

\begin{proof}

By definition, $\FF_{\Delta}(\loct) \subseteq \FF$. Fix $P, Q \in \Delta$, and let $\omega = [(x_1, g_1)|\ldots|(x_n,g_n)] \in \ww(\locl_1)$ be such that $\9{\omega}P \leq Q$ (so that $\omega$ induces a morphism in $\Hom_{\FF}(P,Q)$). Since $P, Q \in \Delta$, it follows that $\omega \in \loct$ by definition, and the statement follows.
\end{proof}

\begin{prop}\label{aux3}

Each $\FF_{\Delta}(\loct)$-conjugacy class of elements of $\Delta$ contains some element $P$ such that $P'$ is fully $\FF'$-normalized and $P''$ is fully $\FF''$-normalized.

\end{prop}

\begin{proof}

Recall from Lemma \ref{propext5-2} that $\Hom_{\FF_{\Delta}(\loct)}(P,Q) = \Hom_{\FF}(P,Q)$ for all $P, Q \in \Delta$. First we show that we may assume $P''$ to be fully $\FF''$-normalized. Suppose otherwise that $P''$ is not fully $\FF''$-normalized. Since $\FF''$ is saturated, there exists some morphism $f'' \colon P'' \to R''$ such that $R''$ is fully $\FF''$-normalized. Since $\FF'' = \FF_{\Delta''}(\locl'')$ and $P'' \in \Delta''$, there exists some $[g] \in \locl''$ such that $f''$ is the left conjugation homomorphism induced by $[g]$. Let $\sigma(g) \in \locl$ be as in Corollary \ref{aux2}. Then, $R = \9{\sigma(g)}P$ is such that $\tau(R) = R''$ is fully normalized in $\FF''$.

Suppose for simplicity that $P$ already satisfies that $P''$ is fully normalized in $\FF''$. Consider the locality $(\locl(H''), \Delta(H''), S(H''))$, where $H'' = N_{\locl''}(P'')$. Note that the associated fusion system $\FF(H'')$ is saturated by Corollary \ref{particular2}.

If $P'$ is not fully $\FF'$-normalized, then there are some subgroup $R' \leq S(H'')$ and some morphism $f \in \Hom_{\FF(H'')}(N_{S(H'')}(P'), N_{S(H'')}(R'))$ such that $f(P') = R'$, where $R'$ is chosen to be fully $\FF(H'')$-normalized. In particular, $R'$ is fully $\FF'$-normalized because $S' \leq S(H'')$ is strongly $\FF(H'')$-closed.

Notice that $P \leq N_{S(H'')}(P')$ by construction, and thus $R \defin f(P) \leq S(H'') \leq S$ is $\FF$-conjugate to $P$ and $R'$ is fully $\FF'$-normalized. Furthermore, $R'' = P''$ by construction, and $P''$ was already assumed to be fully $\FF''$-normalized by the first part of the proof.
\end{proof}

We are now ready to introduce rigid and admissible extensions. Before we do so, let us motivate the definition of admissible extension. The following is a (partial) generalization of \cite[Lemma 5.9]{OV}.

\begin{prop}\label{aux4}

Set $S_1'' = \Ker(S'' \Right2{\incl} \locl'' \Right2{\varepsilon} \Out(\locl';S') \Right2{} \Out_{\fus}(S'))$. Then, for each $P \in \Delta$ that is $\FF_{\Delta}(\loct)$-centric $\FF_{\Delta}(\loct)$-radical, the following holds.
\begin{enumerate}[(i)]

\item $P'$ is $\FF'$-centric; and

\item $C_{S_1''}(P'') \leq P''$.

\end{enumerate}

\end{prop}

\begin{proof}

By Lemma \ref{aux3} we may assume that both $P'$ and $P''$ are fully normalized in the corresponding fusion systems. Let us prove first that $P'$ is $\FF'$-centric. Suppose otherwise that $P'$ is not $\FF'$-centric, and let $K = P \cdot C_{S'}(P')$. By hypothesis, $P \lneqq K$ and hence $P \lneqq N_K(P)$, and $P' \lneqq N_K(P) \cap S'$. Since $P$ is $\FF_{\Delta}(\loct)$-centric and the elements in $N_K(P)$ induce the identity both on $P'$ and $P''$ (modulo an inner automorphism of $P$), we have
$$
\{1\} \neq N_K(P)/P \leq \Out_S(P) \cap O_p(\Out_{\FF_{\Delta}(\loct)}(P)),
$$
contradicting the hypothesis that $P$ is $\FF_{\Delta}(\loct)$-radical.

Finally we prove the second part of the statement. Again, we proceed by contradiction, so assume that $C_{S_1''}(P'') \not \leq P''$. For a subgroup $Q' \leq S'$ set
$$
K(Q') = \{[(x,g)] \in S \,\, | \,\, [(x,g)] \in C_S(Q') \mbox{ and } [g] \in C_{S''}(P'')\}.
$$
We claim that $K(P') \not \leq P$. Indeed, if $K(P') \leq P$ then $\tau(K(P')) \leq P''$. However, by definition we have
$$
K(S') \leq K(P') \qquad \mbox{and} \qquad \tau(K(S')) = C_{S_1''}(P'').
$$
Since $C_{S_1''}(P'') \not \leq P''$ it follows that $K(P') \not \leq P$.

Next we claim that $K(P') \cap P \lneqq K(P') \cap N_S(P)$. Suppose otherwise that we have $K(P') \cap P = K(P') \cap N_S(P)$, and let $J(P) \leq N_S(P')$ be the subgroup generated by $P$ and $K(P')$ (note that by definition $K(P') \leq N_S(P')$). By assumption, $P \lneqq J(P)$, and thus $P \lneqq N_{J(P)}(P)$. Now, by definition, for each $\alpha \in N_{J(P)}(P)$ there is some $\beta \in P$ such that
$$
(c_{\alpha})|_{P'} = (c_{\beta})|_{P'} \qquad \mbox{and} \qquad (c_{\tau(\alpha)})|_{P''} = (c_{\tau(\beta)})|_{P''}.
$$
This means that if $\alpha \notin P$, then conjugation by $\beta^{-1} \cdot \alpha \in N_{J(P)}(P) \setminus P$ induces the identity on both $P'$ and $P''$, which contradicts the original hypothesis that $K(P') \cap P = K(P') \cap N_S(P)$.

Thus we have $K(P') \cap P \lneqq K(P') \cap N_S(P)$. Since $P$ is assumed to be $\FF_{\Delta}(\loct)$-centric, we have
$$
\{1\} \neq (K(P') \cap N_S(P)) \cdot P / P \leq \Out_S(P) \cap O_p(\Out_{\FF_{\Delta}(\loct)}(P)),
$$
which contradicts the hypothesis that $P$ is $\FF_{\Delta}(\loct)$-radical.
\end{proof}

\begin{defi}\label{rigidext}

Let $\locl' \Right1{} \locl \Right1{\tau} \locl''$ be the isotypical extension fixed in Hypothesis \ref{hyp3}, with associated outer action $\varepsilon \colon \locl'' \to \Out(\locl';S')$.
\begin{enumerate}[(i)]

\item The above extension is \emph{rigid} if $S_0'' \defin \Ker(S'' \Right1{\incl} \locl'' \Right1{\varepsilon} \Out(\locl';S')) \in \Delta''$.

\item The above extension is \emph{adimissible} if, upon setting
$$
S_1'' \defin \Ker(S'' \Right2{\incl} \locl'' \Right2{\varepsilon} \Out(\locl';S') \Right2{} \Out_{\fus}(S')),
$$
the following condition holds: if $P'' \leq S''$ is fully $\FF''$-centralized and $C_{S_1''}(P'') \leq P''$, then $P'' \in \Delta''$.

\end{enumerate}

\end{defi}

\begin{rmk}\label{perfect1}

We have some observations to make about the above definition.

\begin{enumerate}[(i)]

\item It may seem that rigidity is a very strong condition to impose on an extension. Consider for example the situation where the action $\varepsilon$ is \emph{faithful} (i.e., $\Ker(\varepsilon) = \{1\}$). Rigidity implies in this case that the trivial subgroup of $S''$ must be an object in $\Delta''$, which forces $\locl''$ to be a group (with $\locl'' = \ww(\locl''_1)$). Notice that in this case $\locl''$ is essentially acting on $\locl'$ as an actual group, via the subgroup of $\Out(\locl';S')$ generated by the image of $\varepsilon$. Thus, the idea behind rigidity is that is a locality is acting as a group, it should be a group.

\item Let $S_0'', S_1'' \leq S''$ be as above. By definition, these two subgroups are strongly $\FF''$-closed, and $S_0'' \leq S_1''$. Furthermore, this inclusion is an equality for all odd primes by \cite[Theorem C]{Oliver}. Thus, every admissible extension is rigid for all odd primes. Whether the same is true for $p = 2$ remains an open question.

\item The extensions considered in Corollary \ref{particular2} are always rigid and admissible, since in this case $\Delta''$ is the collection of all subgroups of $S''$.

\item Let $(\locl', \Delta', S')$, $(\locl'', \Delta'', S'')$ and $\varepsilon \colon \locl'' \to \Out(\locl';S')$ be as in Hypothesis \ref{hyp3}. If an extension of this data is rigid (respectively admissible), then so is any other extension of this data, just by definition of rigid (respectively admissible) extensions.

\end{enumerate}

\end{rmk}

\begin{thm}\label{propext6}

Let $\locl' \Right1{} \locl \Right1{\tau} \locl''$ be the isotypical extension fixed in Hypothesis \ref{hyp3}. Then the following holds.
\begin{enumerate}[(i)]

\item If the extension is rigid, then $\locl'$ is a partial normal subgroup of $\loct$, and $\FF_{\Delta}(\loct)$ contains $\FF'$ as a subsystem.

\item If the extension is admissible, then the fusion system $\FF_{\Delta}(\loct)$ is saturated.

\end{enumerate}
In particular, if the extension is both rigid and admissible, then it is a good extension.

\end{thm}

\begin{proof}

Clearly, if the extension is both rigid and admissible, then it is good, so there is nothing to prove in this case. Suppose first that the extension is rigid. Let $C_S(\locl') = \{[(x,g)] \in S \,\, | \,\, \forall [y] \in \locl', \, \omega_y =[(x,g)|(y,1)|(x,g)^{-1}] \in \locl \mbox{ and } \Pi(\omega_y) = [(y,1)]\}$. Clearly, this is a subgroup of $S$. Furthermore, it satisfies the following condition: for each $P' \in \Delta'$, we have $P' \cdot C_S(\locl') \in \Delta$ since
$$
(P' \cdot C_S(\locl')) \cap S' \geq P' \qquad \mbox{and} \qquad \tau(P' \cdot C_S(\locl')) = S_0''.
$$
Thus, if $[x_1|\ldots|x_n] \in \locl'_n$ via some $P_0', \ldots, P_n' \in \Delta'$, then $[(x_1,1)|\ldots|(x_n,1)] \in \loct$ via the sequence $P_0, \ldots, P_n \in \Delta$, where $P_i = P_i' \cdot C_S(\locl')$. This shows that $\locl' \subseteq \loct$. That $\locl'$ is a partial subgroup of $\loct$ follows by the description of the simplices of $\locl'$ as pairs $[(x,1)]$, and normality follows because $\locl'$ is normal in $\locl$. The inclusion $\FF' \subseteq \FF_{\Delta}(\loct)$ follows easily.

Suppose now that the extension is admissible. By definition, $\FF_{\Delta}(\loct)$ is $\Delta$-generated and $\Delta$-saturated. Since $\locl$ is an admissible extension it follows from Proposition \ref{aux4} that $\Delta$ contains all the $\FF_{\Delta}(\loct)$-centric $\FF_{\Delta}(\loct)$-radical subgroups, and hence by \cite[Theorem A]{BCGLO1} or \cite[Proposition 3.6]{OV} it follows that $\FF_{\Delta}(\loct)$ is saturated.
\end{proof}

\begin{cor}\label{propext6-1}

Let $(\locl', \Delta', S')$ be a saturated locality such that $\Delta'$ contains all the centric subgroups, let $(\locl'', \Delta'', S'')$ be a saturated locality such that $\Delta''$ contains all the centric subgroups, and let $\varepsilon \colon \locl'' \to \Out(\locl';S')$ be an outer action. If $\varepsilon$ is the trivial morphism, then every extension of $(\locl', \locl'', \varepsilon)$ is good.

\end{cor}

\begin{proof}

If $\varepsilon \colon \locl'' \to \Out(\locl';S')$ is trivial, then any extension of $(\locl', \locl'', \varepsilon)$ is clearly rigid and admissible, and thus by Theorem \ref{propext6} the extension is good.
\end{proof}


\section{Applications}\label{Sapp}

In this section we study some situations related to isotypical extensions of localities. We start by relating extensions of finite groups to isotypical extensions of localities.

\begin{expl}\label{Expl1}

Let $K \Right1{} G \Right1{\tau} Q$ be an extension of finite groups, and let $p$ be a prime. Fix Sylow $p$-subgroups
$$
S_K \in \Syl_p(K) \qquad S_G \in \Syl_p(G) \qquad S_Q \in \Syl_p(Q)
$$
so that the above extension of groups restricts to an extension $S_K \to S_G \to S_Q$. Set also $\FF_K = \FF_{S_K}(K)$, $\FF_G = \FF_{S_G}(G)$ and $\FF_Q = \FF_{S_Q}(Q)$. Finally, let $(\locl_K, \Delta_K, S_K)$ be the proper locality associated to $K$, where $\Delta_K$ is the collection of centric subgroups in $\FF_K$.

An easy computation shows that $Z(K) \cong Z(\locl_K) \times Z'(K)$, where $Z'(K) = O_{p'}(Z(K))$. Furthermore, there is a natural group homomorphism $\Out(K) \to \Out((BK)^{\wedge}_p) \cong \Out(\locl_K;S_K)$. Thus, there is a natural map
$$
\underline{\aut}(BK) \Right3{\Omega} \underline{\aut}((BK)^{\wedge}_p) \cong \underline{\aut}(\locl_K;S_K).
$$
The group extension $K \to G \to Q$ is classified by a map $BQ \to B \underline{\aut}(BK)$, and composition with $\Omega$ yields a map $BQ \to \underline{\aut}(\locl_K;S_K)$ which in turn determines an isotypical extension
$$
\locl_K \Right2{} \locl \Right2{} BQ.
$$
Moreover, this is a good extension in the sense of \ref{defigood} (here, $Q$ is seen as a locality $(BQ, \Delta_Q, S_Q)$, where $\Delta_Q$ is the collection of all subgroups of $S_Q$).

Consider now the following modification. The classifying map $BQ \to B\underline{\aut}(\locl_K;S_K)$ induces an outer action
$$
\alpha \colon Q \Right2{} \Out(\locl_K;S_K).
$$
Given such an action, let $\Delta_Q$ be the smallest collection of subgroups of $S_Q$ that satisfies the following conditions:
\begin{enumerate}[(1)]

\item it is closed by $\FF_Q$-conjugation and overgroups;

\item it contains every $\FF_Q$-centric $\FF_Q$-radical subgroup of $S_Q$; and

\item the action $\alpha$ is both rigid and admissible.

\end{enumerate}
Let also $(\locl_Q, \Delta_Q, S_Q)$ be the locality associated to $Q$ with $N_{\locl_Q}(X, Y) = N_Q(X, Y)$ for each $X, Y \in \Gamma_Q$. This is indeed a locality  by \cite[Example 2.10]{Chermak}, and there is an obvious monomorphism of partial groups $\iota \colon \locl_Q \to BQ$. Moreover, a group extension $K \to H \to Q$ with the above outer action produces a map 
$$
\locl_Q \Right2{\iota} BQ \Right2{} B\underline{\aut}(BK) \Right2{\Omega} B\underline{\aut}((BK)^{\wedge}_p) \cong B\underline{\aut}(\locl_K;S_K),
$$
where the middle arrow is the classifying map of the group extension. Note that the outer action associated to $\locl_Q \to B\underline{\aut}(\locl_K;S_K)$ is the composition $\alpha' \colon \locl_Q \Right1{\iota} Q \Right1{\alpha} \Out(\locl_K;S_K)$, and thus it is rigid and admissible by hypothesis. Thus, this determines a good extension
$$
\locl_K \Right2{} \locl_H \Right2{} \locl_Q,
$$
where $|\locl_H|^{\wedge}_p \simeq (BH)^{\wedge}_p$.

Let $\alpha'$ and $(\locl_Q, \Delta_Q, S_Q)$ be as above. In general, it is not clear whether a map $\locl_Q \to B\underline{\aut}(\locl_K)$ (which classifies a good extension of localities) determines a group extension $K \to H \to Q$, since it depends on the following lifting/extension problem
$$
\xymatrix@C=3cm{
BQ \ar@{.>}[r] & B\underline{\aut}(BK) \ar[d]^{\Omega} \\
\locl_Q \ar[u]^{\iota} \ar[r] & B\underline{\aut}(\locl_K;S_K)
}
$$

\end{expl}

Next, we relate certain fibrations involving $p$-local finite groups to isotypical extensions of localities. The following is a generalization of \cite[Proposition 7.1]{BLO6}.

\begin{prop}\label{fibrationfb}

Let $\g = \ploc$ be a $p$-local finite group and let $(\locl, \Delta, S)$ be its associated proper locality. Then, for each saturated locality $(\overline{\locl}, \overline{\Delta}, \overline{S})$ there is a bijection from the set of equivalence classes of fibre bundles over $\overline{\locl}$ with fibre $\locl$ and structure group $N\autcat(\locl;S)$ to the set of equivalence classes of fibrations over $\overline{\locl}$ with fibre homotopy equivalent to $B\g$: a bijection which sends the class of a fibre bundle to the equivalence class of its fibrewise $p$-competion.

\end{prop}

\begin{proof}

By \cite[Theorem IV.5.6]{BGM}, there is a bijection between the set of equivalence classes of fibrations over $|\overline{\locl}|$ with fibre $B\g$ and the set $[|\overline{\locl}|, B\underline{\aut}(B\g)]_{\ast}$ of homotopy classes of pointed maps. Similarly, there is a bijection between the set of equivalence classes of $|\autcat(\locl;S)|$-bundles over $\overline{\locl}$ with fibre $\locl$ and the set $[\overline{\locl}, B\autcat(\locl;S)]_{\ast}$ of homotopy classes of pointed maps.

By Corollary \ref{autloc2} the natural map
$$
\Lambda \colon |\autcat(\locl;S)| \Right4{} \underline{\aut}(\locl) \Right4{} \underline{\aut}(B\g),
$$
where the leftmost arrow is an inclusion of simplicial sets (see Lemma \ref{autloc0}) and the rightmost arrow is induced by $p$-completion, is a homotopy equivalence. It follows that the map $\Phi \colon [\overline{\locl}, B\autcat(\locl;S)]_{\ast} \Right6{B\Lambda \circ -} [|\overline{\locl}|, B\underline{\aut}(B\g)]_{\ast}$ that sends the class of an $|\autcat(\locl;S)|$-bundle to the class of its fibrewise $p$-completion is a bijection.
\end{proof}

Part of the proof for the following result was communicated by R. Levi.

\begin{thm}\label{app2}

Let $F \to X \to B$ be a fibration where both $B$ and $F$ are homotopy equivalent to classifying spaces of $p$-local finite groups. Then, $X$ is homotopy equivalent to the classifying space of a $p$-local finite group. Moreover, there exist proper localities $(\locl_F, \Delta_F, S_F)$ and $(\locl_B, \Delta_B, S_B)$, and a commutative diagram of fibre bundles
$$
\xymatrix{
F \ar[r] & X \ar[r] & B \\
|\locl_F| \ar[u] \ar[r] & |\locl| \ar[r] \ar[u] & |\locl_B| \ar[u]
}
$$
where the bottom row is (the realization of) a good extension and all the vertical arrows are equivalences after $p$-completion.

\end{thm}

\begin{proof}

Let $(\locl_F, \Delta_F, S_F)$ be the proper locality associated to the $p$-local finite group for $F$, so $|\locl_F|^{\wedge}_p \simeq F$. For each linking system $\LL_B$ with $|\LL_B|^{\wedge}_p \simeq B$ (we assume that $\Ob(\LL_B)$ contains all the centric radicals), let $(\locl_B, \Delta_B, S_B)$ be the proper locality associated to $\LL_B$. By Theorem \ref{equivnerves} there is an equivalence $|\locl_B|^{\wedge}_p \simeq B$, and thus there is a commutative diagram
\begin{equation}\label{diagfibr}
\xymatrix{
F \ar[r] \ar@{=}[d] & X \ar[r] & B \\
F \ar[r] & X_0 \ar[u] \ar[r] & |\locl_B| \ar[u] \\
|\locl_F| \ar[r] \ar[u] & |\locl| \ar[r] \ar[u] & |\locl_B| \ar@{=}[u] \\
}
\end{equation}
where each row is a fibre bundle. The middle row is obtained from the top row by pulling back along the completion map $|\locl_B| \to B$, and the bottom row is obtained form the middle row by Proposition \ref{fibrationfb}. In particular, the classifying map for $\locl_F \to \locl \to \locl_B$ (and thus the associated twisting function) is given by the composition
$$
\alpha \colon |\locl_B| \Right2{} B \Right2{} B\underline{\aut}(F) \cong B\underline{\aut}(\locl_F;S_F).
$$
By comparing the Serre spectral sequences associated to each of the fibre bundles in (\ref{diagfibr}) it follows that $\locl$ is homotopy equivalent to $X$ after $p$-completion. However, this still does not imply that $X$ is the classifying space of a $p$-local finite group. We consider three cases: when $B$ is $2$-connected, when $B$ is $1$-connected, and the general case. Each case is done in a separate step.

\textbf{Step 1.} Suppose first that $B$ is $2$-connected. By \cite[Theorem 8.1]{BLO2}, we have $\pi_i(B\underline{\aut}(F)) = 0$ for all $i \geq 3$, and thus the map $B \to B\underline{\aut}(F)$ is nulhomotopic. As a consequence, we have
$$
X \simeq F \times B \qquad \mbox{and} \qquad \locl \cong \locl_F \times \locl_B.
$$
Clearly, the extension $\locl_F \to \locl \to \locl_B$ is good in the sense of \ref{defigood}, and $X$ is equivalent to the classifying space of a $p$-local finite group. Note also that in this case the choice of the locality $(\locl_B, \Delta_B, S_B)$ is irrelevant, as long as $\Delta_B$ contains all the centric radicals and $|\locl_B|^{\wedge}_p \simeq B$.

\textbf{Step 2.} Suppose now that $B$ is $1$-connected. Set $A = H_2(B, \Z^{\wedge}_p)$, which is a finite abelian $p$-group, and consider the commutative diagram of fibration sequences
\begin{equation}\label{diag34}
\xymatrix{
 & K(A, 1) \ar@{=}[r] \ar[d] & K(A,1) \ar[d] \\
F \ar@{=}[d] \ar[r] & Y \ar[r] \ar[d] & D \ar[d] \\
F \ar[r] & X \ar[r] \ar[d] & B \ar[d] \\
 & K(A,2) \ar@{=}[r] & K(A,2)
}
\end{equation}
From now on we abbreviate $K(A, 1) \simeq BA$. By construction, $D$ is the $2$-connected cover of $B$. To see that $D$ is also homotopy equivalent to the classifying space of a $p$-local finite group, consider the fibration $BA \to D \to B$. By definition, this is a central extension in the sense of \cite{BCGLO2}, and thus $D$ is homotopy equivalent the classifying space of a $p$-local finite group by \cite[Theorem E]{BCGLO2}. Step 1 applies to show that $Y \simeq F \times D$ is homotopy equivalent to the classifying space of a $p$-local finite group. Finally, we see that $X$ is equivalent to the classifying space of a $p$-local finite group since it corresponds to the (central) quotient of the $p$-local finite group associated to $Y$ by the subgroup $A$.

For each space $Z$ in the two middle rows of the above diagram, let $S_Z$ denote the Sylow $p$-subgroup of the corresponding $p$-local finite group. These $p$-groups can be chosen so that there is a commutative diagram of group extensions
$$
\xymatrix{
 & A \ar@{=}[r] \ar[d] & A \ar[d] \\
S_F \ar[r] \ar@{=}[d] & S_Y \ar[r] \ar[d] & S_D \ar[d] \\
S_F \ar[r] & S_X \ar[r] & S_B
}
$$
Let $(\locl_D, \Delta_D, S_D)$ be the proper locality associated to $D$, where $\Delta_D$ is the collection of quasicentric subgroups that contain the subgroup $A \leq S_D$. Let also $\Delta_B$ be the collection of subgroups of $S_B$ of the form $PA/A$ for $P \in \Delta_D$, and let $(\locl_B, \Delta_B, S_B)$ be the proper locality associated to $B$ with object set $\Delta_B$. Notice that $\Delta_B$ is the collection of all quasicentric subgroups by \cite[Lemma 6.4]{BCGLO2}.

By first pulling back along the completion map $|\locl_B| \to B$ and then applying Lemma \ref{fibrationfb} in (\ref{diag34}), we produce a commutative diagram of isotypical extensions
$$
\xymatrix{
 & BA \ar@{=}[r] \ar[d] & BA \ar[d] \\
\locl_F \ar@{=}[d] \ar[r] & \locl_Y \ar[r] \ar[d] & \locl_D \ar[d] \\
\locl_F \ar[r] & \locl \ar[r] & \locl_B
}
$$
By Corollary \ref{propext6-1}, the rightmost column is a good extension in the sense of \ref{defigood}, since the associated outer action is trivial; and so is the middle row, by Step 1. Note also that there is a commutative diagram of morphisms of partial groups
$$
\xymatrix{
\locl_D \ar[d] \ar[rd]^{\varepsilon_D} & \\
\locl_B \ar[r]_{\varepsilon} & \Out(\locl_F;S_F)
}
$$
where $\varepsilon \colon \locl_B \to \Out(\locl_F;S_F)$ and $\varepsilon_D \colon \locl_D \to \Out(\locl_F;S_F)$ denote the corresponding outer actions. Since $D$ is $2$-connected, it follows by Step 1 that $\varepsilon_D$ is the trivial map, and thus $\varepsilon$ is also the trivial map (since $\locl_D \to \locl_B$ is surjective). Thus, the extension $\locl_F \to \locl \to \locl_B$ is good by Corollary \ref{propext6-1} (note that $\Delta_B$ contains all the centrics).

\textbf{Step 3.} The general case. There is a commutative diagram of fibrations
$$
\xymatrix{
F \ar@{=}[d] \ar[r] & W \ar[r] \ar[d] & C \ar[d] \\
F \ar[r] & X \ar[r] \ar[d] & B \ar[d] \\
 & B\pi_1(B) \ar@{=}[r] & B \pi_1(B)
}
$$
where $\pi_1(B)$ is a finite $p$-group by \cite[Proposition 1.12]{BLO2}, since $B$ is the classifying space of a $p$-local finite group. Note that $W \to X$ and $C \to B$ are (regular) covering spaces, and thus both $C$ and $X$ are equivalent to classifying spaces of $p$-local finite groups by \cite[Theorem A]{BLO6}.

Since all the spaces involved are homotopy equivalent to classifying spaces of $p$-local finite groups, the above diagram induces a commutative diagram of group extensions
$$
\xymatrix{
S_F \ar[r] \ar@{=}[d] & S_W \ar[r]^{\tau_W} \ar[d] & S_C \ar[d] \\
S_F \ar[r] & S_X \ar[r]_{\tau} & S_B
}
$$
where $S_Z$ denotes the Sylow $p$-subgroup for the $p$-local finite group associated to the space $Z$. Even if the $p$-local finite group associated to $Z$ is only unique up to isomorphism, we can still choose Sylow $p$-subgroups as above to make the diagram commutative, and we assume these choices fixed. Consider the following.
\begin{enumerate}[(1)]

\item $\Delta_C$ is the collection of all quasicentric subgroups of $S_C$ with respect to the $p$-local finite group associated to $C$.

\item $\Delta_B = \{P \leq S_B \, | \, P \cap S_C \in \Delta_C\}$.

\item $\Delta_F$ is the collection of all centric subgroups of $S_F$ with respect to the $p$-local finite group associated to $F$.

\item $\Delta_W = \{P \leq S_W \, | \, P \cap S_F \in \Delta_F \mbox{ and } \tau_W(P) \in \Delta_C\}$.

\item $\Delta_X = \{P \leq S_X \, | \, P \cap S_W \in \Delta_W\}$.

\end{enumerate}
We claim that each of the above collections of subgroups contains all the centric radical subgroups with respect to the corresponding $p$-local finite group.
\begin{enumerate}[(a)]

\item Let $(\locl_C, \Delta_C, S_C)$ be the proper locality associated to $C$. By Lemma \ref{fibrationfb} there is an isotypical extension $\locl_C \Right2{} \locl_B \Right2{} B\pi_1(B)$, which is rigid and admissible by Remark \ref{perfect1} (iii) (and thus it is good by Theorem \ref{propext6}). In particular $\Delta_B$ contains all the centric radical subgroups of $S_B$ with respect to the $p$-local finite group associated to $B$.

\item The fibration $F \to W \to C$ induces an isotypical extension $\locl_F \Right2{} \locl_W \Right2{} \locl_C$, which is good by Step 2. As a consequence, $\Delta_W$ contains all the centric radical subgroups of $S_W$ with respect to the $p$-local finite group associated to $W$.

\item Let $(\4{\locl}_W, \Delta_W, S_W)$ be the proper locality associated to $W$, with $\Delta_W$ as above. By Lemma \ref{fibrationfb} there is an isotypical extension $\4{\locl}_W \Right2{} \locl_X \Right2{} B\pi_1(B)$, which again is rigid and admissible by Remark \ref{perfect1} (iii) (and thus good by Theorem \ref{propext6}). In particular this implies that the collection $\Delta_X$ contains all the centric radical subgroups of $S_X$ with respect to the $p$-local finite group associated to $X$.

\end{enumerate}

Finally, consider the fibration $F \to X \to B$. By first pulling back along the $p$-completion map $|\locl_B| \to B$, and then applying Lemma \ref{fibrationfb}, we obtain an extension
$$
\locl_F \Right2{} \locl \Right2{\tau} \locl_B.
$$
Let $S_X$ be as above, and set
$$
\Delta = \{P \leq S_X \, | \, P \cap S_F \in \Delta_F \mbox{ and } \tau(P) \in \Delta_B\}.
$$
We claim that the above extension is good in the sense of \ref{defigood}. To show this, we check that this extension is rigid, and that $\Delta$ contains all the centric radical subgroups of $S_X$ with respect to the $p$-local finite group associated to $X$.

By Step 2, the outer action $\varepsilon_C \colon \locl_C \Right2{} \Out(\locl_F;S_F)$ is trivial. Furthermore, it factors through the outer action $\varepsilon \colon \locl_B \to \Out(\locl_F;S_F)$ associated to the extension $\locl_F \to \locl \to \locl_B$ above, which  implies that
$$
S_C \leq \Ker(S_B \Right2{} \locl_B \Right2{} \Out(\locl_F;S_F)).
$$
On the other hand, $S_C \in \Delta_B$ by definition, and thus the extension $\locl_F \to \locl \to \locl_B$ is rigid, since $\Delta_B$ is closed by overgroups.

To show that $\Delta$ contains all the centric radical subgroups of $S_X$ with respect to the $p$-local finite group associated to $X$, we prove that $\Delta = \Delta_X$ above. Indeed, let $P \in \Delta_X$. Then,
\begin{itemize}

\item $P \cap S_W \in \Delta_W$, which implies that that $P \cap S_F = (P \cap S_W) \cap S_F \in \Delta_F$; and

\item $\tau(P) \cap S_C = \tau_W(P \cap S_W) \in \Delta_C$, and thus $P \in \Delta$.

\end{itemize}
Conversely, let $P \in \Delta$. Then,
\begin{itemize}

\item $(P \cap S_W) \cap S_F = (P \cap S_F) \cap S_W = P \cap S_F \in \Delta_F$; and

\item $\tau_W(P \cap S_W) = \tau(P) \cap S_C \in \Delta_C$, since $\tau(P) \in \Delta_B$. Thus $P \in \Delta_X$.

\end{itemize}
We already showed above that $\Delta_X$ contains all the centric radical subgroups, and thus the extension $\locl_F \to \locl \to \locl_B$ is good and satisfies the properties required in the statement.
\end{proof}


\appendix

\section{Localities and transporter systems}\label{AppA}

In this appendix we discuss the relation between a locality and its associated transporter category. Let us start by recalling some constructions and results from \cite{Chermak}.

Given an objective partial group $(\locl, \Delta)$, we can form the associated transporter category $\TT = \TT_{\Delta}(\locl)$, as described in \ref{asscat}, where $\Ob(\TT) = \Delta$ and $\Mor_{\TT}(P,Q) = N_{\locl}(P,Q)$. Composition in $\calt$ is given by the product $\Pi$ in $\locl$: for $u \in \Mor_{\calt}(X,Y)$ and $v \in \Mor_{\calt}(Y,Z)$,
$$
v \circ u = \Pi[v|u] \in \Mor_{\calt}(X,Z).
$$
In view of the composition rule, we will represent morphisms in the category $\TT$ by $X \Left2{u} Y$. This way, the above composition becomes
$$
\big(X \Left2{u} Y \Left2{v} Z\big) = \big(X \Left4{u \circ v} Z\big).
$$
Note that every morphism in $\calt$ factors uniquely as an isomorphism followed by an inclusion morphism.

The following example sketches the construction of a locality out of a transporter system (in the sense of \cite{OV}). It also establishes a simplicial map from the nerve of this transporter system to the resulting locality.

\begin{expl}\label{nervimodrel}

Let $\calc$ be a small category and $\cald\subset\calc$ a subcategory. Then we can define and equivalence relation on $N(\calc)$ by declaring two $n$-simplices $c_0\Left1{\varphi_1} c_1\Left1{\varphi_2}\dots \Left1{\varphi_n}c_n$ and $d_0\Left1{\psi_1}  d_1\Left1{\psi_2}\dots \Left1{\psi_n}d_n$ related if there is a commutative diagram
 $$
\xymatrix{
c_0 \ar[d]^{\alpha_0} & c_1 \ar[l]_{\varphi_1} \ar[d]^{\alpha_1} & \ldots \ar[l]_{\varphi_2} & c_n \ar[l]_{\varphi_n} \ar[d]^{\alpha_n} \\
d_0 & d_1 \ar[l]^{\psi_1} & \ldots \ar[l]^{\psi_2} & d_n \ar[l]^{\psi_n}
}
$$
where the vertical maps are morphisms of $\calc$ or of $\calc\op$, i.e. for each $i$, either $\alpha_i\in \Mor_\cald(c_i,d_i)$ or $\alpha_i\in \Mor_\cald(d_i,c_i)$.

This relation is not necessarily symmetric or transitive, but it determines an equivalence relation where two simplices are related if there is a finite zig-zap of diagrams like the one above connecting both simplicies (the empty zig-zag makes the relation reflexive as well). This relation is compatible with face maps and degeneracies, so the set of equivalence classes is again a simplicial, denoted $N(\calc)/\cald$ set and the projection $$N(\calc)\Right2{proj}N(\calc)/\cald$$ a simplicial map.

\end{expl}

In particular, if $(\TT, \varepsilon, \rho)$ is a transporter system in the sense of \cite{OV}, we can set $\pp$ to be the poset of objects of $\TT$, with
$$
\Mor_{\TT}(P,Q) = \left\{
\begin{array}{ll}
\{\varepsilon_{P,Q}(1)\} & \mbox{if } P \leq Q \\
\emptyset & \mbox{otherwise,}
\end{array}
\right.
$$
and the resulting simplicial set $\locl = N(\TT)/\pp$ is a locality (details are left to the reader). Naturally, there is a projection map $\tau \colon N(\TT) \to \locl$, and we say that a simplex $\omega \in \locl$ is \emph{represented} by a simplex $\7 \varphi = (P_0 \Left1{\varphi_1} \ldots \Left1{\varphi_n} P_n) \in N(\TT)$ if $\tau(\7 \varphi) = \omega$. The goal of this section is to show that the above map $N(\TT) \to \locl$ is a weak equivalence of simplicial sets. First we need two technical lemmas.

\begin{lmm}\label{maxrep}

Let $(S, \calt)$ be a transporter system, let $(\locl, \Delta, S)$ be the associated locality, and let $\tau \colon N(\calt) \to \locl$ be the projection map described above. Then, each $\omega \in \locl$ has a representative $\7 \varphi$ which is maximal with respect to restrictions (and thus is maximal among all representatives of $\omega$). Equivalently, for each sequence of composable isomorphisms $\7 \varphi = (P_0 \Left2{\varphi_1} P_1 \Left2{\varphi_2} \ldots \Left2{\varphi_n} P_n)$, there is a sequence $\7 \psi = (R_0 \Left2{\psi_1} R_1 \Left2{\psi_2} \ldots \Left2{\psi_n} R_n)$ of composable isomorphisms which is maximal with respect to restrictions: $\7 \varphi$ is the restriction of $\7 \psi$, and if $\7 \varphi$ is the restriction of some other sequence $\7 \phi$, then $\7 \phi$ is a restriction of $\7 \psi$.

\end{lmm}

\begin{proof}

The argument is by induction on $n$, the length of the sequence $\7 \varphi$. If $n = 0$, there is nothing to prove since $S$ contains every other object of $\TT$, and if $n = 1$ then the statement corresponds to \cite[A.8]{Chermak}.

Assume $n \geq 2$. By induction hypothesis, the sequence $(P_1 \Left2{\varphi_2} \ldots \Left2{\varphi_n} P_n)$ has a maximal, unique extension, $(X_1 \Left2{\3{\varphi_2}} \ldots \Left2{\3{\varphi_n}} X_n)$,
$$
\xymatrix{
 & X_1 & X_2 \ar[l]_{\3{\varphi_2}} & \ldots \ar[l]_{\3{\varphi_3}} & X_n \ar[l]_{\3{\varphi_n}} \\
P_0 & P_1 \ar[l]^{\varphi_1} \ar[u]^{\iota} & P_2 \ar[l]^{\varphi_2} \ar[u]^{\iota} & \ldots \ar[l]^{\varphi_3} & P_n \ar[l]^{\varphi_n} \ar[u]_{\iota}
}
$$
Let also $Y_0 \Left2{\3{\varphi_1}} Y_1$ be the unique maximal extension of $P_0 \Left2{\varphi_1} P_1$. Both $Y_0 \Left2{\3{\varphi_1}} Y_1$ and $(X_1 \Left2{\3{\varphi_2}} \ldots \Left2{\3{\varphi_n}} X_n)$ can be restricted to $R_1 = X_1 \cap Y_1$, to obtain a sequence of composable isomorphisms $\7 \psi = (R_0 \Left2{\psi_1} R_1 \Left2{\psi_2} \ldots \Left2{\psi_n} R_n)$ whose restriction to $P_0$ is the original sequence $\7 \varphi$.

Maximality and uniqueness of $\7 \psi$ follow by construction. Indeed, if $\7 \varphi$ is the restriction of a sequence $\7 \phi = (Q_0 \Left2{\phi_1} Q_1 \Left2{\phi_2} \ldots \Left2{\phi_n} Q_n)$, then $Q_0 \Left2{\phi_1} Q_1$ is a restriction of $Y_0 \Left2{\3{\varphi_1}} Y_1$ and $(Q_1 \Left2{\phi_2} \ldots \Left2{\phi_n} Q_n)$ is a restriction of $(X_1 \Left2{\3{\varphi_2}} \ldots \Left2{\3{\varphi_n}} X_n)$ by the induction hypothesis, and thus $Q_1 \leq X_1 \cap Y_1$. The claim follows easily.
\end{proof}

In \cite[Section III.5]{GoJa}, the authors introduce the concept of \emph{extra degeneracy} for a simplicial set, and they show that a simplicial set with an extra degeneracy is always contractible. We review here this definition, and add a variation of it.

\begin{defi}\label{altextdeg1}

Let $X$ be a simplicial set. Define $X_{-1} = \pi_0(X)$, and set $d_0 \colon X_0 \to X_{-1}$ as the canonical map. An \emph{extra degeneracy} on $X$ is a collection of maps $\{s_{-1} \colon X_n \to X_{n+1}\}_{n \geq -1}$, satisfying the following conditions for all $n \geq -1$:
\begin{enumerate}[(i)]

\item $d_0 \circ s_{-1}$ is the identity on $X_n$; and

\item for all $0 \leq i, j, \leq n$, there are identities
$$
d_{i+1} \circ s_{-1} = s_{-1} \circ d_i \qquad \mbox{and} \qquad s_{j+1} \circ s_{-1} = s_{-1} \circ s_j.
$$

\end{enumerate}
Similarly, an \emph{alternative extra degeneracy} on $X$ is a collection of maps $\{\3{s}_{-1} \colon X_n \to X_{n+1}\}_{n \geq -1}$, satisfying the following conditions for all $n \geq -1$:
\begin{enumerate}[(a)]

\item $d_{n+1} \circ \3{s}_{-1}$ is the identity on $X_n$; and

\item for all $0 \leq i, j \leq n$, there are identities
$$
d_{i} \circ \3{s}_{-1} = \3{s}_{-1} \circ d_i \qquad \mbox{and} \qquad s_{j} \circ \3{s}_{-1} = \3{s}_{-1} \circ s_j.
$$

\end{enumerate}

\end{defi}

\begin{lmm}\label{altextdeg2}

Let $X$ be a simplicial set. If $X$ admits either an extra degeneracy or an alternative extra degeneracy, then the canonical map $X \to K(\pi_0(X),0)$ is a homotopy equivalence.

\end{lmm}

\begin{proof}

The case where $X$ admits an extra degeneracy corresponds to \cite[Lemma III.5.1]{GoJa}. Suppose that $X$ admits an alternative extra degeneracy. The proof for this case follows similar arguments as that in \cite{GoJa}. For simplicity we may assume that $X$ is connected. Consider the cone of $X$,
$$
CX \defin \colim_{\Delta^n \to X} \Delta^{n+1},
$$
where the colimit is taken over the simplex category of $X$. The inclusions $\Delta^n \to \Delta^{n+1}$ given by $i \mapsto i$ define an inclusion map $j \colon X \to CX$, and \cite[Section III.5]{GoJa} gives necessary and sufficient conditions for a map $f \colon X \to Y$ to factor through $j \colon X \to CX$. Similarly, the inclusions $\Delta^n \to \Delta^{n+1}$ given by $i \mapsto i+1$ induce an inclusion map $k \colon X \to CX$, and we can deduce a corresponding set of necessary and sufficient conditions for a map $f \colon X \to Y$ to factor through $k \colon X \to CX$. As happened in \cite[Lemma III.5.1]{GoJa}, the alternative extra degeneracy provides the necessary data to construct a factorization of $\Id \colon X \to X$ through $k \colon X \to CX$. Thus, since $CX$ is contractible, the statement follows.
\end{proof}

\begin{thm}\label{equivnerves}

Let $\calt$ be a transporter system defined over a finite $p$-group $S$, and let $\locl$ be the nerve of the associated locality. The projection map $\pi \colon N(\calt) \to \locl$ is a weak homotopy equivalence of simplicial sets.

\end{thm}

\begin{proof}

Following the argument in \cite[Proof of Theorem. A]{Quillen}, we define a bisimplicial set $\Gamma$ that will allow us to compare the homotopy types of $N(\calt)$ and $\locl$.

\textbf{Step 1.} Define $\Gamma$ as the set with $(q,r)$-simplices the symbols $\sigma = (\omega, [x], \7 \varphi)$, where
\begin{enumerate}[(1)]

\item $\omega = [z_1|\ldots|z_q]$ is a $q$-simplex in $\locl$;

\item $[x] \in \locl_1$; and

\item $\7 \varphi = (P_0 \Left1{\varphi_1} P_1 \Left1{\varphi_2} \ldots \Left1{\varphi_r} P_r)$ is a $r$-simplex in $N_q(\calt)$;

\end{enumerate}
subject to the condition that $[z_1|\ldots|z_q|x|\tau(\varphi_1)|\ldots|\tau(\varphi_r)] \in \locl$. Let $\sigma = (\omega, [x], \7 \varphi) \in \Gamma_{q,r}$, with $\omega = [z_1|\ldots|z_q]$ and $\7 \varphi = (P_0 \Left1{\varphi_1} \ldots \Left1{\varphi_r} P_r)$. The horizontal face and degeneracy operators are given by the following formulas.
$$
d_i^h(\sigma) = \left\{
\begin{array}{ll}
(d_i(\omega), [x], \7 \varphi), & 0 \leq i \leq q-1\\
(d_q(\omega), [z_q \cdot x], \7 \varphi), & i = q\\
\end{array}
\right.
\qquad
s_j^h(\sigma) = (s_j(\omega), [x], \7 \varphi), \,\, 0 \leq j \leq q.
$$
Similarly, the vertical face and degeneracy operators are given by the following formulas.
$$
d_i^v(\sigma) = \left\{
\begin{array}{ll}
(\omega, [x \cdot \tau(\varphi_1)], d_0(\7 \varphi)), & i = 0\\
(\omega, [x], d_i(\7 \varphi)), & 1 \leq i \leq r\\
\end{array}
\right.
\qquad
s_j^v(\sigma) = (\omega, [x], s_j(\7 \varphi)), \,\, 0 \leq j \leq r.
$$
Note that these operators are well defined by definition of $\Gamma$. The following holds.
\begin{enumerate}[(a)]

\item The realization of a bisimplicial set is the diagonal simplicial set. There are maps 
\begin{equation}\label{diag1}
\locl  \Left4{\pr_1} \diag(\Gamma) \Right4{\pr_2} N(\calt)
\end{equation}
induced by obvious projections.

\item Equivalent realizations are obtained by first realizing the horizontal simplicial sets, thus obtaining (vertically) a simplicial space, and then realizing this simplicial space, or the other way around, that is, realizing first the vertical simplicial sets and then the remaining horizontal simplicial space.

\end{enumerate}

\textbf{Step 2.}  The simplicial set in \emph{rows}. Fix a row $r$, and consider the simplicial subset $\Gamma_{\bullet, r} \subseteq \Gamma$. This can be partitioned on one simplicial set $\Gamma_{\bullet, \7 \varphi}$ for each $\7 \varphi \in N_r(\calt)$, where
$$
\Gamma_{q,\7 \varphi} =   \{(\omega, [x], \7 \varphi)\in \Gamma_{q,r} \}.
$$
We show that $\Gamma_{\bullet, \7 \varphi}$ is contractible by showing that it is connected and admits an alternative extra degeneracy.

Fix $\7 \varphi = (P_0 \Left1{\varphi_1} \ldots \Left1{\varphi_r} P_r) \in N_r(\calt)$, and let $v \in \locl_0$ be its unique vertex. Let also $w_0 = (v, [1], \7 \varphi), w = (v, [x], \7 \varphi) \in \Gamma_{0, \7 \varphi}$. Then, the $1$-simplex $\sigma = ([x^{-1}], [x], \7 \varphi)$ satisfies
$$
d_0^h(\sigma) = w \qquad \mbox{and} \qquad d_1^h(\sigma) = w_0.
$$
In particular, this shows that $\Gamma_{\bullet, \7 \varphi}$ is connected. Set $\Gamma_{-1, \7 \varphi} = \{\ast\}$, and define $\3{s}_{-1}^h$ as follows. In dimension $q = -1$, set $\3{s}_{-1}^h(\ast) = w_0$, where $w_0$ is the vertex specified above. In dimensions $q \geq 0$, set
$$
\3{s}_{-1}^h(\omega, [x], \7 \varphi) = ([z_1|\ldots|z_q|x], [1], \7 \varphi),
$$
where $\omega = [z_1|\ldots|z_q]$. That $\3{s}_{-1}^h$ satisfies the conditions to be an alternative extra degeneracy is an easy exercise left to the reader. By Lemma \ref{altextdeg2}, it follows that $\Gamma_{\bullet, \7 \varphi}$ is contractible. Hence, the realization of the $r$-th row is
$$
\coprod_{\7 \varphi \in N_r(\calt)} |\Gamma_{\bullet, \7 \varphi}| \simeq \coprod_{\7 \varphi \in N_r(\calt)}\{\ast\}.
$$
Thus, the  \emph{vertical} simplicial space is equivalent to the simplicial set $N(\calt)$, and the projection $\pr_2\colon \diag(\Gamma) \to N(\calt)$ is a weak homotopy equivalence.

\textbf{Step 3.} The simplicial sets in \emph{columns}. Fix a column $q$, and consider $\Gamma_{q, \bullet} \subseteq \Gamma$. This a union of simplicial sets $\Gamma_{\omega, \bullet}$, with $\omega \in \locl_q$. Following an argument similar to Step 2, we show that $\Gamma_{\omega, \bullet}$ is contractible by proving that it is connected and that it admits an extra degeneracy.

Fix $\omega = [z_1|\ldots|z_q] \in \locl_q$, and let $\7 \psi = (R_0 \Left1{\psi_1} \ldots \Left1{\psi_q} R_q)$ be the maximal representative of $\omega$, as described in Lemma \ref{maxrep}. To check that $\Gamma_{\omega, \bullet}$ is connected, let $u = (\omega, [x], (P_0)) \in \Gamma_{\omega, 0}$ be a vertex. By definition of $\Gamma$, we have $[z_1|\ldots|z_q|x] \in \locl$, which means that there is some sequence
$$
K_0 \Left2{\gamma_1} K_1 \Left2{\gamma_2} \ldots \Left2{\gamma_q} K_q \Left2{\gamma} P_0
$$
that represents $[z_1|\ldots|z_q|x]$. Without loss of generality we may assume that $\gamma$ is an isomorphism in $\calt$. Notice that $K_0 \Left1{\gamma_1} \ldots \Left1{\gamma_q} K_q$ is a representative of $\omega$, and thus by Lemma \ref{maxrep} it is the restriction of $\7 \psi$. In particular, $K_q \leq R_q$. Consider the vertices $u_0 = (\omega, [1], (R_q)), u_1 = (\omega, [1], (K_q)) \in \Gamma_{\omega, 0}$, and the $1$-simplices $\sigma_1 = (\omega, [x], (P_0 \Left1{\gamma^{-1}} K_q))$ and $\sigma_2 = (\omega, [1], (R_q \Left1{\incl} K_q))$. Then,
$$
d_0^v(\sigma_1) = u_1 \qquad d_1^v(\sigma_1) = u \qquad d_0^v(\sigma_2) = u_1 \qquad d_1^v(\sigma_2) = u_0,
$$
and this shows that $\Gamma_{\omega, \bullet}$ is connected.

Set $\Gamma_{\omega, -1} = \{\ast\}$ and define an extra degeneracy on $\Gamma_{\omega, \bullet}$ as follows. In dimension $r = -1$, set $s_{-1}^v(\ast) = u_0$, the vertex specified in the previous paragraph. In dimensions $r \geq 0$, define
$$
s_{-1}^v(\omega, [x], (P_0 \Left1{\varphi_1} \ldots \Left1{\varphi_r} P_r)) = (\omega, [1], (R_q \Left1{\alpha} P_0 \Left1{\varphi_1} \ldots \Left1{\varphi_r} P_r)),
$$
where $R_q \Left1{\alpha} P_0$ is the (unique) morphism from $P_0$ to $R_q$ such that $\tau(\alpha) = [x]$. Checking that $s_{-1}^h$ is an extra degeneracy is again a routine exercise which is left to the reader. Thus $\Gamma_{\omega, \bullet}$ is contractible, and as a consequence the map $\pr_1\colon \diag(\Gamma) \Right1{} \locl$ is a weak homotopy equivalence.

\textbf{Step 4.}  The same construction with the identity map $\Id \colon N(\calt) \Right1{} N(\calt)$ gives a new bisimplicial set $\widetilde{\Gamma}$ and a commutative diagram of simplicial sets
$$
\xymatrix{
N(\calt) \ar[d]_{\tau} & \diag(\widetilde{\Gamma}) \ar[d]^{\tau_\sharp} \ar[l]_{\pr_1} \ar[r]^{\pr_2} &            N(\calt) \ar@{=}[d]  \\
\locl & \diag(\Gamma) \ar[l]_{\pr_1} \ar[r]^{\pr_2} & N(\calt)
}
$$
where the horizontal arrows are weak homotopy equivalences. As a consequence the map $ \tau \colon N(\calt) \Right1{} \locl$ is also a weak equivalence.
\end{proof}


\bibliographystyle{alpha}
\bibliography{/Users/agondem/Dropbox/MATH/Tex-extras/Main}

\end{document}